\documentclass[leqno,a4paper,11pt]{article}

\usepackage[colorlinks,
linkcolor=blue,
urlcolor=magenta,
citecolor=red]{hyperref}
\usepackage{amsthm}
\usepackage{amsmath}
\usepackage{amssymb}
 \usepackage{soul}
 \usepackage[shortlabels]{enumitem}

\usepackage[usenames,dvipsnames]{xcolor}

\usepackage{mathtools}

\newcommand{\Con}{\ensuremath{\mathcal{C}}}

\newcommand{\U}{\ensuremath{{\mathcal U}}}

\newcommand{\mb}[1]{\ensuremath{\mathbb{#1}}}
\newcommand{\N}{\mb{N}}

\newcommand{\R}{\mb{R}}

\newfont{\bl}{msbm10 scaled \magstep2}

\newcommand{\beq}{\begin{equation}}
\newcommand{\eeq}{\end{equation}}

\newcommand{\dis}[2]{\langle #1 , #2 \rangle}
\newcommand{\notmid}{\mid\kern-0.5em\not\kern0.5em}

\newcommand{\eps}{\varepsilon}

\newenvironment{pr}{\begin{proof}[\textbf{Proof:}] \ }{\end{proof}}
\newtheorem{thm}{Theorem}[section]
\newtheorem{lem}[thm]{Lemma}
\newtheorem{prop}[thm]{Proposition}
\newtheorem{ex}[thm]{Example}
\newtheorem{cor}[thm]{Corollary}
\newtheorem{rem}[thm]{Remark}
\newtheorem{defi}[thm]{Definition}

\newcommand{\diam}{\mathrm{diam}}
\newcommand{\J}{\mathcal{J}}

\newcommand{\LLSn}{Lorentzian length space}
\newcommand{\LpLS}{Lorentzian pre-length space }
\newcommand{\LpLSn}{Lorentzian pre-length space}

\newcommand{\Lie}{\mathcal L}
\newcommand{\Hess}{\mathrm{Hess}}

\newcommand{\mm}{\mathfrak{m}}

\newcommand{\M}{\mathcal{M}}

\newcommand{\V}{\mathcal{V}}
\newcommand{\X}{\mathfrak{X}}
\renewcommand{\dis}{\mathrm{dis}}

\newcommand{\Xt}{(X,\ell)}
\newcommand{\Xtn}{(X_n,\ell_n)}
\newcommand{\Xto}{(X,\ell,o)}

\newcommand{\XtoU}{(X,\ell,o,\U)}
\newcommand{\XtnoU}{(X_n,\ell_n,o_n,\U_n)}
\newcommand{\XtnjoU}{(X_{n_j},\ell_{n_j},o_{n_j},\U_{n_j})}
\newcommand{\dd}{\mathrm{d}}

\newcommand{\Xtb}{(\overline{X},\overline{\ell})}
\newcommand{\LL}{\mathbb{L}^2}

\newcommand{\LGHtop}{\stackrel{\mathrm{LGH}}{\longrightarrow}}
\newcommand{\mLGHtop}{\stackrel{\mathrm{mLGH}}{\longrightarrow}}

\newcommand{\pLGHtop}{\stackrel{\mathrm{pLGH}}{\longrightarrow}}
\newcommand{\pmLGHtop}{\stackrel{\mathrm{pmLGH}}{\longrightarrow}}
\newcommand{\GHto}{\stackrel{\mathrm{GH}}{\longrightarrow}}

\renewcommand{\labelenumi}{(\roman{enumi})}
\renewcommand\theenumi\labelenumi

\usepackage{todonotes}

\usepackage{tcolorbox}   
\definecolor{mycolor}{rgb}{0.122, 0.435, 0.698}

\newtheorem{ithm}{Theorem}[section]

\title{Lorentzian Gromov--Hausdorff convergence and pre-compactness}

\author{Andrea Mondino\thanks{{\tt andrea.mondino@maths.ox.ac.uk} Mathematical Institute, University of Oxford, UK.}\ \ and Clemens S\"amann\thanks{{Faculty of Mathematics, University of Vienna, Oskar-Morgenstern-Platz 1, 1090 Vienna, Austria. \texttt{clemens.saemann@univie.ac.at}}}}

\begin{document}
 \maketitle
 \setcounter{tocdepth}{1}

\begin{abstract}
The goal of this paper is to introduce a notion of convergence à la Gromov–Hausdorff for Lorentzian spaces, based on $\varepsilon$-nets formed by causal diamonds and depending solely on the time-separation function. This provides a geometric framework for convergence that applies both to synthetic Lorentzian spaces (\LpLSn s) and to smooth spacetimes.
Among our main results, we establish a Lorentzian analogue of Gromov’s celebrated precompactness theorem for metric spaces, where controlled covers by balls are replaced with controlled covers by diamonds. This leads to two geometric precompactness results for families of globally hyperbolic $n$-dimensional spacetimes $(M^n,g)$.
The first relies on suitable causality control together with the existence of a Cauchy hypersurface $\Sigma$ satisfying a uniform doubling property. The second assumes:
\begin{itemize}
\item the existence of a compact Cauchy surface $\Sigma$ with bounded second fundamental form and bounded diameter;
\item a lower bound on the ambient Ricci curvature along $\Sigma$;
\item a lower bound on timelike sectional curvature on $M$.
\end{itemize}
This is the first Lorentzian precompactness result obtained under curvature–diameter assumptions.
In the final part of the paper, we present several applications. We show that Chruściel–Grant approximations \cite{CG:12} are special cases of the Lorentzian Gromov–Hausdorff convergence introduced here, prove that timelike sectional curvature bounds are stable under this convergence, introduce timelike blow-up tangents, and discuss connections with the main conjecture of causal set theory.

\vskip 1em

\medskip

\end{abstract}

\newpage
 \tableofcontents

\section{Introduction}
Gromov--Hausdorff convergence and Gromov's pre-compactness theorem\\ \cite[Ch.\ 5, §A]{Gro:99} are fundamental results in Riemannian geometry and metric geometry. For instance, Gromov--Hausdorff limits of Riemannian manifolds, whose sectional curvatures satisfy a uniform lower bound,  are length spaces with curvature bounded below in a synthetic sense via triangle comparison, i.e., \emph{Alexandrov spaces}, see, e.g., \cite{BGP:92, BBI:01}. Similarly, limits of Riemannian manifolds whose Ricci curvature is uniformly bounded below --- so-called \emph{Ricci limit spaces} \cite{CC:97, CC:00a, CC:00b} --- are very rich structures, extremely useful to understand the degeneration of Riemannian metrics. These were a precursor to studying synthetic Ricci curvature lower bounds, e.g., the $\mathsf{CD}(K,N)$-spaces of Lott--Sturm--Villani \cite{Stu:06a, Stu:06b, LV:09}.  Synthetic curvature bounds for non-smooth spaces proved to be extremely fruitful and opened up many venues in geometry, PDE, geometric inequalities, optimal transport and more, see for instance \cite{BH:99, Gro:99, BBI:01, Vil:09}.
\smallskip

An analogous program in Lorentzian geometry was recently initiated  in \cite{KS:18} (after earlier work by Kronheimer--Penrose \cite{KP:67} and Busemann \cite{Bus:67}), where \emph{\LpLSn s} are introduced as analogues of metric spaces, and \emph{\LLSn s} as analogues of length spaces. Moreover, timelike curvature bounds were introduced via triangle comparison with the two-dimensional Lorentzian model spaces (Minkowski-, (anti-)\\ de Sitter spaces) and applications to low regularity spacetime geometry were given. This led to a surge of activity in the field of non-smooth spacetime geometry, whose main developments include establishing a relation between spacetime inextendability and synthetic curvature blow-up \cite{GKS:19}, the study of cones and  related  synthetic singularity theorems \cite{AGKS:23}, the introduction of timelike Ricci curvature bounds and a synthetic Hawking singularity theorem \cite{CM:24a} (after \cite{McC:20, MS:23}), a Lorentzian synthetic splitting theorem \cite{BORS:23}, the development of Lorentzian Hausdorff measures and dimension \cite{McCS:22}, the study of time functions in \LLSn s \cite{BGH:24}, a synthetic formulation of the Einstein vacuum equations \cite{MS:23}, a differential calculus and d'Alembertian comparison \cite{BBCGMORS:24, Bra:24}, characterizations and variable timelike Ricci curvature bounds \cite{Bra:23b, BMcC:23}, timelike Ricci curvature bounds for Finsler spacetimes \cite{BO:24}, globalization and gluing results \cite{BNR:23, BHNR:23, BR:24}. Moreover, hyperbolic angles have been introduced and used to characterize timelike sectional curvature bounds \cite{BS:23, BMS:22}, the causal hierarchy of \LLSn s has been established \cite{ACS:20}, causal boundaries have been studied in this setting \cite{ABS:22, BFH:23} and Lorentzian isoperimetric inequalities have been established \cite{CM:24b}. Finally, a recent direction concerns the study of the \emph{null energy condition (NEC)} from general relativity by synthetic methods \cite{McC:24, Ket:24, CMM:24, CMM:25}. We also refer to the recent reviews \cite{CM:22, Sae:24, McC:25, Bra:25} for overviews on this fast growing research field across Lorentzian geometry, metric geometry, geometric analysis, and optimal transport.
\smallskip

Despite so much progress in synthetic spacetime geometry, one essential ingredient had been missing until recently --- a Lorentzian version of Gromov--Hausdorff convergence. A first attempt was made by Noldus \cite{Nol:04} more than twenty years ago. Recently, Minguzzi and Suhr  \cite{MS:24} introduced the concept of an (abstract) bounded Lorentzian metric space, motivated by the properties of the time-separation function on compact, causally convex subsets of smooth globally hyperbolic spacetimes. Their framework assumes bounded time-separation and was subsequently extended to the unbounded case in a joint work with Bykov \cite{BMS:24}. In particular, the first paper develops a Gromov–Hausdorff–type notion of convergence and establishes an abstract precompactness theorem formulated in terms of $\epsilon$-nets with respect to an associated metric distance, referred to as the \emph{distinction metric}. An open problem that remains is to relate this pre-compactness result to curvature conditions. A similar development has been pursued independently by M\"uller \cite{Mue:22}, by studying Cauchy slabs. In another direction, Sakovich and Sormani recently used the the \emph{null distance}, introduced by Sormani and Vega \cite{SV:16}, to define a \emph{timed Gromov--Hausdorff convergence} \cite{SS:24}. Also, metric Gromov--Hausdorff convergence of \LLSn s with respect to the null distance was studied in \cite{KS:22}, whereas in \cite{AB:22, All:23} convergence of (warped product) spacetimes was studied. In \cite{CP:26} Che, Perales and Sormani establish a precompactness result for the timed Gromov-Hausdorff covergence. Our approach is substantially different, as we use causal diamonds to define a Lorentzian analog of an \emph{$\eps$-net}; this approach yields a purely Lorentzian pre-compactness result which is geometric in nature and does not rely on Gromov's pre-compactness theorem for metric spaces (i.e., in  positive signature). The precise implications between the various notions of convergence will be investigated in future work.

\subsection*{Main results and outline of the paper}
In Section \ref{sec-lpls}, we introduce the basic setting and fix notations. In particular, a \LpLS $\Xt$ is a topological space $X$ with a time-separation function $\ell\colon X\times X \rightarrow \{-\infty\}\cup[0,\infty]$ that satisfies the reverse triangle inequality
\begin{equation*}
    \ell(x,y) + \ell(y,z) \leq \ell(x,z)\,,
\end{equation*}
for all $x,y,z\in X$ and with the convention that the left-hand-side is $-\infty$ if it were undefined otherwise.
Then, in Section \ref{sec-conv}, we introduce the basic ingredient of the Lorentzian Gromov--Hausdorff convergence, namely $\eps$-nets. These are collections of causal diamonds whose timelike diameter (which is the time-separation of its vertices) is smaller than a given $\eps>0$ and whose union covers a given set. More precisely, a collection $\J=(J(p_i,q_i))_{i\in I}$ of causal diamonds is an \emph{$\eps$-net for a subset $A\subset X$} if:
\begin{itemize}
\item $\ell(p_i,q_i)\leq \eps$, for all $i\in I$, and
\item $A\subseteq \bigcup_{i\in I} J(p_i,q_i)$.
\end{itemize}
For a study of causal diamonds in \LLSn s see \cite{BMdOS:24}.

The Lorentzian Gromov--Hausdorff convergence of subsets is defined by using $\eps$-nets and correspondences of their vertices. To extend the convergence from subsets to the whole space, it is convenient to specify a covering. This aspect differs from the  metric case, where it suffices to fix a base point $p\in X$ and then use the exhaustion of metric balls $B_N(p)$, $N\in\N$, centered at $p$ with diverging radii. This will be called \emph{pointed Lorentzian Gromov--Hausdorff convergence} ($\mathrm{pLGH}$ for short) and denoted by $$\XtnoU\pLGHtop\XtoU.$$
In Section \ref{sec-uniq} we show uniqueness of the limit within the class of globally hyperbolic Lorentzian length spaces (actually we prove uniqueness in higher generality, see Theorem \ref{thm-uni-lim}, Proposition \ref{prop-uni-den}, and Remark \ref{rem:UniqGH}).

Afterwards, in Section \ref{sec-quo}, we establish that quotienting out the points that cannot be distinguished by the time-separation function does not affect convergence, cf.\ Theorem \ref{thm-quo-lim}. Section \ref{sec-pre-com} contains the first major result of the paper, stating  that  a uniform bound on the cardinality of the $\eps$-nets yields  pre-compactness. More precisely, we prove:

\setcounter{section}{6}
\setcounter{ithm}{1}
\begin{ithm}[Pre-compactness I]
 Let $\X$ be a class of covered \LpLSn s such that each $\XtoU\in\X$, with covering $\U=(U_k)_{k\in\N}$,  satisfies the following properties.
  \begin{enumerate}
   \item For each fixed $k\in\N$, the timelike diameter of $U_k$ is uniformly bounded; i.e., $\diam^\tau(U_k)\leq T_k$, for a constant $T_k\geq 0$ (independent of $X$).
   \item For all $k\in\N$ and $\eps>0$, there exists a (finite) constant $N=N(k,\eps)>0$ (independent of $X$) such that $U_k$ admits an $\eps$-net $S_\eps^k$ of cardinality at most $N$.
   \item For all $k\in\N$ and $\eps>0$, it holds that $S_\eps^k\subseteq S_\eps^{k+1}$.
     \end{enumerate}
Then any sequence in $\X$ has a converging subsequence; i.e., for any sequence  $\big(\XtnoU\big)_{n\in \N} \subset \X$ there exists a subsequence $(n_j)_j\subset \N$ and a covered \LpLS $\XtoU$ such that 
$$
\XtnjoU\pLGHtop \XtoU \quad  \text{as $j\to \infty$}.
$$ 
\end{ithm}

Next, in Section \ref{sec-compl} we show how to obtain a forward completion of a \LpLSn. This builds on the notion of forward completeness for Lorentzian length spaces introduced in \cite{BBCGMORS:24} and it is inspired by the recent work of Gigli \cite{Gig:25}, who obtained existence and uniqueness of a forward completion in the general setting of partial orders. In  Appendix \ref{app}, we also discuss an independent way of obtain a completion, more tailored to the globally hyperbolic setting.  We refer the reader to  \cite{Gig:25} for more details and results on forward completing partially ordered and Lorentzian spaces. In the present paper, we prove the following result:

\setcounter{section}{7}
\setcounter{ithm}{1}
\begin{ithm}[Existence of forward completion]
 Let $\Xt$ be a \LpLSn. Then $\Xt$ admits a completion $\Xtb$. Moreover,  there exists at most one completion $\Xtb$ (up to isometry), satisfying:
 \begin{enumerate}
     \item The time-separation function $\overline{\tau}$ is continuous;
     \item The causal relation $\overline{\leq}$ is a closed partial order;
     \item For all $\bar x\in \overline{X}\backslash X$,  it holds that $\bar x \in \overline{I^\pm(\bar x)}$.
 \end{enumerate}
 \end{ithm}
 Moreover, taking the forward completion preserves convergence (under mild conditions, which are satisfied in all the applications in this article), see Theorem \ref{thm-lim-compl} for the precise statement.

As a first application of such results,  we will  show that any globally hyperbolic spacetime is the limit of discrete \LpLSn s, see Theorem \ref{thm-dis-lim}. Then, in Section \ref{sec-geo-pre-comp} we will establish a second main result, yielding a geometric pre-compactness theorem for  globally hyperbolic spacetimes, satisfying a uniform doubling property on Cauchy hypersurfaces and a suitable uniform control on the causality conditions.  More precisely, writing $g \preceq g'$ for two Lorentzian metrics $g,g'$ if every $g$-causal vector is also $g'$-causal, we prove:

\setcounter{section}{8}
\setcounter{ithm}{3}
 \begin{ithm}[Geometric pre-compactness]
Let $C\colon(0,\infty)\rightarrow(0,\infty)$ and $N\colon(0,\infty)\rightarrow \N$ be given functions. Consider the following family  $\M_{C,N}$  of smooth globally hyperbolic spacetimes
\begin{align*}
 \M_{C,N} :=\{(\R\times\Sigma, -\beta {\rm d}t^2 + & h_t) :\; \Sigma \text{ is a compact smooth manifold},\\
 &\beta\colon\R\times \Sigma\rightarrow(0,1] \text{ is a  smooth function},\\
 & \forall \eps>0\, \exists \eps\text{-net } S \text{ in } \Sigma \text{ w.r.t.\ } \mathsf{d}^{h_0} \text{ with } |S|\leq N(\eps),\\
 & \forall T>0:\,\rho_{C(T)} \preceq  -\beta {\rm d}t^2 + h_t \text{ on } [-T,T]\times \Sigma\}\,,
\end{align*}
where $\rho_C:= -C^2 {\rm d}t^2 + h_0$. 

Then, 
%
$\M_{C,N}$ is sequentially pre-compact; i.e., for each sequence in $\M_{C,N}$ there is a subsequence that $\mathrm{pLGH}$-converges to a covered, forward complete, \LpLS that satisfies the point distinction property \eqref{eq-pdp}. 
\end{ithm}
 Further, in Subsection \ref{subsec-precomp-curv} we prove a pre-compactness result for smooth globally hyperbolic spacetimes under curvature assumptions: 
\setcounter{ithm}{8}
\begin{ithm} [Pre-compactness under curvature assumptions]
Let $n\in \N_{\geq 1}$ and $K, K_\perp, \Lambda,D>0$.
Consider the family $\mathcal{M}_{n,K,K_\perp,\Lambda,D}$ of globally hyperbolic spacetimes $(M^{n+1},g)$ with
\begin{enumerate}
  \item \textbf{Timelike sectional lower bound:} For all $x\in M$, it holds that ${\rm TSec}^M_x\geq - K_\perp$;
\end{enumerate}
 admitting a compact Cauchy hypersurface $\Sigma^{n}\subset M$ with
\begin{enumerate}
 \item[(ii)] \textbf{Ambient Ricci lower bound along $\Sigma$:} $\mathrm{Ric}^M \ge -K\,g$ along $\Sigma$,
  \item[(iii)] \textbf{Bounded second fundamental form:} $\lvert h^\Sigma \rvert \le \Lambda$,
   \item[(iv)]  \textbf{Bounded diameter:} the diameter of $\Sigma$ with respect to the induced Riemannian metric is bounded by $D$; i.e., ${\rm diam}_{g_{|T\Sigma}} (\Sigma)\leq D$.
\end{enumerate}
Then $\M_{n,K,K_\perp, \Lambda, D}$ is sequentially pre-compact in the strong  $\mathrm{pLGH}$-convergence; i.e., for each sequence in $\M_{n,K,K_\perp, \Lambda,D}$ there is a subsequence that strongly $\mathrm{pLGH}$-converges to a covered, forward complete, \LpLS that satisfies the point distinction property \eqref{eq-pdp}. 
\end{ithm}

\setcounter{section}{1}
\begin{rem}
One important form of Gromov’s celebrated pre-compactness theorem states that the class of $n$-dimensional Riemannian manifolds with Ricci curvature bounded below by a constant $K \in \mathbb{R}$ and diameter bounded above by $D>0$ is pre-compact with respect to Gromov--Hausdorff convergence. Theorem \ref{thm:CurvPrecomp} is similar in spirit---curvature and diameter bounds imply pre-compactness---but with key differences. In particular, the global curvature lower bound is assumed \emph{only along timelike vectors} (rather than for all vectors as in the Riemannian case). This provides merely a ``polarized'' control of the geometry, seemingly too weak by itself to guarantee pre-compactness. To also obtain control in the spacelike directions (and thus over the full geometry), two additional assumptions appear both natural and necessary:
\begin{enumerate}
    \item[(ii)] a lower bound on the ambient Ricci curvature along $\Sigma$,
    \item[(iii)] a bound on the second fundamental form of $\Sigma$.
\end{enumerate}

These assumptions are also well motivated from both the physical and PDE viewpoints. Lorentzian manifolds provide the natural geometric framework for Einstein’s equations (EE) of general relativity, which are evolution equations posed on a Cauchy surface $\Sigma$ serving as initial data. In this context:
\begin{itemize}
    \item[(ii)] controls the ambient gravitational field along $\Sigma$, expressed through Ricci curvature, which the EE relate directly to matter and energy;
    \item[(iii)] controls the extrinsic curvature of the initial hypersurface $\Sigma$. Since
    \[
    h^\Sigma = \tfrac{1}{2}\,\mathcal{L}_{\nu} g^{\Sigma},
    \]
    where $\mathcal{L}_{\nu}$ is the Lie derivative along the unit normal $\nu$ and $g^{\Sigma}$ the induced Riemannian metric, a bound on the second fundamental form provides control of the rate of change of the spatial metric. Moreover, in the Hamiltonian formulation, the induced metric $g^{\Sigma}$ plays the role of the configuration variable (``position''), while the second fundamental form $h^\Sigma$ serves as its canonical conjugate momentum.
\end{itemize}

Thus, assumptions (ii)--(iii) furnish geometric control of the initial data $\Sigma$, which propagates into the spacetime $M$ under assumption (i), the lower bound on timelike sectional curvature.  

In light of Gromov’s Riemannian pre-compactness theorem, it is natural to conjecture that assumption (i) could be relaxed from a lower bound on timelike \emph{sectional} curvature to one on timelike \emph{Ricci} curvature.
\end{rem}

In Section \ref{sec-meas-conv}, we introduce a measured variant of the Lorentzian Gromov--Hausdorff convergence and establish a measured pre-compactness result (see Theorem \ref{thm-pre-comp-Ibis}). Finally, in Section \ref{sec-app}, we give four applications. First, we show that Chru\'sciel--Grant approximations of continuous spacetimes \cite{CG:12} are an instance of the Lorentzian Gromov--Hausdorff convergence, introduced in this work:

\setcounter{section}{10}
\setcounter{ithm}{0}
\begin{ithm}[Pointed Lorentzian Gromov--Hausdorff convergence for continuous spacetimes]
  Let $(M,g)$ be a continuous, causally plain\footnote{Using the modified time-separation function of \cite{Lin:24} and adapting the proof of \cite[Lem.\ A.1]{McCS:22} one could drop this assumption.} and globally hyperbolic spacetime and fix $o\in M$. Then there exists an approximation $\hat{g}_n\to g$ locally uniformly from the outside (i.e., such that $g\preceq \hat{g}_{n+1}\preceq \hat{g}_n$ for all $n\in\N$) and there exist coverings $\U$, $\U_n$ of $M$ with respect to $g, \hat{g}_n$ such that $(M,\ell_{\hat{g}_n},o,\U_n) \pLGHtop (M,\ell_g,o,\U)$.
\end{ithm}
As a second application, we show that  timelike sectional curvature lower bounds (in the form of the timelike four-point condition \cite{BKR:24}) are stable under Lorentzian Gromov--Hausdorff convergence:
\setcounter{ithm}{3}
\begin{ithm}[Stability of the four-point condition]
\ \\Let $\XtnoU \pLGHtop \XtoU$. Assume that each $\Xtn$ has global timelike sectional curvature bounded below by $K\in\R$ and that $\tau=\max(0,\ell)$ is continuous. Then $\Xt$ has global timelike sectional curvature bounded below by $K$.
\end{ithm}
As a third application, we introduce timelike blow-up tangents in Definition \ref{def-tl-blow-up-tc} and study their basic properties. The last application is a rigorous mathematical statement connected to the main conjecture (the so-called  ``Hauptvermutung'') of causal set theory: we show that if a sequence of causal sets faithfully embeds and approximates two smooth globally hyperbolic spacetimes, then they are isometric (see Theorem \ref{thm-hau-ver}). We conclude the paper with Appendix \ref{app}, which is of independent interest, where a different type of completion is used to ensure that global hyperbolicity is preserved in the limit.

\setcounter{section}{1}

\section{Lorentzian pre-length spaces}\label{sec-lpls}
Here we present a variant of \LpLSn s introduced in \cite{KS:18}. Indeed, its basic axiomatization is still evolving, as is apparent from the variations used in e.g.\ \cite{McC:24, BMcC:23, BBCGMORS:24}. For different approaches to synthetic Lorentzian geometry see for instance  \cite{SV:16, SS:24} based on the null distance and the bounded Lorentzian metric spaces studied in \cite{MS:24, BMS:24}. 
 
\begin{defi}[Reverse triangle inequality and causal relations]
 Let $X$ be a set and let $\ell\colon X\times X\rightarrow \{-\infty\}\cup[0,\infty]$ satisfy the \emph{reverse triangle inequality}
 \begin{equation}\label{eq-rev-tri-ine}
  \ell(x,y) + \ell(y,z) \leq \ell(x,z)\qquad \forall x,y,z\in X\,.
 \end{equation}
  Here we employ the convention that $-\infty + \infty = \infty + (-\infty) = -\infty$ on the left-hand-side.
 The \emph{timelike} $\ll$ and \emph{causal $\leq$ relations} are defined by 
 \begin{equation*}
  \ll:=\ell^{-1}((0,\infty])\,\qquad \leq:=\ell^{-1}([0,\infty])\,.
 \end{equation*}
 Moreover, the \emph{chronological} and \emph{causal future (resp.\ and past)} of a point $x\in X$ are defined by
 \begin{align*}
  I^+(x)&:=\{y\in X: x\ll y\}\,, \qquad I^-(x):=\{y\in X: y\ll x\}\,,\\
  J^+(x)&:=\{y\in X: x\leq y\}\,, \qquad J^-(x):=\{y\in X: y\leq x\}\,.
 \end{align*}
Finally, \emph{chronological} and \emph{causal diamonds} are sets of the form
\begin{equation*}
 I(x,y):= I^+(x)\cap I^-(y)\,,\qquad J(x,y):=J^+(x)\cap J^-(y)\,,
\end{equation*}
and we set $\ell(I(x,y)):= \ell(J(x,y)):=\ell(x,y)$ for $x,y\in X$.
\end{defi}

\begin{defi}[Topologies from the causal relations]
 Let $X$ be a set and $\ell\colon X\times X\rightarrow \{-\infty\}\cup[0,\infty]$. One can endow $X$ with the topologies generated by the sub-base of
 \begin{enumerate}
  \item chronological futures and pasts $I^\pm(x)$ ($x\in X$), called the \emph{chronological topology},
  \item chronological diamonds $I(x,y)$ ($x,y\in X$), called the \emph{Alexandrov topology},
  \item complements of causal futures and pasts $X\backslash J^\pm(x)$ ($x\in X$), called the \emph{causal topology}, and
  \item $I^\pm(x)$, $X\backslash J^\pm(x)$ ($x\in X$), called the \emph{chronocausal topology}.
 \end{enumerate}
\end{defi}

We generalize the definition of a \emph{\LpLS} \cite[Def.\ 2.8]{KS:18} as
\begin{defi}[\LpLSn]\label{defi-lpls}
 Let $X$ be a set, let $$\ell\colon X\times X\rightarrow \{-\infty\}\cup[0,\infty]$$ satisfy the reverse triangle inequality \eqref{eq-rev-tri-ine} and $\ell(x,x)\geq 0$ for all $x\in X$. Let $X$ be endowed with a topology that is finer than the chronological one. Then $\Xt$ is called a \emph{\LpLSn}. Moreover, $\tau:=\max(0,\ell)$ is called the \emph{time-separation function}, while $\ell$ is called the \emph{extended time-separation function}.
\end{defi}
Clearly, every \LpLS in the sense of \cite{KS:18} is a \LpLS in the sense above as lower semi-continuity of $\tau$ implies openness of all $I^\pm(x)$. Moreover, every metric spacetime in the sense of \cite{BBCGMORS:24} is a \LpLS in the sense above if one adds a topology that is finer than the chronological one.
\medskip

Next, we introduce the \emph{causality conditions}, which are another essential aspect of Lorentzian geometry.
\begin{defi}[Causality conditions]
A \LpLS $\Xt$ is called
\begin{enumerate}
    \item \emph{chronological} if $x\not\ll x$ for all $x\in X$ (which is equivalent to $\ell(x,x)=0$ for all $x\in X$),
    \item \emph{causal} if $\leq$ is a partial order, i.e., if $x\leq y \leq x$ implies $x=y$ for all $x,y\in X$, and
    \item \emph{globally hyperbolic} if $\Xt$ is causal, $\leq$ is closed and all causal diamonds are compact, cf.\ \cite{Min:23}.
\end{enumerate}
\end{defi}

Note that this definition of global hyperbolicity is stronger than the one in \cite{KS:18}. However,   by \cite{Min:23}, the two notions agree in large classes of \LpLSn s.

Finally, we introduce
\begin{defi}[Causal convexity]
 A subset $A\subseteq X$ of a \LpLS $X$ is \emph{causally convex} if for all $x,y\in A$ one has that $J(x,y)\subseteq A$.
\end{defi}

\section{Convergence} \label{sec-conv}
In this section we will define $\eps$-nets for \LpLSn s and build the Lorentzian Gromov--Hausdorff convergence on top of them. In doing so we will consider \emph{correspondences} between different \LpLSn s. In the case of metric spaces, a reformulation of Gromov--Hausdorff convergence using correspondences is classical, see e.g.\ \cite[Subsec.\ 7.3.3]{BBI:01}; in the setting of Lorentzian metric spaces this approach was recently employed in  \cite[Sec.\ 4]{MS:24} and \cite{BMS:24} but without using causal diamonds. Our approach is tailored to obtain a Lorentzian pre-compactness result. To this aim we focus on coverings by $\eps$-nets made out of causal diamonds.
\smallskip

Throughout this section, $\Xt$ is a \LpLSn. For notational convenience, we define
\begin{defi}[Set of vertices]\label{def:V(S)}
 Let $S:= \{ J_i := J(p_i,q_i): i\in\Omega\}$ be a set of causal diamonds of $X$. Then the \emph{set of vertices of $S$} is defined as
 \begin{align*}
 V(S):=\{x\in X:\,& x \text{ is a vertex of a causal diamond of } S, \\
 &\; i.e.,\ x=p_i \text{ or } x=q_i \text{ for some } i\in\Omega\}\,.
 \end{align*}
\end{defi}

\begin{defi}[$\eps$-net]\label{def:epsNet}
 Let $\eps>0$ and $A\subseteq X$. An \emph{$\eps$-net $S$ for $A$} is a collection of causal diamonds $S= (J_i)_{i\in \Omega}$ satisfying:
 \begin{enumerate}
  \item  $\tau(J_i)\leq \eps$, for all $i\in \Omega$;
   \item $A \subseteq \bigcup_{i\in \Omega} J_i$.
 \end{enumerate}
Without loss of generality, we will always assume that $J_i\cap A\neq\emptyset$ for all $i\in \Omega$.
\end{defi}

Directly from the definition we obtain
\begin{lem}
 Let $0<\eps\leq \eps'$ and $A\subseteq X$. Let $S$ be an $\eps$-net for $A$ and $S'$ an $\eps'$-net for $A$. Then $S$ and $S\cup S'$ are $\eps'$-nets for $A$.
\end{lem}

Now we recall the notion of \emph{correspondence} and its \emph{distortion}, following the classical metric \cite[Subsec.\ 7.3.3]{BBI:01} and  Lorentzian case \cite[Sec.\ 4]{MS:24} and \cite{BMS:24}.

\begin{defi}[Correspondences and their distortions]
 Let $X,Y$ be sets. A binary relation $R\subseteq X\times Y$ is a \emph{correspondence} if
 \begin{enumerate}
  \item for all $x\in X$ there is a $y\in Y$ such that $(x,y)\in R$ and
  \item for all $y\in Y$ there is a $x\in X$ such that $(x,y)\in R$.
 \end{enumerate}
The \emph{distortion} of a correspondence $R$ between two \LpLSn s $\Xt$, $(Y,\rho)$ is
\begin{equation*}
 \dis(R):= \sup_{(x,y), (x',y')\in R} |\ell(x,x')-\rho(y,y')|\,,
\end{equation*}
with the convention that $\pm \infty  - \pm \infty = 0$ and $\pm \infty - \mp \infty = \mp \infty - \pm \infty = +\infty$. However, later we will only use this construction when the time-separation functions do not attain the value $+\infty$, hence the only relevant cases are $-\infty - (-\infty) = 0$ and $|-\infty - z| = |z - (-\infty)| = +\infty$, for all $z\in \R$.
\end{defi}

Note that, if $R$ has finite distortion, then it preserves the  causal relations, i.e., $x\leq x'$ if and only of $y\leq y'$ for all $(x,y), (x',y')\in R$.

\begin{defi}[Composition of correspondences]
 Let $X,Y,Z$ be sets, $R\subseteq X\times Y$ a correspondence between $X$ and $Y$, and $Q\subseteq Y\times Z$ a correspondence between $Y$ and $Z$. Then the \emph{composition} $Q\circ R$ of $R$ and $Q$ is defined as 
 \begin{equation*}
  Q\circ R:= \{(x,z)\in X\times Z: \exists y\in Y \text{ such that } (x,y)\in R, (y,z)\in Q\}\,.
 \end{equation*}
\end{defi}

Clearly, $Q\circ R$ is a correspondence between $X$ and $Z$ and if all three spaces are \LpLSn s, then  the distortion satisfies a useful sub-additivity property under composition \cite[Lem.\ 4.8]{MS:24}, i.e., 
\begin{equation}\label{eq-com-dis}
 \dis(Q\circ R) \leq \dis(Q) + \dis(R)\,.
\end{equation}
The \emph{inverse} $R^{-1}$ of a correspondence $R$ between $X$ and $Y$ is defined as
\begin{equation*}
 R^{-1} := \{(y,x)\in Y\times X:(x,y)\in R\}\,.
\end{equation*}
Clearly, $R^{-1}$ is a correspondence between $Y$ and $X$, If both $X$ and $Y$ are \LpLSn s, then $\dis(R^{-1}) = \dis(R)$.  Any correspondence $R$ yields a map $f\colon X\rightarrow Y$, by choosing $f(x)\in Y$ such that $(x,f(x))\in R$. It is clear that 
\begin{equation*}
 \dis(f):= \sup_{x,x'\in X} |\ell(x,x') - \rho(f(x),f(y))| \leq \dis(R)\,.
\end{equation*}
Conversely,  a surjective map $f\colon X\rightarrow Y$ gives rise to a correspondence $R:=\{(x,f(x)):x\in X\}$ of $X$ and $Y$, with the same distortion of $f$.

\medskip
Using the notation above, we next define the Lorentzian Gromov--Hausdorff convergence,  for subsets of \LpLSn s.
\begin{defi}[LGH-convergence of subsets]\label{def-con-subs}
 Let $\Xtn$ and $\Xt$ be \LpLSn s, for $n\in\N$.  For each $n\in\N$, let $A_n$ be a subset of $X_n$ and let $A$ be a subset of $X$. We say that $A_n$ converges to $A$ in Lorentzian Gromov--Hausdorff sense ({\rm LGH} for short), and write $A_n\LGHtop A$, if for all $\eps>0$ there exist $n_0\in\N$ and finite $\eps$-nets $S$ for $A$ in $X$ and $S_n$ for $A_n$ in $X_n$ (for all $n\geq n_0$) such that
 \begin{enumerate}
 \item 
 $|S_n| = |S|$;
 \item For all $n\geq n_0$, there exists a correspondence $R_n$ of $V(S_n)$ and $V(S)$, with $\dis(R_n)\to 0$ as $n\to\infty$.
  \item\label{def-con-subs-ext-prop} (Extension property of correspondences) For all $l\in \N$, each $\frac{1}{l}$-net in the limit can be enlarged to include an $\frac{1}{l+1}$-net while preserving convergence of the vertices. More precisely: for every $l\in\N,l\geq 1$ let $S^l$ and $S^l_n$ be $\frac{1}{l}$-nets for $A$ and $A_n$, respectively, as above. Let $f^l_n\colon V(S^l)\rightarrow V(S^l_n)$ be a map realizing a correspondence of $V(S^l)$ and $V(S^l_n)$. Then there exist an extension $f^{l+1}_{n'}\colon V(S^l)\cup V(S^{l+1})\rightarrow V(S^l_{n'})\cup V(S^{l+1}_{n'})$ of $f^l_{n'}$ with $\dis(f^{l+1}_{n'})\leq \dis(f^l_{n'})$, for some $n'\geq n$.
    \item\label{def-con-subs-for-den} (Forward density)
     Every collection of $\frac{1}{l}$-nets as above is \emph{forward dense} in the limit space, in the following sense. For every $l\in\N,l\geq 1$, let $S^l$ be a $\frac{1}{l}$-net for $A$. Set $\V:=\bigcup_{l=1}^\infty V(S^l)$ to be the total set of vertices. Then, for each $x\in A\backslash\V$, there exists a sequence $(x_k)_k \in\V$ such that $x_k\leq x_{k+1}\leq x$ and $x_k\to x$.
 \end{enumerate}
 We say that $A_n\LGHtop A$ \emph{strongly} if for each $x\in A\backslash\V$, there exists a sequence $(x_k)_k \in\V$ such that $x_k\ll x_{k+1}\ll x$ and $x_k\to x$. Such a reinforcement of point  \ref{def-con-subs-for-den} will be called \emph{timelike forward density}.
\end{defi}

\begin{ex}
Note that the requirement of finite $\eps$-nets rules out e.g.\ a spacelike strip in Minkowski spacetime.
\end{ex}

As discussed in the introduction, there is no canonical cover of a pointed, unbounded, Lorentzian space. Therefore, it is convenient to specify a cover when discussing convergence of pointed, unbounded, \LpLSn s.
\begin{defi}[\LpLSn s with a cover]\label{def-cov-lpls}
 A \emph{covered \LpLSn} $(X,\ell,o, \U)$ is a pointed \LpLS $(X,\ell,o)$, $o\in X$, with a countable cover $\U=(U_k)_{k\in\N}$ of $X$ such that
 \begin{enumerate}
  \item $U_k\subseteq U_{k+1}$, for all $k\in\N$;
  \item $o\in U_k$, for all $k\in\N$;
  \item $\sup_{x,y\in U_k} \tau(x,y) <\infty$, for all $k\in\N$. 
 \end{enumerate}
If all $U_k$s are relatively compact, we say that the cover $\U$ is \emph{proper}. A \LpLS with a proper cover is said to be \emph{properly covered}.
\end{defi}

\begin{rem}[Analogies and differences with pointed metric spaces]
Let us stress that the key structure in Definition \ref{def-cov-lpls}, added on top of a Lorentzian pre-length space, is the covering $\U=(U_k)_{k\in\N}$.  The only assumption on the marked point  $o\in X$ is to be contained in each element of the covering; a careful reader will notice that, throughout the paper,  one could take any other marked point $o'\in U_0$ without affecting the arguments and the convergence. In this sense, one may argue that specifying the marked point $o\in X$ is slightly redundant. We decided to keep it in the notation, although not strictly necessary, in order to stress that all the elements of the covering must contain the marked point $o$, and for analogy with pointed metric spaces. Indeed, for a non-compact metric space $(X,\mathsf{d})$, it is convenient to fix a marked point $\bar{x}$. A natural covering of $X$ is then given by the metric balls $B_N(\bar{x})$ centered at $\bar{x}$ and diverging radii $N\in \N$. For a Lorentzian manifold, there is not such a natural exhaustion by bounded subsets around a marked point (observe that the sub-level sets of the time-separation are typically non-compact, as for instance it is the case in Minkowski spacetime). Fixing an exhaustion by ``bounded'' subsets also in Lorentzian signature will be key in the pre-compactness results proved below.
\end{rem}

\begin{rem}[A covered \LpLS is chronological]\label{rem-cov-lpls-chr}
The time-separation function of a covered \LpLS is finite-valued and hence zero on the diagonal. In particular,  the space is chronological, i.e., $x\not\ll x$ for all $x\in X$.
\end{rem}

\begin{lem}[Covering of spacetimes]\label{lem-cov-st}
 Let $(M,g)$ be a globally hyperbolic spacetime with continuous Lorentzian metric, where $I^+$ is an open relation\footnote{For example, by using piecewise $\Con^1$-curves \cite{GKSS:20} or via causal plainness \cite{CG:12} or use nearly timelike curves \cite{Lin:24}.}. Fix $o\in M$. Then there is a cover $\U = (U_k)_{k\in\N}$ of $M$ such that 
 \begin{enumerate}
  \item Each $U_k$ is open, causally convex and relatively compact;
  \item  $o\in U_k \subseteq U_{k+1}$, for all $k\in\N$.
 \end{enumerate} 
In particular $(M,\ell_g, o, \U)$ is a properly covered \LpLS in the sense of Definition \ref{def-cov-lpls}, where 
$$\ell_g(p,q):=\sup\{L_g(\gamma): \gamma \text{ future directed causal from $p$ to $q$}\}$$
is the time-separation function with respect to $g$.
\end{lem}
\begin{pr}
 As the manifold topology of $M$ is second countable and locally compact there is a countable cover $\V = (V_i)_{i\in\N}$ of $M$ by open and relatively compact sets. Fix $i\in \N$ such that $o\in V_i$. Set $U_0:= I^+(\bigcup_{j=0}^i V_i) \cap I^-(\bigcup_{j=0}^i V_i)$. Then $U_0$ is open, causally convex and by globally hyperbolicity it is relatively compact as $\bigcup_{j=0}^i V_i$ is relatively compact. Moreover, $o\in U_0$ as by openness of $V_i$ we can find a small chronological diamond centered at $o$ that is contained in $V_i$. For $k\in\N, k\geq 1$ set
 \begin{equation*}
  U_k := I^+(U_{k-1}\cup V_k)\cap I^-(U_{k-1}\cup V_k)\,.
 \end{equation*}
Then again each $U_k$ is open, causally convex and relatively compact. It remains to show monotonicity. To this end let $u\in U_{k-1}$, then by openness of $U_{k-1}$ we find a small chronological diamond centered at $u$ that is contained in $U_{k-1}$. Thus $u\in I^+(U_{k-1})\cap I^-(U_{k-1})\subseteq U_k$. Analogously, one establishes that $V_k\subseteq U_k$ and hence $M= \bigcup_{k\in\N} V_k \subseteq \bigcup_{k\in\N} U_k \subseteq M$, hence $\U = (U_k)_{k\in\N}$ covers $M$. Finally, note that $\tau$ is continuous and finite-valued on $M$, hence bounded on each $U_k$ by relative compactness. This shows that $\U$ is a cover in the sense of Definition \ref{def-cov-lpls} and finishes the proof.
\end{pr}

Finally, we are in a position to define pointed Lorentzian Gromov--Hausdorff convergence for covered \LpLSn s.
\begin{defi}[pLGH-convergence of covered \LpLSn s]\label{def-con-cov}
 Let $\XtoU$ and $\big(\XtnoU\big)_{n\in \N}$ be covered \LpLSn s with $\U=(U_{k,\infty})_{k\in\N}$ and $\U_n = (U_{k,n})_{k\in\N}$. We say that $\big(\XtnoU\big)_{n\in \N}$ converges to $\XtoU$ in the (resp.\,strong) pointed Lorentzian Gromov--Hausdorff sense ({\rm pLGH}, for short), and write $$\XtnoU\pLGHtop\XtoU \quad \text{(resp.\;strongly)}$$
 if for each $k\in\N$ it holds that $U_{k,n}\LGHtop U_{k,\infty}$  (resp.\;strongly) as $n\to\infty$.
\end{defi}

We conclude this section by considering convergence of Lorentzian products where the metric fibers Gromov--Hausdorff converge.
\begin{ex}[Convergence of Lorentzian products]
 Let $((X_n,\mathsf{d}_n))_{n\in \N}$ be a sequence of compact metric spaces that converges in the Gromov--Hausdorff sense to a compact metric space $(X_\infty,\mathsf{d}_\infty)$. Consider the Lorentzian products $Y_n := [-1,1]\times X_n$ with time-separation functions
 \begin{align*}
  \ell_n((t,x),(t',x')):=\begin{cases}
                          \sqrt{(1+\frac{1}{n})^2(t'-t)^2 - \mathsf{d}_n(x,x')^2} &t\leq t'\text{ and}\\
                          &\mathsf{d}_n(x,x')\leq (1+\frac{1}{n})(t'-t)\,,\\
                          -\infty  &\text{otherwise}\,,
                         \end{cases}
 \end{align*}
where $t,t'\in [-1,1]$, $x,x'\in X_n$ for $n\in\N\cup\{\infty\}$ (and $\frac{1}{\infty}=0$). We slightly enlarge the causal relation of $\ell_n$ by adding the $\frac{1}{n}$-term so that causally related points in the limit will be eventually also causally related with respect to $\ell_n$. Clearly, each product $(Y_n,\ell_n)$ is a \LpLS with the product topology, which is finer than the chronological topology (see \cite{AGKS:23} for more details). We claim that, up to taking subsequences: 
$$
\left[0,\frac{1}{2}\right]\times X_n\LGHtop \left[0,\frac{1}{2}\right]\times X_\infty \quad \text{strongly.}
$$
To see this, let $0<\eps<1$. By the Gromov--Hausdorff convergence, see e.g.\ \cite[Prop.\ 7.4.12]{BBI:01}, there are metric $\frac{\eps}{3}$-nets $S_n^\eps\subseteq X_n$ for $n\in\N\cup\{\infty\}$ such that (for large $n\in\N$) they have the same (finite) cardinality and $S_n^\eps \GHto S_\infty^\eps$ as $n\to\infty$. Thus, without loss of generality, we assume that $S_n^\eps = \{s^{n,\eps}_1,\ldots, s^{n,\eps}_N\}$ for some fixed $N\in\N$ and all $n\in\N\cup\{\infty\}$. Moreover, by (possibly) taking subsequences, we can ensure that $$|\mathsf{d}_n(s_i^{n,\eps},s_j^{n,\eps}) - \mathsf{d}_\infty(s_i^{\infty,\eps},s_j^{\infty,\eps})|\leq\frac{1}{n^2},$$ for all $n\in\N$ and $i,j$.

Set $L:= \lceil\frac{3}{\eps}\rceil$, $t_i:=i\frac{\eps}{3}$ for $i=-1,\ldots,L+1$, and set $x_{i,j}^{n,\eps}:=(t_i,s_j^{n,\eps})$ for $i=-1,\ldots,L+1$, $j=1,\ldots, N$,  $n\in\N\cup\{\infty\}$ and 
 $$\J^{n,\eps}:=\{J(x_{i-1,j}^{n,\eps},x_{i+2,j}^{n,\eps}):i=0,\ldots,L-1; j=1,\ldots,N\}.$$ 
 Then $\J^{n,\eps}$ is a collection of causal diamonds $J$, satisfying
 $$\tau_n(J) = \tau_n(x_{i-1,j}^{n,\eps},x_{i+2,j}^{n,\eps})= t_{i+2} - t_{i-1} =  \eps,$$ 
 for all $n\in\N\cup\{\infty\}$. 
 
 Next, we show that $[0,\frac{1}{2}]\times X_n\subseteq \bigcup_{J\in\J^{n,\eps}} J$, for all $n\in\N\cup\{\infty\}$. Let $y=(t,x)\in [0,\frac{1}{2}]\times X_n$, then there exists $i\in\{0,\ldots,L\}$ such that $t_i\leq t\leq t_{i+1}$ and $j\in\{1,\ldots,N\}$ such that $\mathsf{d}_n(x,s_j^{n,\eps})<\frac{\eps}{3}$. Consequently, we have $\mathsf{d}_n(x,s_j^{n,\eps}) < \frac{\eps}{3} < \min(t_{i+2}-t, t-t_i)$ and hence
\begin{align*}
\ell_n(x^{n,\eps}_{i-1,j},x) \geq \sqrt{(t-t_{i-1})^2 - \mathsf{d}_n(x,s_j^{n,\eps})^2} \geq 0,
\end{align*}
i.e., $y\in J_{Y_n}^+(x_{i-1,j}^{n,\eps})$ for all $n\in\N\cup\{\infty\}$. Analogously, one shows that $y\in J_{Y_n}^-(x_{i+2,j}^{n,\eps})$ for all $n\in\N\cup\{\infty\}$. In conclusion, we have $\eps$-nets for $[0,\frac{1}{2}]\times X_n$ of the same cardinality and considering the obvious correspondences we see that their distortions are estimated via
\begin{align*}
 &\Bigl|\ell_n(x_{i,j}^{n,\eps},x_{i',j'}^{n,\eps}) - \ell_\infty(x_{i,j}^{\infty,\eps},x_{i',j'}^{\infty,\eps})\Bigr|\\
 &=  \Bigl|\sqrt{(1+\frac{1}{n})^2(t_{i'}-t_i)^2 - \mathsf{d}_n(
 s_{j}^{n,\eps},s_{j'}^{n,\eps})^2} - \sqrt{(t_{i'}-t_i)^2 - \mathsf{d}_\infty(s_{j}^{\infty,\eps},s_{j'}^{\infty,\eps})^2}\,\Bigr|\,.
\end{align*}
This holds if the points are causally related. However, we see that if\\ $\ell_\infty(x_{i,j}^{\infty,\eps},x_{i',j'}^{\infty,\eps})\geq~0$, i.e., they are causally related, then so are $x_{i,j}^{n,\eps},x_{i',j'}^{n,\eps}$ eventually. So the distortions converge as $\mathsf{d}_n(s_{j}^{n,\eps},x_{j'}^{n,\eps})\to \mathsf{d}_\infty(s_{j}^{\infty,\eps},s_{j'}^{\infty,\eps})$ by assumption. Moreover, the extension property of the correspondences is clear by the construction.
\\To show timelike forward density, let $y=(t,x)\in [0,\frac{1}{2}]\times X_\infty$ that is not a vertex of an $\frac{1}{l}$-net $(l\in\N,l\geq 1)$. In particular, we have $t\geq 0$. There is a sequence $(s_k)_k$ in $\bigcup_{l=1}^\infty S^{\frac{1}{l}}_\infty$ such that $s_k\to x$ in $X_\infty$. Now, for every $k\in\N$ we can find a $t_k\in[-\frac{1}{3},1]$ such that $t_k<t$, $t-t_k\to 0$ and $(t_k,s_k)$ is a vertex of an $\frac{1}{l}$-net. Clearly, $(t_k,s_k)\to y$ in $Y_\infty$. It is now straightforward to extract a monotone subsequence.

A similar construction is carried out later for smooth globally hyperbolic spacetimes in the context of the geometric pre-compactness theorem, see Lemma \ref{lem-card-eps-prod}.
\end{ex}

\section{Uniqueness of limits} \label{sec-uniq}
In analogy with Gromov--Hausdorff convergence of general metric spaces, also in the Lorentzian signature, limits may not be unique.  The key fact in positive signature is that the Gromov--Hausdorff limit of  a sequence of \emph{compact} metric spaces is unique (up to isometry), cf.\ e.g.\ \cite[7.3.30]{BBI:01}. In the non-compact case, the pointed Gromov--Hausdorff limit of \emph{complete} metric spaces is unique, see for instance \cite[Thm.\ 8.1.7]{BBI:01}. In Lorentzian signature, the natural analog of proper metric spaces are \emph{globally hyperbolic} \LpLSn s, while the natural analog of complete metric spaces are \emph{forward complete} \LpLSn s. Inspired by this analogy,  in this section we will show that Lorentzian Gromov--Hausdorff  limits are unique in a suitable class of forward complete \LpLSn s.
\medskip

First, we introduce the \emph{point distinction property} (cf.\ \cite[Def.\ 1.1,(iii)]{MS:24} for a similar property required for $\tau$).
\begin{defi}[Point distinction property]
Let $\Xt$ be a \LpLS and $S\subseteq X$. We say that $S$ has the \emph{point distinction property} if,
 for all $x,y\in X$ with $x\neq y$, there exists $z\in S$ such that
 \begin{align}
  \label{eq-pdp} \tag{$\mathrm{PDP}$} &\ell(x,z) \neq \ell(y,z)\,\text{ or } \,\,\ell(z,x) \neq \ell(z,y)\,.
 \end{align}
We say that $X$ has the point distinction property, if $S=X$ does so.
\end{defi}

\begin{rem}
 Note that the point distinction property \eqref{eq-pdp} holds for $S=X$ if $X$ is \emph{distinguishing} (cf.\ e.g.\ \cite[Def.\ 3.4]{ACS:20}), i.e., if $I^+(x) = I^+(y)$ or $I^-(x) = I^-(y)$ implies $x=y$. In the setting of \LLSn s distinguishability is implied by strong causality \cite[Thm.\ 3.17]{ACS:20}.
\end{rem}

In our setting, the point distinction property \eqref{eq-pdp} implies causality, i.e., that the causal relation is a partial order.
\begin{lem}[PDP implies causality]\label{lem-pdp-cau}
 Let $\Xt$ be a \LpLS  satisfying the point distinction property \eqref{eq-pdp}. Then $\Xt$ is causal, i.e., $\leq$ is antisymmetric, hence a partial order.
\end{lem}
\begin{pr}
 Let $x\leq y\le x$ and $z\in X$. Since $\ell(x,y)\geq 0$ and $\ell(y,x)\geq 0$, then 
 \begin{align*}
  \ell(x,z) \leq \ell(y,x) + \ell(x,z)\leq \ell(y,z)\leq \ell(x,y) + \ell(y,z) \leq \ell(x,z)\,.
 \end{align*}
Similarly, one shows that $\ell(z,x)=\ell(z,y)$. Consequently, the point distinction property \eqref{eq-pdp} implies that $x=y$.
\end{pr}

We next introduce the notion of isometry between \LpLSn s and note that we call a classical isometry between spacetimes a \emph{smooth isometry}.
\begin{defi}[Isometry]
 Let $\Xt$ and $(\tilde X,\tilde \ell)$ be two \LpLSn s. A map $f\colon X\rightarrow \tilde X$ is
 \begin{enumerate}
  \item \emph{$\ell$-preserving} if 
  \begin{equation*}
   \tilde\ell(f(x),f(y)) = \ell(x,y)\quad \forall x,y\in X\,,
  \end{equation*}
\item \emph{$\tau$-preserving} if
\begin{equation*}
   \tilde\tau(f(x),f(y)) = \tau(x,y)\quad \forall x,y\in X\,,
  \end{equation*}
  and 
  \item an \emph{isometry} if it is $\ell$-preserving and bijective.
 \end{enumerate}
\end{defi}

Clearly, if $X$ has the point-distinction property \eqref{eq-pdp}, then an a $\ell$-preserving map is injective.
\medskip

 Next we recall the notion of \emph{forward completeness} --- a concept first introduced in \cite{BBCGMORS:24} for Lorentzian spaces and later elaborated in \cite{Gig:25} in the general framework of partially ordered spaces.
\begin{defi}[Forward completeness]\label{defi-for-compl}
 Let $\Xt$ be a \LpLS and let $(x_k)_k\subset X$ be a sequence. We say that
 \begin{itemize}
 \item $(x_k)_k$ is \emph{monotone increasing and bounded} if there exists $\hat{x}\in X$ such that $x_k \leq x_{k+1}\leq \hat{x}$, for all $k\in \N$.
\item  $\Xt$ is  \emph{forward complete} if every monotone increasing and bounded sequence $(x_k)_{k\in\N}$  converges.
\end{itemize}
\end{defi}

\begin{rem}[Global hyperbolicity implies forward completeness]\label{rem-gh-for-compl}
A globally hyperbolic \LpLS is forward complete as any sequence $x_k\leq x_{k+1}\leq \hat{x}$ is contained in the compact causal diamond $J(x_0,\hat x)$ and $\leq$ being a partial-order implies that any converging subsequence converges to the same limit (see also \cite[Rem.\ 2.10]{BBCGMORS:24}).
\end{rem}

Next, we give a sufficient condition guaranteeing that limits are unique within the class of globally hyperbolic \LpLSn s with metrizable chronological topology, continuous time-separation function and satisfying the point-distinction-property \eqref{eq-pdp} (actually, we prove a more general statement, where \emph{globally hyperbolic} is relaxed to \emph{forward complete, properly covered, with closed anti-symmetric causal relation}). Of course, uniqueness will be  always understood \emph{up to isometry}. Note that we strengthen here Definition \ref{def-con-subs}, by requiring that whenever we have a finite $\eps$-net in the approximating sequence, we find a corresponding $\eps$-net in the limit. Definition \ref{def-con-subs} only requires independent existence of both.

\begin{thm}[Sufficient condition for uniqueness of limits]\label{thm-uni-lim}
 Let \\$\big(\XtnoU\big)_{n\in \N}$, be a sequence of covered \LpLSn s satisfying the following property: For any strong $\mathrm{pLGH}$-limit $\XtoU$ it holds that, for all $\eps>0$ and for all finite $\eps$-nets $S_n$ for $U_n\in\U_n$ in $X_n$, there exists a finite $\eps$-net $S$ for $U\in\U$ in $X$ such that properties (i)---(iv) of Definition \ref{def-con-subs} hold. Then, the strong  $\mathrm{pLGH}$-limit of the sequence $\big(\XtnoU\big)_{n\in \N}$ is unique in the class $\X$ of forward complete, properly covered \LpLSn s $X$ with continuous time-separation functions $\tau$, closed anti-symmetric causal relation, metrizable chronological topology and that satisfy the point-distinction property \eqref{eq-pdp} and $x\in\overline{I^-(x)}$ for all $x\in X$.
\end{thm}
\begin{pr}
Let $\XtoU$ and $(\tilde X, \tilde \ell, \tilde o, \tilde\U)$ be both strong pLGH limits of the sequence $\big(\XtnoU\big)_{n\in \N}$, 
where both $\XtoU$ and $(\tilde X, \tilde \ell, \tilde o, \tilde\U)$ are forward complete, properly covered, having closed and anti-symmetric causal relations, satisfy the point distinction property \eqref{eq-pdp}, have metrizable chronological topologies and have continuous time-separation functions $\tau$ and $\tilde\tau$, respectively.
\smallskip

We show that  each of the covering sets $U_k\in\U$, $\tilde U_k\in\tilde\U$ are isometric.
\\ Fix $k\in\N$ and write $U:=U_k\in\U$, $\tilde U:=\tilde U_k\in\tilde\U$ and $U_n:=U_{k,n}\in\U_n$. Moreover, we write $\ell$ and $\tau$ instead of $\ell\rvert_{U\times U}$ and $\tau\rvert_{U\times U}$, respectively, and analogously just $\tilde\ell$ and $\tilde\tau$. By assumption, for each $l\in\N,l\geq 1$ we obtain $\frac{1}{l}$-nets $S_n^l$ for $U_n$, $S^l$ for $U$ and $\tilde S^l$ for $\tilde U$, together with correspondences $R_n^l$, $\tilde R_n^l$ between $V(S_n^l)$ and $V(S^l)$, and between $V(S_n^l)$ and $V(\tilde S^l)$, respectively. Thus we obtain a correspondence $\tilde R^l_n \circ (R^l_n)^{-1}$ between $V(S^l)$ and $V(\tilde S^l)$ with $\dis(\tilde R^l_n \circ (R^l_n)^{-1})\leq \dis(\tilde R^l_n) + \dis(R^l_n) \to 0$ by \eqref{eq-com-dis} for all $l\in\N,l\geq 1$. Consequently, up to  a subsequence, we have a correspondence $Q^l_n$ between $V(S^l)$ and $V(\tilde S^l)$ with $\dis(Q^l_n)\leq \frac{1}{n}$. This gives bijective maps $f^l_n\colon V(S^l)\rightarrow V(\tilde S^l)$ with $\dis(f^l_n)\leq \frac{1}{n}$. We denote the inverse of $f^l_n$, as $g^l_n:=(f^l_n)^{-1}\colon V(\tilde S^l)\rightarrow V(S^l)$. Further, by the extension property (cf.\ Definition \ref{def-con-subs},\ref{def-con-subs-ext-prop}) $f^l_n$ can be defined on $\V^l:=\bigcup_{l'=1}^l V(S^l)$ without increasing its distortion (and similarly for $g^l_n$, which is still the inverse of $f^l_n$ on $\V^l$).

We claim that $f^l_n$ is  $\ll$- and $\leq$-preserving, for $n$ large enough. Let $x,y\in \V^l$ with $x\ll y$ (or $x\leq y$). Then $\tilde\ell(f^l_n(x),f^l_n(y))\geq -\dis(f^n_l) + \ell(x,y)$, and the right-hand-side becomes positive for $\frac{1}{n}< \ell(x,y)$ (or greater than $-\infty$). Since $\V^l$ only has finitely many points, the claim follows. 

Next, let us view $f^l_n$ as a map into $\tilde\V:= \bigcup_{l'=1}^\infty V(\tilde S^{l'})$. By the extension property (cf.\ Definition \ref{def-con-subs},\ref{def-con-subs-ext-prop}), we can achieve that each $f^{l+1}_n$ extends $f^l_n$, while preserving $\dis(f^{l+1}_n)\leq \frac{1}{n}$ (up to taking a further subsequence). Consequently, we obtain a map 
$$f_n\colon \V \rightarrow \tilde \V\quad  \text{with}\quad  \dis(f_n)\leq\frac{1}{n}\,,$$
 where $\V:=\bigcup_l \V^l = \bigcup_l V(S^l)$.

Clearly, $f_n$ is injective, $\ll$- and $\leq$-preserving: indeed, for each $x,y\in \V$ there is an $l\in\N$ such that $x,y\in \V^l$ and $f_n = f_n^l$ on $\V^l$. Moreover, each $g^{l+1}_n$ extends $g^l_n$, so we obtain
$$g_n\colon \tilde\V \rightarrow \V\quad  \text{with}\quad  \dis(g_n)\leq\frac{1}{n}\,,$$
in an analogous manner. Clearly, $g_n=f_n^{-1}$.

Next, we extend $f_n$ to $U$ as follows. By the timelike forward density property (cf.\ Definition \ref{def-con-subs},\ref{def-con-subs-for-den}), for $x\in U\backslash\V$ there is a monotone sequence $x_k\ll x_{k+1} \ll x$ converging to $x$. Let $x^+\in\V\cap J^+(x)$, which exists since the causal diamonds with vertices in $\V$ cover $U$. This yields a monotone sequence 
$$\tilde x_k:= f_n(x_k)\ll \tilde x_{k+1} \leq \tilde x^+:=f_n(x^+)\in\tilde\V.$$ 
By forward completeness, the sequence $(\tilde x_k)_k$ converges to some $\tilde x\in J(\tilde x^-,\tilde x^+)$, for some $\tilde x^-\in\tilde\V$. Setting $f_n(x):= \tilde x$ gives a well-defined map, by forward completeness and causality; moreover $f_n$ is $\leq$-preserving. We claim that for all $x,y\in U$ we have $x\leq y$ if and only if $f_n(x)\leq f_n(y)$. To see this let $x\leq y$ and first consider the case $x,y\not\in\V$ (the case where both are in $\V$ follows since $f_n\colon\V\rightarrow\tilde\V$ preserves $\leq$). Let $x_k\ll x_{k+1}\ll x$, $y_k\ll y_{k+1}\ll y$ for all $k\in\N$ and $x_k\to x$, $y_k\to y$. Then $x_k\ll x\leq y$, hence $x_k\ll y$ for all $k\in\N$ and so for each $k\in\N$, there is a $l_k\in\N$ such that for all $l\geq l_k$ we have $y_l\in I^+(x_k)$ (as $y_l\to y\in I^+(x_k)$). Consequently, we get $f_n(x_k)\ll f_n(y_l)$ for all $k\in\N$ and $l\geq l_k$. Letting $l\to\infty$, using the closedness of $\leq$ and the definition of $f_n(y)$ we obtain $f_n(x_k)\leq f_n(y)$. Similarly, letting $k\to \infty$ yields $f_n(x)\leq f_n(y)$ as required. For the converse assume $f_n(x)\leq f_n(y)$ and let $(x_k)_k, (y_l)_l$ as above. Analogously, to the first implication we get $f_n(x_k)\ll f_n(x)\leq f_n(y)$, and hence $f_n(x_k)\ll f_n(y_l)$ for all $k\in\N$ and all $l$ large. As $f_n^{-1}=g_n$ on $\tilde \V$ preserves $\ll$ we have $x_k\ll y_l$ for all $k\in\N$ and $l$ large. Taking the limit $l\to\infty$ and then $k\to\infty$ yields $x\leq y$ as claimed. The case $x\not\in\V$, $y\in\V$ follows analogously, whereas for the case $x\in\V$, $y\not\in\V$ we need to use $x\in\overline{I^-(x)}$ to approximate $x$ from below and then conclude as in the first case.

Finally, we show that $\dis(f_n)\leq \frac{1}{n}$. Let $x,y\in U$ and let $x_k\to x$, $y_k\to y$ be sequences as above if $x\not\in\V$ or $y\not\in\V$, and otherwise choose constant sequences. Then we only need to consider the case $x\leq y$ and hence only estimate the difference of the $\tau$s, i.e.,
\begin{align*}
 |\tau(x,y) - \tilde\tau(f_n(x),f_n(y))| \leq & |\tau(x,y) - \tau(x_k,y_k)|\\
 &+ |\tau(x_k,y_k) - \tilde\tau(f_n(x_k),f_n(y_k))|\\
 &+ |\tilde\tau(f_n(x_k)),f_n(y_k)) - \tilde\tau(f_n(x),f_n(y))|\,,
\end{align*}
where the first and the last term can be made arbitrarily small due the continuity of $\tau$, $\tilde\tau$ and since $f_n(x_k)\to f_n(x)$, $f_n(y_k)\to f_n(y)$ by construction. Moreover, the middle term is bounded by $\dis(f_n)\leq\frac{1}{n}$. Consequently, we get that $\dis(f_n)\leq\frac{1}{n}$ for all $n\in\N, n\geq 1$.

 Applying the same procedure to $g_n$ produces a map $g_n\colon \tilde U\rightarrow \overline{U}$ with $\dis(g_n)\leq \frac{1}{n}$. We claim that $g_n=f_n^{-1}$. To see this we only need to consider points $\tilde x\in\tilde U\backslash\tilde\V$. Let $\tilde x\in\tilde U\backslash\tilde\V$, then there is a sequence $\tilde x_k \ll \tilde x_{k+1}\ll\tilde x$ with $g_n(\tilde x_k)\to g_n(\tilde x)$. As $g_n = f_n^{-1}$ on $\tilde \V$ we obtain $f_n^{-1}(\tilde x_k)\ll f_n^{-1}(\tilde x_{k+1})\ll g_n(\tilde x)$ and $f_n^{-1}(\tilde g_k)\to g_n(\tilde x)$. Thus by the definition of $f_n$ (and it being well-defined) we get that
\begin{equation*}
 f_n(g_n(\tilde x)) = \lim_k f_n(f_n^{-1}(\tilde x_k)) = \lim_k \tilde x_k = \tilde x\,, 
\end{equation*}
as claimed.

\smallskip

We are now ready to show that $U$ and $\tilde U$ are isometric. To this aim, we prove that there exists a limit map $f\colon \V\rightarrow \overline{\tilde U}$ with $\dis(f)=0$, obtained as a limit of $(f_n)_n$ for $n\to\infty$. We use compactness of $\overline{U}$ and $\overline{\tilde U}$. Enumerate $\V$ and $\tilde \V$ as $\V=\{s_k\}_k$ and $\tilde \V =\{\tilde s_k\}_k$. Then for each $k$ there is a sequence $(n^k_m)_m$ such that $f_{n^k_m}(s_k)$ converges in $\overline{\tilde U}$ and $g_{n^k_m}(\tilde s_k)$ converges in $\overline{U}$. We call this limits $f(s_k)$ and $g(\tilde s_k)$, respectively. We can also arrange it so that $(n^{k+1}_m)_m$ is a subsequence of $(n^k_m)_m$ for all $k\in\N$. We claim that $f$ preserves $\ell$. Similarly to the above we only need to estimate the differences of the $\tau$s (since $f(s_k)\leq f(s_l)$ if and only if $s_k\leq s_l$ as can be shown in the same way as for $f_n$), thus let $s_k,s_l\in\V$ with $s_k\leq s_l$ and without loss of generality that $k\leq l$. Then
\begin{align*}
 |\tau(s_k,s_l) - \tilde\tau(f(s_k),f(s_l))| &=  \lim_{m\to\infty} |\tau(s_k,s_l) - \tilde\tau(f_{n^k_m}(s_k),f_{n^l_m}(s_l)|\\
 &= \lim_{m\to\infty} |\tau(s_k,s_l) - \tilde\tau(f_{n^l_m}(s_k),f_{n^l_m}(s_l)| = 0\,,
\end{align*}
where we used the continuity of $\tilde\tau$ and $\dis(f_{n^l_m}) \leq \frac{1}{n^l_m}$. Analogously, we have that $g$ preserves $\tilde\ell$. Now we can extend $f$ from $\V$ to $U$ as above while preserving $\ell$: for each $x\in U\backslash\V$ there is a monotone sequence $x_k\ll x_{k+1} \ll x$ converging to $x$, then as before $\exists\lim_{k\to\infty} f(x_k) =: f(x)\in\overline{\tilde U}$ and we have $\dis(f)=0$. Performing the same procedure for $g$ produces a map $g\colon\tilde U\rightarrow \overline{U}$ with $\dis(g)=0$ with $g\rvert_{\tilde\V} = f^{-1}\rvert_{\tilde \V}$. To see the latter, observe that by definition
\begin{align*}
 g(\tilde s_k) = \lim_{m\to\infty} g_{n^k_m} (\tilde s_k) = \lim_{m\to\infty} f_{n^k_m}^{-1}(\tilde s_k)\,.
\end{align*}
Hence, applying $f_{n^k_m}$ and taking the limit as $m\to\infty$, we obtain $f(g(\tilde s_k)) = \tilde s_k$.

At this point we directly show that $g\rvert_{\tilde U}= f^{-1}$ along similar lines as the proof of $g_n=f_n^{-1}$ above. Let $\tilde x\in \tilde U$ and $\tilde x_k\ll\tilde x_{k+1}\ll \tilde x$, $\tilde x_k\in\tilde \V$ for all $k\in\tilde\V$ with $\tilde x_k\to \tilde x$. Then $g(\tilde x_k)\ll g(\tilde x_{k+1}) \ll g(\tilde x)$, which converges by definition to $g(\tilde x)$. Thus by definition of $f$ we get $f(g(\tilde x))=\lim f(g(\tilde x_k)) = \lim \tilde x_k = \tilde x$, where we used $g\rvert_{\tilde\V} = f^{-1}\rvert_{\tilde \V}$.

\smallskip

The construction above was carried out for a fixed pair of covering sets $U=U_k\in\U$ and $\tilde U = \tilde U_k\in\tilde\U$. It is clear that starting from $U_0$ and $\tilde U_0$, it is possible to iterate the procedure by adding vertices $\V_{k+1}$ and $\tilde \V_{k+1}$ to obtain maps $f^k\colon U_k\rightarrow \overline{\tilde U_k}$ and $g^k\colon\tilde U_k\rightarrow\overline{U}_k$ with $f^{k+1}\rvert_{U_k}= f^k$ and $g^{k+1}\rvert_{\tilde U_k} = g^k$, for all $k\in\N$. This yields global, $\ell$-preserving maps $f\colon X\rightarrow \tilde X$ and $g\colon \tilde X\rightarrow X$. Then $f=g^{-1}$ and the two spaces are isometric.
\end{pr}

\bigskip
Uniqueness of the limit is also guaranteed in case one can find a timelike forward dense set of vertices of $\eps$-nets in the sequence of spaces $X_n$. However, first we need the following lemma (cf.\ \cite[Thm.\ 3.3]{MS:24} for a similar result in the setting of bounded Lorentzian metric spaces).
\begin{lem}[Time-separation preserving self-maps and surjectivity]\label{lem-dis-pre-sel-map-sur}
 Let $\Xt$ be a \LpLSn, where the chronological topology is metrizable, $\tau$ is continuous and vanishes on the diagonal. Let $h\colon X \rightarrow X$ be a $\tau$-preserving map and let $K\subseteq X$ be compact with $h(K)\subseteq K$. Then the interior of $K$ is contained in the image of $K$, i.e., $K^\circ\subseteq h(K)$.
\end{lem}
\begin{pr}
First of all, observe that since $h$ is $\tau$-preserving, it is continuous with respect to the chronological topology. Assume by contradiction that there exists $z\in K^\circ\backslash h(K)$.

\medskip
 \noindent\underline{Step 1:} There is a neighborhood of $z$ disjoint from $h(K)$.\\
 If not, there is a sequence $z_n\to z$ with $z_n\in h(K)$ for all $n\in\N$. Writing $z_n=h(z_n')$ for $z_n'\in K$, we can assume by compactness of $K$, that $z_{n}'\to z'\in K$. Thus, by continuity of $h$, we would have $z = \lim_n z_n = \lim_n h(z_{n}') = h(z') \in h(K)$ --- a contradiction. So there exists a chronological neighborhood of $z$ that does not intersect $h(K)$.
 
\medskip
  \noindent\underline{Step 2:} First, we consider the case that
\begin{equation}\label{eq:zinI+p-hK}
z\in I^+(p)\subseteq K^\circ\backslash h(K), \text{ for some } p\in K.
\end{equation}
 We consider the sequences $(h^n(z))_n$, $(h^n(p))_n$, obtained by  applying  iteratively $h$, starting from  $z$ and $p$, respectively. By compactness of $K$, there exists an increasing subsequence $(n_k)_k$ such that $h^{n_k}(z)\to \tilde z\in K$, $h^{n_k}(p) \to \tilde p\in K$. Recalling that $h$ preserves $\tau$, we infer that  
 \begin{align*}
  0 < \tau(p,z) = \lim_{k\to\infty} \tau(h^{n_k}(p),h^{n_k}(z)) = \tau(\tilde p, \tilde z)\,,
 \end{align*}
by continuity of $\tau$. Consequently, for $k$ large enough, it holds that $0<\tau(h^{n_k}(p),\tilde z)$. Therefore,   for $m>k$ large enough,  we infer that
\begin{align*}
 0< \tau(h^{n_k}(p),h^{n_m}(z)) = \tau(p,h^{n_m - n_k}(z))\,,
\end{align*}
which implies $h^{n_m-n_k}(z)\in h(K)\cap I^+(p)$ since $n_m - n_k>0$. This is a contradiction to $I^+(p)\cap h(K)=\emptyset$, i.e., equation \eqref{eq:zinI+p-hK}.

 The case $z\in I^-(p)\subseteq K^\circ\backslash h(K)$, for some $p\in K$, and the general case (finite intersections) are  analogous.
\end{pr}

\begin{prop}[Uniqueness if the set of vertices is dense in the sequence]\label{prop-uni-den}
 Let $\big(\XtnoU\big)_{n\in \N}$ be a sequence of covered \LpLSn s with continuous time-separation functions $\tau_n$. For each $U_{k,n}\in\U_n$ and each $l\in\N,l\geq 1$, let $S^l_n$ be a $\frac{1}{l}$-net for $U_{k,n}$ corresponding to $\frac{1}{l}$-nets of a possible limit. If $\V_n:=\bigcup_{l=1}^\infty V(S^l_n)$ is timelike forward dense in $U_{k,n}$, then the strong {\rm pLGH}-limit of $\big(\XtnoU\big)_{n\in \N}$ (if it exists) is unique in the class $\X$ of properly covered \LpLSn s with continuous time-separation functions $\tau$, closed anti-symmetric causal relation\footnote{Recall that \emph{being forward complete, properly covered, with closed anti-symmetric causal relation} is implied by \emph{global hyperbolicity.}}, metrizable chronological topology and that satisfy the point-distinction property \eqref{eq-pdp}.
\end{prop}
\begin{pr}
 We first proceed as in the proof of Theorem \ref{thm-uni-lim}. Let 
$$\XtnoU \; \pLGHtop \XtoU \quad \text{ and } \quad \XtnoU\pLGHtop (\tilde X, \tilde \ell, \tilde o, \tilde\U),$$
 where the convergences are in the strong sense and both $\XtoU$ and $(\tilde X, \tilde \ell, \tilde o, \tilde\U)$ are forward complete, properly covered \LpLSn s with closed anti-symmetric causal relations,  metrizable chronological topologies, satisfy the point-distinction property \eqref{eq-pdp} and have continuous time-separation functions $\tau$ and $\tilde\tau$, respectively. We show that each of the covering sets $U_M\in\U$, $\tilde U_M\in\tilde\U$ ($M\in\N$) are isometric. Fix $M\in\N$ and write $U:=U_M\in\U$, $\tilde U:=\tilde U_M\in\tilde\U$ and $U_n:=U_{M,n}\in\U_n$. Moreover, we write $\ell$ and $\tau$ instead of $\ell\rvert_{U\times U}$ and $\tau\rvert_{U\times U}$, respectively, and analogously just $\tilde\ell$ and $\tilde\tau$.
 
 Fix $l\in\N,l\geq1$, and let $S^l, S^l_n$ be finite $\frac{1}{l}$-nets for $U$ and $U_n$, respectively, together with a bijective map $f^l_n\colon V(S^l)\rightarrow V(S^l_n)$ realizing the correspondence of the set of vertices and $\dis(f^l_n)\leq\frac{1}{n}$. Similarly, for each $l'\in\N,l'\geq 1$, there are finite $\frac{1}{l'}$-nets $\tilde S^{l'},\tilde S^{l'}_n$ for $\tilde U$ and $U_n$, respectively, together with a bijective map $g^{l'}_n\colon V(\tilde S^{l'})\rightarrow V(\tilde S^{l'}_n)$ such that $\dis(g^{l'}_n)\leq \frac{1}{n}$. As in the proof of Theorem \ref{thm-uni-lim} we can extend $f^l_n$ and $g^l_n$ to bijective maps
 \begin{align*}
   f_n&\colon \V:=\bigcup_{l=1}^\infty V(S^l)\rightarrow \V_n:=\bigcup_{l=1}^\infty V(S^{l}_n), \quad  \text{with}\quad  \dis(f_n)\leq\frac{1}{n}\,,\\ 
g_n&\colon \tilde\V:=\bigcup_{l=1}^\infty V(\tilde S^l)\rightarrow \tilde\V_n:=\bigcup_{l=1}^\infty V(\tilde S^{l}_n), \quad  \text{with}\quad  \dis(g_n)\leq\frac{1}{n}\,.
\end{align*}
We construct a map $h_n\colon\V\rightarrow\overline{\tilde U}$ with $\dis(h_n)\leq\frac{2}{n}$. Let $x\in\V$. If $f_n(x)\in\tilde\V_n$ we set $h_n(x):=g_n^{-1}(f_n(x))\in\tilde\V$. Otherwise, there is, by the assumption of timelike forward density, a sequence $\tilde x_k\in\tilde\V_n$ with $\tilde x_k\to f_n(x)$ and $x_k\ll x_{k+1}\ll f_n(x)$ for all $k\in\N$. Moreover, there is a $\tilde z\in\tilde\V_n$ with $f_n(x)\leq \tilde z$. As $g_n^{-1}$ preserves the causal relations we get a sequence $g_n^{-1}(x_k)\ll g_n^{-1}(x_{k+1}) \ll g_n^{-1}(\tilde z)$. Consequently, by forward completeness, this sequence converges and we set $\lim_{k\to\infty} g_n^{-1}(\tilde x_k)=:h_n(x)$. This is well-defined: Let $\tilde x_k'\ll \tilde x_{k+1}'\ll f_n(x)$ with $\tilde x_k'\to f_n(x)$. Then, for each $r\in\N$, there is a $k_r\in\N$ such that for all $k\geq k_r$ we have that $\tilde x_r'\ll \tilde x_{k} \ll f_n(x)$, hence $g_n^{-1}(\tilde x_r')\ll g_n^{-1}(\tilde x_k)$. Taking the limit $k\to\infty$ and using the closedness of $\leq$ implies that $\tilde g_n^{-1}(x_r')\leq h_n(x)$ for all $r\in\N$. Taking now the limit $r\to\infty$ yields $\lim_{r\to\infty}g_n^{-1}(\tilde x_r')\leq h_n(x)$. Swapping the roles of $(\tilde x_k)_k$ and $(\tilde x_r')_r$ yields $h_n(x)\leq \lim_{r\to\infty}g_n^{-1}(\tilde x_r')$ and so $h_n(x)= \lim_{r\to\infty}g_n^{-1}(\tilde x_r')$ as $\leq$ is a partial order.

At this point we claim that $\dis(h_n)\leq\frac{2}{n}$. Let $x,y\in \V$. The case $f_n(x),f_n(y)\in\tilde \V_n$ is straightforward. We only consider the case $f_n(x),f_n(y)\not\in\tilde\V_n$ (the mixed case can be handled in a simpler manner, cf.\ the proof of Theorem \ref{thm-uni-lim}). Also we can assume that $x\leq y$, hence $f_n(x)\leq f_n(y)$ and only estimate the difference of the time-separation functions. By construction, there are sequences $(\tilde x_k)_k$, $(\tilde y_k)_k$ in $\tilde\V_n$ such that $\tilde x_k\to f_n(x)$, $\tilde y_k\to f_n(y)$ and $\tilde x_k\ll \tilde x_{k+1}\ll f_n(x)$, $\tilde y_k\ll\tilde y_{k+1}\ll\tilde f_n(y)$. Then, we estimate
 \begin{align*}
  |\tau(x,y) - \tilde\tau(h_n(x),h_n(y))| &\leq |\tau(x,y) - \tau_n(f_n(x),f_n(y))|\\
   & \quad + |\tau_n(f_n(x),f_n(y)) - \tau_n(\tilde x_k,\tilde y_k))|\\
   & \quad + |\tau_n(\tilde x_k,\tilde y_k) - \tilde\tau(g_n^{-1}(\tilde x_k),g_n^{-1}(\tilde y_k))|\\
   & \quad + |\tilde\tau(g_n^{-1}(\tilde x_k),g_n^{-1}(\tilde y_k)) - \tilde\tau(h_n(x),h_n(y))|\,.
 \end{align*}
 The first term on the left-hand-side is bounded by $\dis(f_n)\leq\frac{1}{n}$ and the third one by $\dis(g_n^{-1}) = \dis(g_n)\leq\frac{1}{n}$. The second term converges to zero as $\tau_n$ is continuous and $\tilde x_k\to f_n(x)$, $\tilde y_k\to f_n(y)$. Similarly, the last term converges to zero as $\tilde\tau$ is continuous and $g_n^{-1}(\tilde x_k)\to h_n(x)$, $g_n^{-1}(\tilde y_k)\to h_n(y)$ by construction. This shows that $\dis(h^l_n)\leq\frac{2}{n}$.
 
 Swapping the roles of $f_n$ and $g_n$ we obtain a map $\tilde h_n\colon \tilde\V\rightarrow\overline{U}$ as $\tilde h_n(\tilde x) = \lim_{k\to\infty} f_n^{-1}(x_k)$, where $g_n(\tilde x)\not\in\V_n$ and $x_k\ll x_{k+1}\ll g_n(\tilde x)$, $x_k\to g_n(\tilde x)$.
 At this point we can conclude the proof as in Theorem \ref{thm-uni-lim}, i.e., by extending $h_n$ to $h_n\colon U\rightarrow\tilde U$ while preserving $\dis(h_n)\leq\frac{2}{n}$ and then by extending it to $\overline{U}$ while preserving $\dis(h_n)\leq\frac{2}{n}$. Analogously, we extend $\tilde h_n$ to $\overline{\tilde U}$ while preserving $\dis(\tilde h_n)\leq\frac{2}{n}$. Then, as in the proof of Theorem \ref{thm-uni-lim} we take the limit $n\to\infty$ to obtain maps $h\colon \overline{U}\rightarrow \overline{\tilde U}$, $\tilde h\colon\overline{\tilde U}\rightarrow \overline{U}$ with $\dis(h)=\dis(\tilde h)=0$, i.e., they are $\ell$- and $\tilde\ell$-preserving, respectively. 
 
 Finally (and again as in the proof of Theorem \ref{thm-uni-lim}) we do this iteratively for each covering set, while extending the previous maps. This gives global $\ell$- and $\tilde\ell$-preserving maps $h\colon X\rightarrow \tilde X$ and $\tilde h\colon \tilde X\rightarrow X$, which are continuous and injective as $X$ and $\tilde X$ satisfy the point distinction property \eqref{eq-pdp}. Setting $F:= \tilde h\circ h\colon X\rightarrow X$ yields a $\ell$-preserving map such that for each covering set $U$ as above we have $F(\overline U)\subseteq \overline{U}$, and $\overline{U}$ is compact as $X$ is properly covered. Thus we can apply Lemma \ref{lem-dis-pre-sel-map-sur} to conclude that $U\subseteq F(\overline{U})$, which shows surjectivity of $F$ as the $U$s cover $X$ and $F(\overline{U})\subseteq\overline{U}$. This yields surjectivity of $h$ and shows that $h$ is an isometry. 
 \end{pr}
Note that we could have used Lemma \ref{lem-dis-pre-sel-map-sur} at the end of the proof of Theorem \ref{thm-uni-lim} too, but we opted for the more direct way of exhibiting the inverse maps explicitly. It seems not clear how one would establish $\tilde h = h^{-1}$ here directly.

\begin{rem}[Uniqueness for globally hyperbolic Lorentzian length spaces]\label{rem:UniqGH}
The class $\X$, in which we have uniqueness of limits both in Theorem \ref{thm-uni-lim} and Proposition \ref{prop-uni-den}, includes the class of covered globally hyperbolic Lorentzian length spaces, as they are forward complete by Remark \ref{rem-gh-for-compl}, the time-separation functions are continuous by \cite[Thm.\ 3.28]{KS:18}, they are distinguishing (hence satisfy the \eqref{eq-pdp}) by \cite[Prop.\ 3.17]{ACS:20}, have closed causal relation by \cite[Prop.\ 3.15]{ACS:20} and the chronological topology is the Alexandrov topology, which is metrizable by assumption. Moreover, they are properly covered, as it follows from the proof of Lemma \ref{lem-cov-st}.
\end{rem}

\section{Quotients of \LpLSn s}\label{sec-quo}
In this section we show that, by identifying points that cannot be distinguished by the time-separation function $\ell$, one can always assume that the point distinction property \eqref{eq-pdp} holds. Such a procedure does not affect convergence.  For an analogous construction in the setting of bounded Lorentzian metric spaces, see \cite[Subsec.\ 1.3]{MS:24}.

\begin{defi}[Identifying points with same time-separations]
 Let $\Xt$ be a \LpLS and define the following relation $\sim$ on~$X$
 \begin{equation}\label{eq-sim}
  x\sim y\ {:\Leftrightarrow}\ \ell(x,z)=\ell(y,z)\,,\ \ell(z,x)=\ell(z,y)\,\ \forall z\in X\,.
 \end{equation}
\end{defi}

\begin{prop}[Quotient time-separation]
Let $\Xt$ be a \LpLSn. Then the relation $\sim$ defined in \eqref{eq-sim} is an equivalence relation and the quotient $X/_\sim$ is a \LpLS with the quotient topology and the time-separation function
\begin{equation*}
 \ell^\sim([x],[y]) := \ell(x,y)\,,
\end{equation*}
where $[x],[y]$ are the equivalence classes of $x,y\in X$. Finally, $(X/_\sim,\ell^\sim)$ satisfies the point distinction property \eqref{eq-pdp}.
\end{prop}
\begin{pr}
 Clearly $\sim$ is symmetric, transitive and reflexive. Moreover, $\ell^\sim$ is well-defined and the quotient topology is finer than the chronological topology on $X/_\sim$ as $\pi^{-1}(I^\pm([x])) = \bigcup_{x\in[x]} I^\pm(x)$, where $\pi\colon X\rightarrow X/_\sim$ is the quotient map.
\end{pr}

If a sequence converges in the pointed Lorentzian Gromov--Hausdorff sense to a limit $X$, it also converges to its time-separation quotient $X/_\sim$.
\begin{thm}[Time-separation quotient preserves limits]\label{thm-quo-lim}
 If $$\XtnoU\pLGHtop \XtoU \quad  \text{(resp.\ strongly),}$$
  then $$\XtnoU\pLGHtop (X/_\sim,\ell^\sim, \pi(o),(\pi(U_k))_k) \quad  \text{(resp.\ strongly),}$$ where $\U=(U_k)_{k\in\N}$ and $\pi\colon X\rightarrow X/_\sim$ is the quotient map.
\end{thm}
\begin{pr}
 Any $\eps$-net for a subset $A\subseteq X$ is an $\eps$-net for $\pi(A)\subseteq X/_\sim$ with respect to $\ell^\sim$. Moreover, the distortion of correspondences does not change, the extension property of correspondences holds trivially and forward density of the vertices follows from the continuity of $\pi$.
\end{pr}

\begin{rem}\label{rem-null-sl-inf}

Let \( \Xt \) be  a \LpLS satisfying the point-distinction property \eqref{eq-pdp}. Passing to the time-separation quotient identifies all points that are either not causally related, or only null-related, to the past and future with one of three distinguished equivalence classes. Depending on the situation, these classes may or may not be represented by actual points of the space.   One can characterize the following three distinguished points, whenever they belong to \( X \):
 \begin{enumerate}
  \item The point $i^0$ is characterized as 
  \begin{equation*}
   \ell(x,i^0) = \ell(i^0,x)=-\infty\quad \forall x\in X\backslash\{i^0\}\,,
  \end{equation*}
  and $\ell(i^0,i^0)=0$. Such $i^0$ is called \emph{spacelike boundary} in \cite[Rem.\ 1.2,2]{MS:24} in the setting of bounded Lorentzian metric spaces.
  \item The \emph{future null infinity} $n^0_+$ is characterized as
    \begin{equation*}
   \ell(x,n^0_+) = 0\,,\quad \ell(n^0_+,x)=-\infty\quad  \forall x\in X\backslash\{n^0_+\}\,,
  \end{equation*}
  and $\ell(n^0_+,n^0_+)=0$.
  \item The \emph{past null infinity} $n^0_-$ is characterized as
  \begin{equation*}
   \ell(x,n^0_-) = -\infty\,,\quad \ell(n^0_-,x)=0\quad  \forall x\in X\backslash\{n^0_-\}\,,
  \end{equation*}
   and $\ell(n^0_-,n^0_-)=0$.
 \end{enumerate}
\end{rem}

\section{Pre-compactness}\label{sec-pre-com}
Here we give a first pre-compactness result in a completely general setting. To start, we introduce the timelike diameter of a subset of a \LpLSn.

\begin{defi}[Timelike diameter]
 Let $\Xt$ be a \LpLS and let $A\subseteq X$ be a subset. The \emph{timelike diameter} of $A$ is defined by
 \begin{equation*}
  \diam^\tau(A):=\sup_{x,y\in A} \tau(x,y)\,.
 \end{equation*}
\end{defi}

\begin{thm}[Pre-compactness I]\label{thm-pre-comp-I}
 Let $\X$ be a class of covered \LpLSn s such that each $\XtoU\in\X$, with covering $\U=(U_k)_{k\in\N}$,  satisfies the following properties.
  \begin{enumerate}
   \item For each fixed $k\in\N$, the timelike diameter of $U_k$ is uniformly bounded; i.e., $\diam^\tau(U_k)\leq T_k$ for a constant $T_k\geq 0$ (independent of $X$).
   \item For all $k\in\N$ and $\eps>0$, there exists a (finite) constant $N=N(k,\eps)>0$ (independent of $X$) such that $U_k$ admits an $\eps$-net $S_\eps^k$ of cardinality at most $N$.
   \item For all $k\in\N$ and $\eps>0$, it holds that $S_\eps^k\subseteq S_\eps^{k+1}$.
     \end{enumerate}
Then any sequence in $\X$ has a converging subsequence; i.e., for any sequence  $\big(\XtnoU\big)_{n\in \N} \subset \X$ there exists a subsequence $(n_j)_j\subset \N$ and a covered \LpLS $\XtoU$ such that 
$$
\XtnjoU\pLGHtop \XtoU \quad  \text{strongly, as $j\to \infty$}.
$$ 
\end{thm}

\begin{pr}
Fix a sequence $\big(\XtnoU\big)_{n\in \N}\subset \X$ and write $\U_n=(U_{k,n})_{k\in\N}$ for $n\in\N$. We will inductively construct the limit space $X$ by constructing a cover $\U = (U_{k,\infty})_{k\in\N}$.
 \medskip
 
  \noindent\underline{Step 1:} Constructing a limit of $(U_{1,n})_n$.\\
 For all $m\in\N,m\geq 1$ there is a $\frac{1}{m}$-net $S_{1,n,m}$ for $U_{1,n}$, with
 $$
 S_{1,n,m} = \left\{J\big(x_{1,n,m}^1,y_{1,n,m}^1\big),\ldots, J\big(x_{1,n,m}^{N_{1,m}},y_{1,n,m}^{N_{1,m}}\big)\right\}.
 $$
 By assumption, $|S_{1,n,m}|\leq N(1,\frac{1}{m})=:N_{1,m}$.  By a diagonal argument, we can assume that, up to subsequences:
 \begin{alignat*}{2}
  &\exists\lim_{n\to\infty}\ell_n(x_{1,n,m}^i,y_{1,n,m}^j)\,, \qquad&&\exists\lim_{n\to\infty}\ell_n(y_{1,n,m}^i,x_{1,n,m}^j)\,,\\
 &\exists\lim_{n\to\infty}\ell_n(x_{1,n,m}^i,o_{n})\,, &&\exists\lim_{n\to\infty}\ell_n(y_{1,n,m}^i,o_{n})\,,\\
 &\exists\lim_{n\to\infty}\ell_n(o_{n},x_{1,n,m}^i)\,, &&\exists\lim_{n\to\infty}\ell_n(o_{n},y_{1,n,m}^i)\,,
 \end{alignat*}
 for all $1\leq i,j\leq N_{1,m}$. Thus, we define the (countable) space $U_{1,\infty}$ as
\begin{equation*}
 U_{1,\infty}:= \{x_{1,\infty,1}^1,y_{1,\infty,1}^1,\ldots, x_{1,\infty,1}^{N_{1,1}}, y_{1,\infty,1}^{N_{1,1}},x_{1,\infty,2}^1,\ldots\}  \cup \{o_\infty\}\,.
\end{equation*}

\medskip 
\noindent\underline{Step 2:} Induction step $k-1\mapsto k$.\\
We assume that we have constructed $U_{1,\infty}\subseteq\ldots\subseteq U_{k-1,\infty}$. By assumption, the $\eps$-nets for $U_{k-1,n}$ are contained in the $\eps$-nets for $U_{k,n}$ (for all $n\in\N$). By assumption, $|S_{k,n,m}|\leq N(k,\frac{1}{m})=:N_{k,m}$. Arguing as in the first step (by using a diagonal procedure twice), we preserve the convergence properties with respect to $U_{k-1,\infty}$ and overall we get that for all $1\leq i,j\leq N_{k,m}$:
 \begin{alignat*}{2}
  &\exists\lim_{n\to\infty}\ell_n(x_{1,n,m}^i,y_{1,n,m}^j)\,,\qquad &&\exists\lim_{n\to\infty}\ell_n(y_{1,n,m}^i,x_{1,n,m}^j)\,,\\
 &\exists\lim_{n\to\infty}\ell_n(x_{1,n,m}^i,o_{n})\,, &&\exists\lim_{n\to\infty}\ell_n(y_{1,n,m}^i,o_{n})\,,\\
 &\exists\lim_{n\to\infty}\ell_n(o_{n},x_{1,n,m}^i)\,, &&\exists\lim_{n\to\infty}\ell_n(o_{n},y_{1,n,m}^i)\,,
 \end{alignat*}
 where $U_{k,\infty}:= U_{k-1,\infty}\,\cup\, \{x_{k,\infty,1}^1,y_{k,\infty,1}^1,\ldots, x_{k,\infty,1}^{N_{k,1}}, y_{k,\infty,1}^{N_{k,1}},x_{k,\infty,2}^1,\ldots\}$.

\medskip 
\noindent\underline{Step 3:} Construction of the limit space.\\
We set $X_\infty : = \bigcup_{k\in\N} U_{k,\infty}$ with covering $\U_\infty = (U_{k,\infty})_\infty$. The time-separation function $\ell_\infty$ is given as
\begin{align*}
 \ell_{\infty}(x_{k,\infty,m}^i,y_{k',\infty,m'}^j)&:= \lim_{n\to\infty} \ell_n(x_{k,n,m}^i,y_{k',n,m'}^j)\,,\\
 \ell_{\infty}(y_{k',\infty,m'}^j, x_{k,\infty,m}^i,)&:= \lim_{n\to\infty} \ell_n(y_{k',n,m'}^j,x_{k,n,m}^i)\,,\\
 \ell_{\infty}(x_{k,\infty,m}^i,x_{k',\infty,m'}^j)&:= \lim_{n\to\infty} \ell_n(x_{k,n,m}^i,x_{k',n,m'}^j)\,,\\
 \ell_{\infty}(y_{k,\infty,m}^i,y_{k',\infty,m'}^j)&:= \lim_{n\to\infty} \ell_n(y_{k,n,m}^i,y_{k',n,m'}^j)\,,\\
 \ell_{\infty}(x_{k,\infty,m}^i,o_\infty)&:= \lim_{n\to\infty} \ell_n(x_{k,n,m}^i,o_n)\,,\\
  \ell_{\infty}(y_{k,\infty,m}^i,o_\infty)&:= \lim_{n\to\infty} \ell_n(y_{k,n,m}^i,o_n)\,,\\
   \ell_{\infty}(o_\infty,x_{k,\infty,m}^i)&:= \lim_{n\to\infty} \ell_n(o_n,x_{k,n,m}^i)\,,\\
  \ell_{\infty}(o_\infty,y_{k,\infty,m}^i)&:= \lim_{n\to\infty} \ell_n(o_n,y_{k,n,m}^i)\,,
\end{align*}
for all $i,j,k,k',m,m'\in\N$. From this definition it readily follows that $\ell_\infty$ satisfies the reverse triangle inequality and $\ell_\infty(x,x)\geq 0$ for all $x\in X_\infty$.

Finally, we use the chronological topology on $X_\infty$, hence $(X_\infty, \ell_\infty)$ is a \LpLSn.

\noindent\underline{Step 4:} $X_n\pLGHtop X_\infty$.\\
Convergence of the vertices is clear by construction. Since $X_\infty$ only consists of the vertices and the distinguished point $o_\infty$, the covering property, and the timelike forward density property are trivially satisfied. Moreover, the extension property holds by construction.
This finishes the proof.
\end{pr}

\section{Completion of a \LpLSn}\label{sec-compl}
In the pre-compactness theorem \ref{thm-pre-comp-I} the limit is countable and not a continuum, hence one might wish to complete it. This is particularly relevant when showing that the limit of globally hyperbolic spaces is globally hyperbolic (cf.\ Theorem \ref{thm-lim-gh}, in the appendix).
To that aim, in this section we construct a \emph{forward completion} of a \LpLSn. This builds on the notion of forward completeness for Lorentzian length spaces introduced in \cite{BBCGMORS:24} and it is inspired by the recent work of Gigli \cite{Gig:25}, who obtained existence and uniqueness of a forward completion in the general setting of partial orders. We refer the reader to  \cite{Gig:25} for more details and results on forward completing partially ordered and Lorentzian spaces. In  Appendix \ref{app}, we also discuss an independent way of obtain a completion, more tailored to the globally hyperbolic setting. 

\begin{defi}\label{defi-lpls-for-compl}
 Let $\Xt$ be a \LpLSn. A \emph{forward completion} of $X$ is a \LpLS $\Xtb$ satisfying the following properties:
 \begin{enumerate}
 \item $\Xtb$ is forward complete (cf.\ Definition \ref{defi-for-compl});
 \item  $J^\pm_{\overline{\leq}}(\bar x)\subset \overline{X}$ is closed, for all $\bar x\in\overline{X}$;
 \item There exists an isometric embedding of $\Xt$ into $\Xtb$ such that $X$ is \emph{forward dense} in $\overline{X}$, i.e., for all $\bar x\in\overline{X}$ there exists a sequence $(x_k)_k$ in $X$ such that $x_k\,\overline{\leq}\,x_{k+1}\,\overline{\leq}\,\bar x$ and $x_k\to\bar x$.
 \end{enumerate}
\end{defi}

\begin{thm}[Existence of a forward completion]\label{thm-lpls-compl-ex}
 Let $\Xt$ be a \LpLSn. Then $X$ admits a completion $\Xtb$. Moreover,  there exists at most one completion $\Xtb$ (up to isometry), such that:
 \begin{enumerate}
     \item The time-separation function $\overline{\tau}$ is continuous;
     \item The causal relation $\overline{\leq}$ is a closed partial order;
     \item For all $\bar x\in \overline{X}\backslash X$, it holds  that $\bar x \in \overline{I^\pm(\bar x)}$.
 \end{enumerate}
 \end{thm}
\begin{pr}
\textbf{Step 1}. Construction of the \LpLSn.

\noindent Define
\begin{equation*}
 Y:=\{(x_k)_k \in X^\N \text{ such that } x_k \leq x_{k+1} \leq z\ \forall k\in\N \text{ and for some } z\in X\}\,.
 \end{equation*}
We define a time-separation function $\ell'$ on $Y$ as follows. Let $(x_k)_k, (y_k)_k \in Y$, then we set
\begin{equation*}
 \ell'((x_k)_k, (y_l)_l)) := \limsup_{k,l\to\infty}\ell(x_k,y_l)\,.
\end{equation*}
The function $\ell'$ is valued in $\{-\infty\}\cup[0,\infty]$ and $\ell'((x_k)_k,(x_k)_k)\geq 0$ for all $(x_k)_k\in Y$. Indeed, $\ell(x_k,x_k)\in\{0,\infty\}$ for all $k\in\N$, thus  
\begin{align*}
    \ell'((x_k)_k, (x_k)_k)) = \limsup_{k,l\to\infty} \ell(x_k,x_l) \geq \limsup_{k\to\infty}\ell(x_k,x_k) \geq 0\,.
\end{align*}
We next show that $\ell'$ satisfies the reverse triangle inequality. \\Let $(x_k)_k, (y_l)_l, (z_r)_r\in Y$. Then, for all $k,l,r\in\N$, we have that
 $\ell(x_k, y_l) + \ell(y_l,z_r) \leq \ell(x_k,z_r)$. Thus, 
 \begin{align*}
     L:=\limsup_{k,l,r\to\infty} \Bigl( \ell(x_k,y_l) + \ell(y_l,z_r) \Bigr) \leq \ell'((x_k)_k, (z_r)_r)\,.
 \end{align*}
 We know that $L\leq \ell'((x_k)_k, (y_l)_l) + \ell'((y_l)_l, (z_r)_r)$; we claim that actually equality holds. Let us first consider the case $L=-\infty$. The monotonicity of the sequence $(y_l)_l$, implies that if $\ell(x_k, y_l)=-\infty$, then $\ell(x_k, y_{l'})=-\infty$ for all $l'\leq l$. Thus, $L=-\infty$ implies that $\ell'((x_k)_k, (y_l)_l)) + \ell'((y_l)_l, (z_r)_r))=-\infty$ and the reverse triangle inequality is trivially satisfied. To prove the claim for $L\in \R$, assume by contradiction that there is an $\eps>0$ such that $L+\eps < \ell'((x_k)_k,(y_l)_l) + \ell'((y_l)_l, (z_r)_r)$. There are $k_0, l_0, r_0\in\N$ such that for all $k\geq k_0, l\geq l_0,r\geq r_0$ we have $\ell(x_k,y_l) + \ell(y_l,z_r) \leq L+\frac{\eps}{2}$. Moreover, there are subsequences $(k_n)_n, (l_{n'})_{n'}, (l_{n''})_{n''}, (r_{n'''})_{n'''}$ with
 \begin{align*}
     \ell(x_{k_n}, y_{l_{n'}}) &> \ell'((x_k)_k, (y_l)_l)  - \frac{\eps}{4}\,,\\
     \ell(y_{l_{n''}}, z_{r_{n'''}}) &> \ell'((y_l)_l, (z_r)_r)) - \frac{\eps}{4}\,.
 \end{align*}
 Take $n,n',n'',n'''$ large enough such that $k_n\geq k_0, l_{n'}\geq l_0, l_{n''}\geq l_{n'}, r_{n'''}\geq r_0$. Then   $\ell(x_{k_n},y_{l_{n'}}) \leq \ell(x_{k_n},y_{l_{n''}})$ and we can estimate
 \begin{align*}
 L + \eps &< \ell'((x_k)_k, (y_l)_l)) + \ell'((y_l)_l, (z_r)_r))\\
 &\leq \ell(x_{k_n}, y_{l_{n'}}) + \ell(y_{l_{n''}},z_{r_{n'''}}) + \frac{\eps}{2}\\
 &\leq \ell(x_{k_n}, y_{l_{n''}}) + \ell(y_{l_{n''}},z_{r_{n'''}}) + \frac{\eps}{2}\\
& \leq L +\eps\,,
\end{align*}
which is a contradiction since $L\in\R$.
\smallskip

\textbf{Step 2}. Construction of an equivalence relation $\sim$ on  $Y$.
\\Define an equivalence relation $\sim$ on $Y$ as follows: $(x_k)_k \sim (y_k)_k$ if there is a $z\in X$ and subsequences $(x_{k_l})_l$ and $(y_{k_l})_l$ of $(x_k)_k$ and $(y_k)_k$, respectively such that
\begin{equation*}
 x_{k_l} \leq y_{k_l} \leq x_{k_{l+1}} \leq y_{k_{l+1}}\leq z\ \forall l\in\N\,.
\end{equation*}
In particular, the sequence $(z_l)_l$ given by
\begin{align*}
 z_l:=\begin{cases}
       &x_{k_l}\qquad l\equiv 0 \text{ mod } 2\,,\\
       &y_{k_l}\qquad l\equiv 1 \text{ mod } 2\,,
      \end{cases}
\end{align*}
is in $Y$. Clearly, $\sim$ is reflexive and symmetric by shifting the index of the sequence by $\pm 1$. To show transitivity, let $(x_k)_k, (y_k)_k, (z_k)_k\in Y$ with bounds $x,y,z$, respectively. If $(x_k)_k\sim (y_k)_k$ and $(y_k)_k\sim (z_k)_k$, let $(x_{k_l})_l$, $(y_{k_l})_l$, $(y_{n_l})_l$, $(z_{k_l})_l$ be the corresponding subsequences such that, for all $l\in\N$:
\begin{equation*}
 x_{k_l}\leq y_{k_l} \leq x_{k_{l+1}} \leq y_{k_{l+1}}\,, \quad y_{n_l} \leq z_{k_l} \leq y_{n_{l+1}} \leq z_{k_{l+1}} \leq z\,.
\end{equation*}
Then, for $l=0$ we can find $n'_0$ such that $n'_0\geq k_0=:k'_0$. Hence:
\begin{equation*}
 x_{k'_0} \leq y_{k'_0} \leq y_{n'_0} \leq z_{n'_0} \leq y_{n'_0+1} \leq x_{n'_0+2}\,.
\end{equation*}
Setting $k'_1:=n_0'+2$, we obtain the desired subsequences iteratively;  since all $z_k$s are bounded by $z$, we conclude.
\smallskip

\textbf{Step 3}. Endowing the quotient $\overline{X}:=Y/\sim$ with the quotient time-separation $\bar{\ell}$.
\\Define $\overline{X}:=Y/\sim$ and 
\begin{equation*}
 \bar\ell([x], [y]) := \ell'(x,y)\,,
\end{equation*}
where $[x]$ is the equivalence class of $x\in Y$. First, we show that the time-separation function $\bar{\ell}$ is well-defined. It suffices to fix one slot, so let $(x'_k)_k\sim (x_k)_k$ in $Y$ and without loss of generality we can assume that $\ell(x_k,y_m)\to \ell'((x_k)_k, (y_m)_m)$ and $\ell(x_k',y_{m'})\to \ell'((x_k')_k,(y_m)_m)$ (otherwise take subsequences). By assumption, there are subsequences $(x_{k_l})_l, (x'_{k_l})_l$ of $(x_k)_k$ and $(x'_k)_k$, respectively, such that
\begin{equation*}
 x_{k_l} \leq x'_{k_l} \leq x_{k_{l+1}} \leq x'_{k_{l+1}}\ \forall l\in\N\,.
\end{equation*}
By the reverse triangle inequality, we infer that
\begin{align*}
 \ell(x'_{k_l},y_{m'}) &\leq \ell(x_{k_l},x'_{k_l}) + \ell(x'_{k_l},y_{m'}) \leq \ell(x_{k_l},y_{m'})\text{ and }\\
 \ell(x_{k_{l+1}},y_m) &\leq \ell(x'_{k_{l}},x_{k_{l+1}}) + \ell(x_{k_{l+1}},y_m) \leq \ell(x'_{k_{l}},y_m)\,.
\end{align*}
Now taking the limit superior  as $l,m'\to\infty$ in the first inequality we get
\begin{align*}
\ell'((x_k')_k,(y_m)_m) \leq \limsup_{l,m'\to\infty} \ell(x_{k_l},y_{m'}) \leq \ell'((x_k)_k,(y_m)_m))\,,
\end{align*}
and 
taking the limit superior  as $l,m\to\infty$ in the second one we get
\begin{align*}
\ell'((x_k)_k,(y_m)_m) \leq \limsup_{l,m\to\infty} \ell(x_{k_l}',y_m) \leq \ell'((x_k')_k,(y_m)_m))\,.
\end{align*}
In summary, this gives $\ell'((x_k)_k, (y_m)_m) = \ell'((x_k')_k, (y_m)_m)$.
\smallskip

\textbf{Step 4}.   $\overline{X}$ is forward complete.
\\Now, we topologize $\overline{X}$ with the chronocausal topology (i.e., the topology generated by the subbase $I^\pm(y)$, $X\backslash J^\pm(y)$ for $y\in Y$)  and show that $\Xtb$ is forward complete. It suffices to work in $(Y,\ell')$: Let $y^m \leq y^{m+1} \leq z$ in $Y$ for all $m\in\N$, where $y^m=(y^m_k)_k$, $z=(z_k)_k$ with $y^m_k \leq y^m_{k+1}\leq z^m$ and $z_k\leq z_{k+1}\leq \hat{z}$ for some $z^m,\hat{z}\in X$ for all $k,m\in\N$. The assumption $y^m\leq y^{m+1}\leq z$ implies that there are subsequences such that $\ell(y^m_{k_l},y^{m+1}_{k_l'})\to \ell'(y^m,y^{m+1}) \geq 0$ and $\ell(y^{m+1}_{\tilde k_l'},z_{k_l''})\to \ell'(y^{m+1},z) \geq 0$. In particular, since $\ell$ takes values in $\{-\infty\}\cup[0,+\infty)$, we get that for all $m\in\N$ there is $l_0^m\in\N$ such that for all $l\geq l_0^m$ we have $\ell(y^m_{k_l},y^{m+1}_{k'_l})\geq 0$ and there is $\tilde l_0^m\geq l_0^m$ such that for all $l\geq \tilde l^m_0$ we have $\ell(y^{m+1}_{\tilde k'_l},z_{k''_l})\geq 0$. By always choosing $\tilde k_l'\geq k_l'$ we can without loss of generality assume that $\tilde k_l'=k_l'$. This yields that
$$
y^m_{k_l}\leq y^{m+1}_{k'_l}\leq z_{k''_l}, \quad \text{for all }\; m\in\N,\, l\geq \tilde l_0^m.
$$
Moreover, we can choose $(l^m_0)_m$ and $(\tilde l^m_0)_m$
 to be non-decreasing in $m$.
 
We define the limit $y=(y_k)_k$ inductively:\\
First, we set $y_0:= y^0_{k_{\tilde l^0_0}}\leq z_{k_{\tilde l_0^0}''}\leq \hat{z}$. Then, we assume that we already have defined $y_r$ for $r=0,\ldots,n$ such that \begin{enumerate}
\item $y_r\leq y_{r+1}\leq \hat{z}$ for $r=0,\ldots,n-1$, and
\item $y_r = y^r_{k_{l(r)}}$ for some increasing function $l\colon\{0,\ldots,n\}\rightarrow \N$.
\end{enumerate}
Let $l(n+1)\geq  \max(l(n),\tilde l_0^{n+1})$ such that $k_{l(n)}'\leq k_{l(n+1)}$. Then, we set
\begin{equation*}
 y_{n+1}:=y^{n+1}_{k_{l(n+1)}}\leq z_{k_{l(n+1)}''} \leq \hat{z}\,,
\end{equation*}
hence we get that
\begin{equation*}
 y_n = y^n_{k_{l(n)}} \leq y^{n+1}_{k_{l(n)}'} \leq y^{n+1}_{k_{l(n+1)}} = y_{n+1}\,.
\end{equation*}
Consequently, $y:=(y_n)_n\in Y$. 

Now we show that $y^m\to y$ in $Y$ with respect to the chronological topology. Let $w\in I^-(y)$, i.e., 
 $\ell'(w,y)>0$ and $w=(w_k)_k$ with $w_k\leq w_{k+1}\leq \hat{w}$ for some $\hat{w}\in X$. There are subsequences such that
 \begin{align*}
 0<\eps:=\ell'(w,y)=\limsup_{k,r\to\infty} \ell(w_k,y_r) = \lim_{t,n\to\infty}\ell(w_{k_t},y_{r_n})\,.
 \end{align*}
 In particular, there are $t_0,n_0\in\N$ such that for all $t\geq t_0, n\geq n_0$ we have $\ell(w_{k_t},y_{r_n})\geq \frac{\eps}{2}>0$. Let $s\geq k_{l(r_{n_0})}$, then
 \begin{align*}
 \ell(w_{k_t},y^{r_{n_0}}_s) \geq \ell(w_{k_t},y^{r_{n_0}}_{k_{l(r_{n_0})}}) =  \ell(w_{k_t},y_{r_{n_0}})\geq \frac{\eps}{2}\,.
 \end{align*}
Consequently, we get
\begin{align*}
    \ell'(w, y^{r_{n_0}}) = \limsup_{k,s\to\infty}\ell(w_k, y^{r_{n_0}}_s) \geq \limsup_{t,s\to\infty}\ell(w_{k_t},y^{r_{n_0}}_s)\geq \frac{\eps}{2}>0\,,
\end{align*}
i.e., $w\ll y^{r_{n_0}}$. Hence for all $r'\geq r_{n_0}$ we have $w\ll y^{r_{n_0}} \leq y^{r'}$, i.e., $y^{r'}\in I^+(w)$ for all $r'\geq r_{n_0}$ as required.
 
Similarly, let $w\in I^+(y)$, then there are subsequences such that
 \begin{align*}
 0<\eps:=\ell'(y,w)=\limsup_{k,r\to\infty} \ell(y_r,w_k) = \lim_{t,n\to\infty}\ell(y_{r_n},w_{k_t})\,.
 \end{align*}
 In particular, there are $t_0,n_0\in\N$ such that, for all $t\geq t_0, n\geq n_0$, we have $\ell(y_{r_n}, w_{k_t})\geq \frac{\eps}{2}>0$. Then for $t\geq t_0$, $n\geq n_0$ and $s\geq k_{l(r_{n})}$ we get that
 \begin{align*}
 y^{r_n}_s \leq y^{r_n +1}_{k_{l(r_n)}'} \leq y^{r_n +1}_{k_{l(r_n+1)}} = y_{r_n + 1} \leq \ldots \leq y_{r_{n+1}} \ll w_{k_t}\,.
 \end{align*}
Thus we obtain $\ell(y^{r_n}_s,w_{k_t}) \geq \ell(y_{r_{n+1}},w_{k_t}) \geq \frac{\eps}{2}$ and so
\begin{align*}
    \ell'(y^{r_n},w) \geq \limsup_{s,t\to\infty}\ell(y^{r_n}_s, w_{k_t}) \geq \frac{\eps}{2}>0\,.
\end{align*}
This means that $y^{r_n}\in I^-(w)$ for all $n\geq n_0$. Consequently, for $r'\in\N$ we have $r'\leq r_n$ for some $n\in\N$ with $n\geq n_0$, so $y^{r'}\leq y^{r_n}\ll w$, so $y^{r'}\in I^-(w)$ as required. This concludes the proof of $y^m\to y$ in the chronological topology.

It remains to show that $y^m\to y$ with respect to the causal topology. Let $y\not\in J^+(w)$, i.e., 
\begin{equation*}
\lim_{k,n\to\infty} \ell(w_k,y_n) = -\infty\,.
\end{equation*}
Thus, there are $k_0,n_0\in\N$ such that, for all $k\geq k_0$ and $n\geq n_0$, we have $\ell(w_k,y_n)=-\infty$. Let $n\geq n_0$ and $r\geq k_{l(n)}$, then $y^n_r\leq y^{n+1}_{k_{l(n)}'} \leq y^{n+1}_{k_{l(n+1)}} = y_{n+1}$. So for all $k\geq k_0, n\geq n_0$ and $r\geq k_{l(n)}$ we have $\ell(w_k,y^m_r)\leq \ell(w_k,y_{n+1})=-\infty$, hence $\ell'(w,y^{n+1})=-\infty$ and $y^{n+1}\not\in J^+(w)$ for $n\geq n_0$.

Now assume that $y\not\in J^-(w)$, then analogously as in the previous case we have
\begin{equation*}
\lim_{k,n\to\infty} \ell(y_n,w_k) = -\infty\,,
\end{equation*}
and so $\ell(y_n,w_k)=-\infty$ for all $n\geq n_0, k\geq k_0$ for some $n_0,k_0\in\N$.  Assume by contradiction that there is a sequence $m_r\nearrow \infty$ such that $y^{m_r}\in J^-(w)$ for all $r\in\N$. Thus, there are subsequences such that for all $r\in\N$ there are $t_0(r),s_0(r)\in\N$ with $\ell(y^{m_r}_{k_t},w_{k_s})\geq 0$ for all $t\geq t_0(r), s\geq s_0(r)$. Then, let $r\in \N$ and $s\geq s_0(r)$. If $t_0(r)\leq l(m_r)$ then by the above
\begin{equation*}
 0\leq \ell(y^{m_r}_{k_{l(m_r)}},w_{k_s}) = -\infty\,,
\end{equation*}
a contradiction. If, on the other hand, $t_0(r)> l(m_r)$ we have
\begin{equation*}
 0 \leq \ell(y^{m_r}_{k_{t_0(r)}},w_{k_s}) \leq \ell(y^{m_r}_{k_{l(m_r)}},w_{k_s}) = -\infty\,,
\end{equation*}
again a contradiction. This concludes the proof of the forward completeness of $\Xtb$.
\smallskip

\textbf{Step 5}.  Forward dense isometric embedding of $X$ into $\overline{X}$.
\\ We can embed $X$ into $\overline{X}$ as constant sequences and show that these are forward dense in $\Xtb$. Let 
\begin{equation}\label{eq:defy=(yk)k}
y=(y_k)_k\in Y\,.
\end{equation}
Define the  sequence $(y^m)_m$ in $Y$ by setting 
\begin{equation}\label{eq:defSeqym}
y^m:=(y_m)_k\,.
\end{equation}
In other words, for every $m\in \N$, the element $y^m\in Y$ is given by the constant sequence $(y_m)_k$, where $y_m$ is the $m^{\text {th}}$ term in the sequence defining $y\in Y$ as in \eqref{eq:defy=(yk)k}.

First, we show that $y^m\leq y^{m+1}\leq y$, for all $m\in\N$.  Indeed, by construction, $\ell'(y^m,y^{m+1}) = \ell(y_m,y_{m+1})\geq 0$ and $\ell'(y^m, y) = \limsup_{k\to\infty}\ell(y_m, y_k)\geq 0$, as eventually $k\geq m$.

We now prove that $y^m\to y$, as $m\to\infty$. Let $w\in I^-(y)$, $w=(w_k)_k\in Y$.
Then there are subsequences such that $0<\eps:=\ell'(w,y) = \lim_{s,t\to\infty}\ell(w_{k_t}, y_{m_s})$. Thus, there are $s_0,t_0\in\N$ such that, for all $s\geq s_0, t\geq t_0$, we have $\ell(w_{k_t}, y_{m_s})\geq \frac{\eps}{2}>0$. Then, for $s\geq s_0$, it holds that
\begin{align*}
\ell'(w,y^{m_s}) = \limsup_{k\to\infty} \ell(w_k,y_{m_s}) \geq \limsup_{t\to\infty}\ell(w_{k_t},y_{m_s})\geq \frac{\eps}{2}>0\,,
\end{align*}
i.e., $y^{m_s}\in I^+(w)$ for all $s\geq s_0$. Consequently, for all $m'\in\N$ with $m_{s_0}\leq m'$ we have $w\ll y^{m_{s_0}}\leq y^{m'}$, and so $y^{m'}\in I^+(w)$ as required.

Now let $w\in I^+(y)$. Analogously to the previous case, there are subsequences such that $0<\eps:=\ell'(y,w) = \lim_{s,t\to\infty}\ell(y_{m_s},w_{k_t})$. Thus, there exist $s_0,t_0\in\N$ such that $\ell(y_{m_s},w_{k_t})\geq \frac{\eps}{2}>0$, for all $s\geq s_0, t\geq t_0$. Then, for $s\geq s_0$, we get
\begin{align*}
\ell'(y^{m_s},w) = \limsup_{k\to\infty} \ell(y_{m_s},w_k) \geq \limsup_{t\to\infty}\ell(y_{m_s},w_{k_t})\geq \frac{\eps}{2}>0\,,
\end{align*}
i.e., $y^{m_s}\in I^-(w)$ for all $s\geq s_0$. Consequently, for all $m'\in\N$ there is an $s\geq s_0$ such that $m'\leq m_s$, hence  $y^{m'}\leq y^{m_s}\ll w$, and so $y^{m'}\in I^-(w)$ as required.

It remains to check convergence for the open sets $X\backslash J^\pm(w)$. First, assume that $y\not\in J^+(w)$, then for $k,n\in\N$ large we have $\ell(w_n,y_k)=-\infty$. This immediately implies that $-\infty = \ell(w_n,y_k) = \ell(w_n, (y^k)_k)$, for all large $n,k\in\N$ hence $y^k\not\in J^+(w)$ for all large $k\in\N$. Similar, if $y\not\in J^-(w)$, then $-\infty = \ell(y_k,w_n) = \ell((y^k)_k,w_n)$ for $k,n\in\N$ large, hence $y^k\not\in J^-(w)$ for $k\in\N$ large. This concludes the proof that $y^m\to y$ in $Y$.
\smallskip

\textbf{Step 6}. Uniqueness.\\ 
Assume there exist two completions $\Xtb$ and $(\overline{X}',\overline{\ell}')$ such that the assertions (i)--(iii) hold. We construct a map $f\colon \overline{X}\rightarrow\overline{X}'$ as follows. 

On $X$, $f$ is the identity between the inclusions $X\subset \overline{X}$ and  $X\subset \overline{X}'$.  
\\Next we define $f$ on $\overline{X}\setminus X$.  For $\bar x\in \overline{X}\backslash X$,  let $x_k\to \bar x$ with $x_k\in X$ and $x_k\,\overline{\leq}\,x_{k+1}\,\overline{\leq}\, \bar x$, for all $k\in\N$. By item (iii) we have $I^+_{\overline{\ll}}(\bar x)\neq\emptyset$; so let $\bar w \in I^+_{\overline{\ll}}(\bar x)$ and let $\bar w = \lim w_l$, where $w_l\in X$, $w_l\,\overline{\leq}\,w_{l+1}\,\overline{\leq}\, \bar w$. Then, there exists $l\in\N$ such that $\bar x\,\overline{\ll}\, w_l$. Consequently, for all $k\in\N$, we have $x_k\,\overline{\leq}\,\bar x\,\overline{\ll}\, w_{l}$. Then, in $\overline{X}'$ we also have $x_k\,\overline{\leq}'\,x_{k+1}\,\overline{\leq}'\, w_{l}$. Thus, by forward completeness of $\overline{X}'$,  the  sequence $(x_k)_k\subset \overline{X}'$ converges to some point $\bar x'\in \overline{X}'$. Set $f(\bar x):=\bar x'$.

We next show that $f$ is well-defined. Let $y_k\,\leq\, y_{k+1}\,\overline{\leq}\, \bar x$ with $y_k\to \bar x$ in $\overline{X}$ be another sequence of elements in $X$ converging to $\bar x$ in $\overline{X}$ and to $\bar y'$ in $\overline{X}'$. Let $\bar w\in I^+_{\overline{\ll}'}(\bar y')$ with $w_l\to \bar w$, $w_l\,\overline{\leq}'\, \bar w$, $w_l\in X$ for all $l\in\N$. Therefore, there exists $l_0\in\N$ such that, for all $l\geq l_0$, we have $\bar y'\,\overline{\ll}'\, w_l$. Since $x_k\to\bar x$ in $\overline{X}$, for every $l\geq l_0$ there exists $k_l\in\N$ such that for all $k\geq k_l$ we have $y_k\,\overline{\leq}'\,\bar y'\,\overline{\ll}' w_l$, i.e., $y_k\,\ll\, w_l$. Taking the limit as $k\to\infty$ in $\overline{X}'$ and using the closedness of $\overline{\leq}'$, we get $f(\bar x) = \bar x'\,\overline{\leq}'\, w_l$;  taking the limit as $l\to\infty$ yields $f(\bar x)\,\overline{\leq}'\,\bar w$. Finally, letting $\bar w\to \bar y'$ thanks to item (iii), we obtain that $f(\bar x)\,\overline{\leq}'\,\bar y'$. By swapping the roles of the sequences $(x_k)_k$ and $(y_k)_k$, yields the reverse inequality $\bar{y}' \overline{\leq}' \bar{x}'$. Since by item $(ii)$ the causal relation $\overline{\leq}'$ is a partial order, we conclude that $\bar{x}'= \bar{y}'$, proving that $f$ is well defined.

Since $X$ embeds isometrically into both $\overline{X}$ and $\overline{X}'$, we can identify them via the map $f$,  and suppress $f$ in the rest of the argument. Since $X$ is dense and the time-separations are continuous, we get that $\overline{\tau}=\overline{\tau}'$. It remains to show $\overline{\ell}=\overline{\ell}'$ or, equivalently, that the causal relations coincide. Assume by contradiction that there exist $\bar x,\bar y \in\overline{X}$ with
\begin{equation*}
    \overline{\ell}(\bar x,\bar y) = -\infty < 0 = \overline{\ell}'(\bar x,\bar y) = \overline{\tau}'(\bar x,\bar y)\,.
\end{equation*}
Then, assertion (iii) implies that there exists a sequence $x_k\in X$ with $x_k\,\overline{\ll}'\,\bar x$ for all $k\in\N$. Let $y_l\to \bar y$, where $y_l\in X$. Since $x_k\,\overline{\ll}'\,\bar x\, \overline{\leq}'\,\bar y$ for all $k\in\N$, there is an $l_k\in\N$ such that
$$
x_k\ll y_l, \quad \text{ for all } l\geq l_k.
$$
 Taking the limit as $l\to\infty$ and using closedness of $\overline{\leq}$, we get $x_k\,\overline{\leq}\, y$.  Then, taking the limit as $k\to\infty$ gives the contradiction $\bar x\,\overline{\leq}\,\bar y$. Thus $\overline{\ell}= \overline{\ell}'$ and the proof is complete.
\end{pr}

Finally, we establish that taking the completion does not affect convergence (under mild conditions that will always be satisfied in the applications below).  Let $\XtnoU \pLGHtop \XtoU$. We say that $X$ \emph{consists only of vertices}, if for all covering sets $U_{k,n}\in\U_n$ and $U_{k}\in\U$, there exist $\frac{1}{l}$-nets $S^l$ for $U_k$ and corresponding $\frac{1}{l}$-nets for $U_{k,n}$ together with correspondences satisfying points (i) -- (iii) of Definition \ref{def-con-subs} such that
\begin{equation}\label{eq-onl-ver}
\{o\}\cup \bigcup_{l=1}^\infty V(S^l) = U_k\,.
\end{equation}
For example, this is the case in the pre-compactness Theorem \ref{thm-pre-comp-I}.
\begin{thm}[Forward completion preserves convergence]\label{thm-lim-compl}
\ \\ Let $\XtnoU \pLGHtop \XtoU$ strongly and let $\Xtb$ be a completion of $\Xt$. Assume that
\begin{itemize}
\item $\overline{X}$ is first countable, or \item $X$ consists only of vertices, i.e., \eqref{eq-onl-ver} holds.
\end{itemize}
 Then $\XtnoU \pLGHtop (\overline{X},\overline{\ell},o,\overline{\U})$, where $\U = (U_k)_{k\in\N}$, $\overline{\U}=(\overline{U}_k)_{k\in\N}$, and $\overline{U}_k$ is the closure of $U_k$ in $\overline{X}$.
\end{thm}
\begin{pr}
  Fix $k\in\N$ and let $\eps>0$, then there exists a finite $\eps$-net for $U_k$, which we denote by $(J(p_i,q_i))_{i=1}^{N(k,\eps)}$. By Theorem \ref{thm-lpls-compl-ex} we have $$\overline{J_{\overline{\leq}}(p_i,q_i)} = J_{\overline{\leq}}(p_i, q_i),$$ for $i=1,\ldots,{N(k,\eps)}$. Consequently, $(J_{\overline{\leq}}(p_i,q_i))_{i=1}^{N(k,\eps)}$ is an $\eps$-net for $\overline{U}_k$: As $U_k\subseteq \bigcup_{i=1}^{N(k,\eps)} J(p_i,q_i)$ we get that
 \begin{align*}
  \overline{U}_k \subseteq \overline{\bigcup_{i=1}^{N(k,\eps)} J(p_i,q_i)} = \bigcup_{i=1}^{N(k,\eps)} J_{\overline{\leq}}(p_i,q_i)\,.
 \end{align*}
Moreover, $\overline{\ell}(p_i,q_i) = \ell(p_i,q_i)$ for all $i=1,\ldots,{N(k,\eps)}$ and the time-separations of the vertices still converge, as they are unchanged. Thus,  it only remains to show the extension property for correspondences and forward density. In fact, the extension property for correspondences is preserved as we do not need to change the set of vertices by the above.

To show forward density, let $y\in \overline{U}_k\backslash\V$, where $\V$ is a collection of vertices of $\frac{1}{l}$-nets for $U_k$ that is timelike forward dense in $U_k$. As $X$ is forward dense in $\overline{X}$ we have $y=\lim_{m\to\infty} y^m$, where $y^m \in X$ with $y^m\leq y^{m+1}\leq y$ for all $m\in\N$. If $y^m\in\V$ for infinitely many $m\in\N$, then we found the desired approximating sequence  (if $X$ consists only of vertices this is the case anyways). If not, by the strong convergence, there are sequences $(s^m_k)_k$ in $\V$ such that $s^m_k \ll s^m_{k+1}\ll y^m$ for all $m,k\in\N$ and $\lim_{k\to\infty} s^m_k = y^m$.
Since $\lim_{k\to\infty} s^1_k = y^1$ and $s^0_k \ll y^0 \leq y^1$, there is a $k_1\in\N$ such that for all $k\geq k_1$ we have $s_0^0\ll s^1_k$. Continuing iteratively, we obtain a sequence $(s^m_{k_m})$ with $s^{m}_{k_m}\ll s^{m+1}_{k_{m+1}}\ll y^{m+1}$ for all $m\in\N, m\geq 1$. Then, since $\overline{X}$ is first countable taking the limit $m\to\infty$ yields $s^m_{k_m}\to y$ as required.
\end{pr}

\begin{rem}
Loosing strong convergence in Theorem \ref{thm-lim-compl} above is not a serious issue, as the  main purpose of strong convergence is to establishing uniqueness of limits (see Theorem \ref{thm-uni-lim}); however, the uniqueness of a completion already follows from  Theorem \ref{thm-lpls-compl-ex}, thus establishing the uniqueness of limits also in this case.
\end{rem}

\begin{cor}[Converging sequence has forward complete limit that satisfies PDP]\label{cor-lim-cau-sim}
 Let $\XtnoU \pLGHtop \XtoU$ strongly.
 Assume that
\begin{itemize}
\item $X$ consists only of vertices, i.e., \eqref{eq-onl-ver} holds, or 
\item there exists  a first countable forward completion of $(X,\ell)$.
\end{itemize}
 Then there exists a causal, forward complete, limit \LpLS $\Xtb$ that satisfies the point distinction property \eqref{eq-pdp} and  each $J^\pm(\bar x)$ is closed in $\overline{X}$ for all $\bar x\in \overline{X}$.
\end{cor}
\begin{pr}
 Let $\Xtb$ be a completion of $\Xt$, which exists by Theorem \ref{thm-lpls-compl-ex}, then it is forward complete and  each $J^\pm(\bar x)$ is closed in $\overline{X}$ for all $\bar x\in \overline{X}$. By Theorem \ref{thm-lim-compl}, we get $\XtnoU\pLGHtop(\overline{X},\overline{\ell},o,\overline{\U})$. Taking the time-separation quotient of $\Xtb$, we obtain a \LpLS $(\overline{X}',\overline{\ell}')$ that satisfies the point distinction property \eqref{eq-pdp} and is still a limit by Theorem \ref{thm-quo-lim}. Moreover, the causal futures and pasts  $J^\pm(\bar x')$ are closed in $\overline{X}'$ for all $\bar x'\in \overline{X}'$ by continuity of the quotient map $\pi$. Causality of $(\overline{X}',\overline{\ell}')$ follows from Lemma \ref{lem-pdp-cau}.
\end{pr}

We can now apply the above to show that any globally hyperbolic spacetime arises as a pointed Lorentzian Gromov--Hausdorff limit of finite spaces. This might have implications to approaches to Quantum Gravity like causal set theory \cite{BLMS:87} (cf.\ \cite{Sur:19} for a recent review), see also Subsection \ref{subsec-haupt}. An analogous local result was proved in the context of bounded Lorentzian metric spaces \cite[Cor.\ 4.32]{MS:24}, which implies that any causal diamond in a smooth globally hyperbolic spacetime can be approximated by finite bounded Lorentzian metric spaces (called \emph{causets}).

\begin{thm}\label{thm-dis-lim}
 Each smooth globally hyperbolic spacetime is the strong pointed Lorentzian Gromov--Hausdorff limit of finite (discrete) \LpLSn s. In fact, the same holds true for any globally hyperbolic spacetime with continuous Lorentzian metric, provided $I^+$ is an open relation.
\end{thm}
\begin{pr}
Fix some $o\in M$ and let $\U=(U_k)_{k\in\N}$ be a cover given by Lemma \ref{lem-cov-st}. Then each $U_k$ is a globally hyperbolic spacetime itself (as it is open and causally convex). Moreover, as the Alexandrov topology is the manifold topology (by strong causality) it is separable. Denote by $D$ a countable dense subset of $M$, where without loss of generality we can assume that $o\in D$.
Next we show that, for all $\eps>0$ and $p\in M$ there are $d,d'\in D$ such that $p\in I(d,d')$ and $\tau(d,d')\leq \eps$. To see these two claims, let $p\in M$ and let $O\subseteq M$ be an open neighborhood of $p$ such that $\tau\leq \eps$ on $O\times O$. By strong causality, there are $p^\pm \in M$ such that $p\in I(p^-,p^+)\subseteq O$. As $I(p^-,p)\neq \emptyset \neq I(p,p^+)$ there are $d\in I(p^-,p)\cap D$, $d'\in I(p,p^+)\cap D$, hence $p\in I(d,d')$. Consequently, for all $\eps>0$,   we have
 \begin{align*}
  U_k \subseteq \overline{U}_k \subseteq \bigcup_{i=0}^{N_k^\eps} I(p_i^\eps, q_i^\eps) \subseteq \bigcup_{i=0}^{N_k^\eps} J(p_i^\eps, q_i^\eps)\,,
 \end{align*}
 with $\tau(p_i^\eps,q_i^\eps) \leq \eps$ and $p_i^\eps,q_i^\eps\in D$ for all $i$, i.e., $$S_k^\eps:=\{J(p_i^\eps, q_i^\eps): i=0,\ldots, N_k^\eps\}$$ is a finite $\eps$-net for $U_k$. 
 We want to make sure that we use all elements of $D$ as vertices, so we consider
 $$\tilde S_k^\eps:=\{J(d_i, d_j): \tau(d_i,d_j)\leq \eps;\, i,j\leq k\}.$$ Then, $S_k^\eps\cup\tilde S_k^\eps$ is still a finite $\eps$-net for $U_k$. Now, setting
 \begin{equation*}
  \tilde U_{k,n}:=\bigcup_{m=1}^n V(S_k^{\frac{1}{m}})\cup V(\tilde S_k^{\frac{1}{m}})\,,
 \end{equation*}
 we have that $(\tilde U_{k,n},\ell\rvert_{\tilde U_{k,n}\times \tilde U_{k,n}})\LGHtop (U_k,\ell\rvert_{U_k\times U_k})$ strongly as $n\to\infty$ since the vertices and the time-separations of the vertices are unchanged. At this point we make the $\tilde U_{k,n}$s increasing by setting $U_{k,n}:=\bigcup_{l=0}^k \tilde U_{l,n}$. Then setting 
 $$X_n:= U_{n,n} \quad \text{and}\quad  X:=\bigcup_{n,k\in\N} U_{k,n},$$ gives 
 $$(X_n,\ell\rvert_{X_n\times X_n},o,(U_{k,n})_{k\in\N})\pLGHtop (X,\ell\rvert_{X\times X},o,(U_k)_k)\ \text{ strongly.}$$
 
 Finally, we show that $X$ is dense in $M$. To this end it suffices to show that $D\subseteq \bigcup_{k,n\geq 1} V(\tilde S_k^{\frac{1}{n}})$. Let $d=d_i\in D$, then since $I^+(d)$ is open, non-empty and $\tau(d,d)=0$ there is a $d'=d_j\in I^+(d)$ with $\tau(d,d')\leq \frac{1}{n}$ for some $n\in\N,n\geq 1$. Then $J(d,d')\in \tilde S_{\max(i,j)}^{\frac{1}{n}}$ and so $d\in V(\tilde S_{\max(i,j)}^{\frac{1}{n}})$. By Remark \ref{rem-gh-for-compl}, $(M,\ell_g)$ is a first countable forward completion of $(X,\ell\rvert_{X\times X})$, hence by Theorem \ref{thm-lim-compl} we have $(X_n,\ell\rvert_{X_n\times X_n},o,(U_{k,n})_{k\in\N})\pLGHtop (M,\ell_g,o,\U)$. In fact, by construction the vertices are timelike forward dense, yielding strong convergence.
\end{pr}

\section{Geometric pre-compactness}\label{sec-geo-pre-comp}

\subsection{A pre-compactness result under causal assumptions}
The goal of this section is to apply the Lorentzian Gromov--Hausdorff convergence  developed in the paper, in particular the pre-compactness theorem, to smooth spacetimes. An interesting feature of the geometric pre-compactness result proved below is that it heavily relies on the causal structure, and it has no immediate counterpart in Riemannian signature.
\medskip

We start with the case of a Lorentzian product of the real line with a Riemannian manifold. 
Let $(\Sigma, h)$ be a compact Riemannian manifold; in particular, $(\Sigma, h)$ is  \emph{totally bounded}. Here, total boundedness means that for all $\eps>0$ there exists a finite (metric) $\eps$-net $S_\eps\subseteq \Sigma$; i.e., $\Sigma = \bigcup_{s\in S_\eps} B_\eps^{{\mathsf d}^h}(s)$ and $|S_\eps|\leq N(\eps)<\infty$, where ${\mathsf d}^h$ is the metric induced by $h$. Now we consider the product spacetime $M:=\R\times \Sigma$ with Lorentzian metric $-{\rm d}t^2 + h$, which is globally hyperbolic since $\Sigma$ is complete. Set $\Sigma_t:=\{t\}\times\Sigma$ for $t\in\R$. 

\begin{lem}[Cover by causal diamonds in the future of $\Sigma_0$]
 Let $A\subseteq J^+(\Sigma_{t_0})$ for some $t_0>0$. Then for all $0<\eps \leq t_0$, for all (metric) $\eps$-nets $S$ in $\Sigma$ we have $A\subseteq I^+(\{0\}\times S)$.
\end{lem}
\begin{pr}
 Let $a=(t,\sigma)\in A$, then $t\geq t_0$. Let $0<\eps\leq t_0$ and $S$ be an $\eps$-net in $\Sigma$. Thus there is an $s\in S$ such that ${\mathsf d}^h(\sigma,s)<\eps$. Let $\gamma\colon[0,1]\rightarrow M$, $\gamma(r):=(tr, \vec{\gamma}(r))$, where $\vec{\gamma}\colon[0,1]\rightarrow \Sigma$ is a minimizing geodesic from $\vec{\gamma}(0)=s$ to $\vec{\gamma}(1)=\sigma$ with $|\dot{\vec{\gamma}}|_h= {\mathsf d}^h(s,\sigma)$. Now clearly $\gamma$ is a future directed timelike curve from the point $(0,s)\in \{0\}\times \Sigma$ to $a=(t,\sigma)\in  \{t\}\times \Sigma$.
\end{pr}

In the next lemma we obtain a bound  on the cardinality of (Lorentzian) $\eps$-nets in terms of the cardinality of (metric) $\eps$-nets of $\Sigma$ and the time lapse.
\begin{lem}[Bounding the cardinality of Lorentzian $\eps$-nets]\label{lem-card-eps-prod}
 Let $0<t_-<t_+$ and let $A\subseteq J^+(\Sigma_{t_-})\cap J^-(\Sigma_{t_+})$. Let $0<\eps \leq t_-$ and let $s_1,\ldots,s_N$ be an $\frac{\eps}{3}$-net in $\Sigma$. Then there is a Lorentzian $\eps$-net of cardinality at most $\lceil\frac{3(t_+-t_-)}{\eps}\rceil\cdot N$ covering $A$, where $\lceil x\rceil$ is the smallest integer greater or equal than $x\in \R$.
\end{lem}
\begin{pr}
 Set $L:= \lceil\frac{3(t_+-t_-)}{\eps}\rceil$ and $t_i:=t_- + i\frac{\eps}{3}$ for $i=-1,\ldots,L+1$. Moreover, set $x_{i,j}:=(t_i,s_j)$ for $i=-1,\ldots,L+1$, $j=1,\ldots, N$ and 
 $$\J:=\{J(x_{i-1,j},x_{i+2,j}):i=0,\ldots,L-1; j=1,\ldots,N\}.$$ 
 Then $\J$ is a collection of causal diamonds $J(x_{i-1,j},x_{i+2,j})$, which satisfy $\tau(J(x_{i-1,j},x_{i+2,j})) = \tau(x_{i-1,j},x_{i+2,j})= t_{i+2} - t_{i-1} =  \eps$. Note that $t_{-1} = t_--\eps \geq 0$.
 
 Next, we show that $A\subseteq \bigcup_{J\in\J} J$. For each $a=(t,\sigma)\in A$ there exist $i\in\{-1,\ldots,L+1\}$ such that $t_i\leq t\leq t_{i+1}$ and $j\in\{1,\ldots,N\}$ such that ${\mathsf d}^h(\sigma,s_j)<\frac{\eps}{3}$. Define $\gamma\colon[0,1]\rightarrow M$ as $\gamma(r)=(rt+(1-r)t_{i-1},\vec{\gamma}(1-r))$, where $\vec{\gamma}\colon[0,1]\rightarrow \Sigma$ is a minimizing geodesic from $\vec{\gamma}(0) = s_j$ to $\vec{\gamma}(1) = \sigma$ with $|\dot{\vec{\gamma}}|_h = {\mathsf d}^h(s_j,\sigma)$. Therefore
\begin{align*}
g(\dot\gamma,\dot\gamma) = -(t-t_{i-1})^2 + {\mathsf d}^h(s_j,\sigma)^2 < -\frac{\eps^2}{9} + \frac{\eps^2}{9} = 0\,,
\end{align*}
as $t\geq t_i$. Thus $a\in J^+(x_{i-1,j})$. Similarly, defining $\lambda\colon[0,1]\rightarrow M$ by setting $\lambda(r):=(r t_{i+2} + (1-r)t, \vec{\gamma}(r))$ for $r\in[0,1]$ gives a future directed timelike curve from $a$ to $x_{i+2,j}$ as $t\leq t_{i+1}$ and 
\begin{align*}
 g(\dot\lambda,\dot\lambda) = -(t_{i+2}-t)^2 + {\mathsf d}^h(s_j,\sigma)^2 < -\frac{\eps^2}{9} + \frac{\eps^2}{9}=0\,.
\end{align*}
\end{pr}

\begin{cor}[Scaled product metric]\label{cor-card-eps-scal-prod}
 Let $C>0$ and consider the Lorentzian metric 
 \begin{equation}\label{eq:defrhoC}
 \rho_C:= -C^2 {\rm d}t^2 + h  \quad \text{on }\R\times\Sigma.
 \end{equation}
 Then the corresponding bound on the cardinality of an $\eps$-net as in Lemma \ref{lem-card-eps-prod} is $\lceil\frac{3(t_+-t_-)}{\eps}\rceil\cdot N_C$, where $N_C$ is the cardinality of an $\frac{C\eps}{3}$-net in $\Sigma$.
\end{cor}

Next, we establish a geometric pre-compactness theorem for a general class of globally hyperbolic spacetimes. Recall that  (see \cite[Thm.\ 1.1]{BS:05} and \cite[Lem.\ 3.5]{MS:11}) every globally hyperbolic spacetime $(M,g)$ isometrically splits as
\begin{equation*}
 (\R\times \Sigma, -\beta {\rm d}t^2 + h_t)\,,
\end{equation*}
where $\beta\colon\R\times \Sigma\rightarrow(0,1]$ is smooth, $\Sigma$ is a smooth spacelike Cauchy-hypersurface and $h_t$ is a $t$-dependent Riemannian metric on $\{t\}\times \Sigma$. For notational simplicity we will write $(\Sigma,h_0)$ for $(\{0\}\times \Sigma,h_0)$. The following geometric pre-compactness theorem for globally hyperbolic spacetimes builds on top of such a splitting and on the abstract pre-compactness for pLGH convergence established before.  
Recall that, given two Lorentzian metrics $g, g'$ on $M$, the notation  $g\preceq g'$ means that 
$$g(v,v)\leq 0 \implies g'(v,v)\leq 0, \quad \text{for all $v\in TM$}.$$

\begin{thm}[Geometric pre-compactness]\label{thm-geo-pre-comp}
Let $C\colon(0,\infty)\rightarrow(0,\infty)$ and $N\colon(0,\infty)\rightarrow \N$ be given functions. Consider the following family  $\M_{C,N}$  of smooth globally hyperbolic spacetimes
\begin{align*}
 \M_{C,N} :=\{(\R\times\Sigma, -\beta {\rm d}t^2 + & h_t) :\; \Sigma \text{ is a compact smooth manifold},\\
 &\beta\colon\R\times \Sigma\rightarrow(0,1] \text{ is a  smooth function},\\
 & \forall \eps>0\, \exists \eps\text{-net } S \text{ in } \Sigma \text{ w.r.t.\ } \mathsf{d}^{h_0} \text{ with } |S|\leq N(\eps),\\
 & \forall T>0:\,\rho_{C(T)} \preceq  -\beta {\rm d}t^2 + h_t \text{ on } [-T,T]\times \Sigma\}\,,
\end{align*}
where $\rho_{C(T)}$ is defined as in  \eqref{eq:defrhoC}.
 Then:
 \begin{enumerate}
\item  for each $T>0$, there exists a uniform bound on the cardinality of Lorentzian $\eps$-nets needed to cover the slab $[-T,T]\times \Sigma$. More precisely: Let $T>0$ and $\eps>0$. Then for every $(\R\times\Sigma, -\beta {\rm d}t^2 + h_t)\in \M_{C,N}$ there is a Lorentzian $\eps$-net of cardinality at most
\begin{align*}
 \left\lceil\frac{6T}{\eps}  \right\rceil \cdot N\left(\frac{C(T)\eps}{3}\right)\,,
\end{align*}
covering $[-T,T]\times \Sigma$.

\item  $\M_{C,N}$ is \emph{sequentially pre-compact}; i.e., for each sequence in $\M_{C,N}$ there is a subsequence that strongly $\mathrm{pLGH}$-converges to a covered \LpLS that satisfies the point distinction property \eqref{eq-pdp}.

\item The limit space in (ii) can be forward completed and is still the limit of such a subsequence. 

\item The forward completed limit as in (iii) is \emph{unique} in  the class of properly covered \LpLSn s with continuous time-separation functions $\tau$, closed anti-symmetric causal relation, metrizable chronological topology and that satisfy the point-distinction property \eqref{eq-pdp}.
In particular (see Remark \ref{rem:UniqGH}), it is unique within the class of covered, globally hyperbolic, \LpLSn s.

\item In case the limit is smooth: there exists at most one smooth globally hyperbolic spacetime arising as such a strong limit, up to smooth isometry.
\end{enumerate}
\end{thm}
\begin{pr}
\textbf{Step 1}. Pre-compactness.
\\Let $T>0$, $\eps>0$, and let $(\R\times\Sigma, g)\in\M_{C,N}$, where $g=-\beta {\rm d}t^2 + h_t$. By Corollary \ref{cor-card-eps-scal-prod} applied to $\rho_{C(T)}=-C(T)^2 {\rm d}t^2 + h_0$, there is a (Lorentzian) $\eps$-net $\J$ for $[-T,T]\times\Sigma$ of cardinality $\lceil\frac{6T}{\eps}\rceil\cdot N\bigl(\frac{C(T)\eps}{3}\bigr)$. By construction, each $J_{\rho_{C(T)}}\in\J$ is of the form $J_{\rho_{C(T)}}((s,x),(t,x))$ with $-T\leq s< t\leq T$, $t-s=\eps$ and $x\in\Sigma$. Since $\rho_{C(T)}\preceq g$ on $[-T,T]\times\Sigma$, then   
$$
J_{\rho_{C(T)}}((s,x),(t,x))\subseteq J_g((s,x),(t,x)), \quad \text{ for all }J_{\rho_{C(T)}}\in\J.$$
In particular, the union over $J\in \J$ of $J_g((s,x),(t,x))$ covers $[-T,T]\times\Sigma$. Since by assumption $\beta\leq 1$, it is clear that  $\tau_g(J)\leq t-s=\eps$, for all $J\in\J$. In conclusion, $\J$ is an $\eps$-net for $[-T,T]\times\Sigma$ with respect to $\tau_g$.

Finally, we show that $\M_{C,N}$ satisfies the assumptions of the general pre-compactness theorem \ref{thm-pre-comp-I}. Let $(M,g)\in\M_{C,N}$,  fix a point $\sigma\in\Sigma$, and set $o:=(0,\sigma)\in M$. Let $k\in\N,k\geq 1,$ and set $U_k:=[-k,k]\times\Sigma$. Then the timelike diameter of $U_k$ is uniformly bounded by $2k$ (independently from $(M,g)$) as $\beta\leq1$ and $U_k\subseteq U_{k+1}$ for all $k\in\N, k\geq 1$. Thus each $(M,\ell_g,o,(U_k)_{k\geq 1})$ is a covered \LpLS satisfying the assumptions of Theorem \ref{thm-pre-comp-I}. Therefore, Theorem \ref{thm-pre-comp-I} yields  a subsequence strongly converging to a covered \LpLSn. Then, Theorem \ref{thm-quo-lim} establishes that this subsequence also strongly convergences to the time-separation quotient (which satisfies the point distinction property \eqref{eq-pdp}). Moreover, Theorem \ref{thm-lpls-compl-ex} and Proposition \ref{thm-lim-compl} give that the subsequence  converges also to the forward completion of the limit.
\smallskip

\textbf{Step 2}. Uniqueness of the limit in appropriate classes.\\
First, we can apply Proposition \ref{prop-uni-den} to get uniqueness in the class of properly covered \LpLSn s with continuous time-separation functions $\tau$, closed anti-symmetric causal relation, metrizable chronological topology and that satisfy the point-distinction property \eqref{eq-pdp} (as the vertices are timelike forward dense in each covering set $U_k$). 

For the second part, assume there are two smooth globally hyperbolic spacetimes, arising  as pLGH strong limits of the converging subsequence. Observe that the chronological topologies coincide with the manifold topologies; moreover, the above ensures that the induced \LpLSn s are isometric. In particular, the manifolds are homeomorphic (since isometries are continuous) and have the same dimension (as manifolds) by invariance of domain. Thanks to the Hawking--King--McCarthy theorem \cite{HKM:76} (cf.\ \cite[Prop.\ 3.34]{MS:08}, \cite[Thm.\ 4.17]{BEE:96}) such a time-separation preserving homeomorphism is actually a smooth isometry, yielding  the claimed uniqueness.
\end{pr} 

\begin{rem}
Notice that the condition $\rho_{C(T)} \preceq  -\beta {\rm d}t^2 + h_t$ is conformally invariant; in particular,  it does not depend on the curvature of $(\R\times\Sigma,g)$.
The interesting feature of Theorem \ref{thm-geo-pre-comp} is that it implies that a control on the causality and a control of the cardinality of (metric) $\eps$-nets of the spacelike slice $(\Sigma,h_0)$ allows to control the cardinality of $\eps$-nets for the spacetime.  This feature heavily relies on the causal structure and is markedly Lorentzian; we are not aware of pre-compactness results in Riemannian signature having this flavour.
\end{rem}

\subsection{A pre-compactness result under curvature assumptions}\label{subsec-precomp-curv}

Let $(M^{n+1},g)$ be a Lorentzian manifold, $n \ge 2$.
We adopt the following sign conventions for the Riemann curvature tensor $R^M$ of $(M^{n+1},g)$:
\begin{align*}
&R^M(X,Y)Z:=\nabla_X \nabla_Y Z- \nabla_Y \nabla_X Z- \nabla_{[X,Y]} Z, \\
&{\rm Sec}^M(X,Y):=\frac{g(R^M(X,Y)Y, X)}{g(X,X) g(Y,Y)-g(X,Y)^2}, 
\end{align*}
whenever $X,Y\in T_pM$ span a non-degenerate plane.  Notice the opposite sign (with respect to Riemannian setting) when $X$ is spacelike and $Y$ is timelike. For this reason, we say that $(M,g)$ has \emph{timelike sectional curvature bounded below} by $K_\perp\in \R$ at some $p\in M$ and write ${\rm TSec}^M_p\geq K_\perp$ if
$$
-{\rm Sec}^M(X,Y) \geq K_\perp, \quad \text{for all $X,Y\in T_pM$ with $g(X,X)<0$ and $g(Y,Y)>0$.}
$$

The Ricci tensor of $M$ is defined as
\[
{\rm Ric}^M(X,Y):=\sum_{i=0}^{n}\epsilon_i\;g(R^M(e_i, X)Y, e_i),
\]
where $\{e_0,\ldots, e_{n}\}$ is a Lorentz-orthonormal basis of $T_pM$, i.e. 
$$
\epsilon_i:=g(e_i, e_i)=1 \; \text{ for  }i=1,\ldots, n, \quad \epsilon_{0}:=g(e_{0}, e_{0})=-1.
$$
Let $\iota\colon\Sigma^n \hookrightarrow M$ be a smooth, spacelike hypersurface with:
\begin{itemize}
  \item future-pointing unit normal field $\nu$,
  \item second fundamental form $h^\Sigma(X,Y) = g(\nabla_X \nu, Y)$,
  \item induced metric $g^\Sigma = \iota^* g$,
  \item mean curvature $H = \mathrm{tr}_{g^\Sigma} h^\Sigma$.
\end{itemize}

\begin{lem}[Ricci-lower-bound inheritance for spacelike hypersurfaces]\label{lem:RicciLBSigma}

Let $(M^{n+1},g)$ be a Lorentzian manifold, $n \ge 1$. Let $\Sigma^n \hookrightarrow M$ be a smooth, spacelike hypersurface.
Assume there exist constants $K, K_\perp, \Lambda \ge 0$ such that:
\begin{enumerate}
  \item \textbf{Ambient Ricci lower bound along $\Sigma$:} $\mathrm{Ric}^M \ge -K\,g$ along $\Sigma$,
  \item \textbf{Timelike sectional lower bound along $\Sigma$:} For all $p\in \Sigma$, it holds that ${\rm TSec}^M_p\geq - K_\perp$; 
  \item \textbf{Bounded second fundamental form:} $\lvert h^\Sigma \rvert \le \Lambda$.
\end{enumerate}

Then the intrinsic Ricci tensor of $\Sigma$ satisfies:
\begin{equation}\label{eq:RicSRicM}
\mathrm{Ric}^\Sigma \ge -\left(K + K_\perp + 2n \Lambda^2 \right)\,g^\Sigma.
\end{equation}
\end{lem}

\begin{proof}
Fix a tangent vector $X \in T_p \Sigma$, and extend it locally. The contracted Gauss equation gives:
\begin{equation*}
\mathrm{Ric}^\Sigma(X,X) = \mathrm{Ric}^M(X,X) + g(R^M(\nu,X) X, \nu) + H\, h^\Sigma(X,X) - \lVert h^\Sigma(X,\cdot) \rVert^2. 
\end{equation*}
Using the three assumptions, we can estimate each term and obtain:
\[
\mathrm{Ric}^\Sigma(X,X) \ge -\left( K + K_\perp + 2n \Lambda^2  \right) |X|^2.
\]

Since $X$ was arbitrary, this proves the claim. 
\end{proof}

The next step is to use lower bounds on the timelike sectional curvature to control the causal relation. Let us start by recalling some classical facts of smooth differential geometry. Assume that
\[
r\colon U\subset M\longrightarrow\mathbb R
\qquad\text{is smooth with}\qquad
g(\nabla r,\nabla r)=-1 .
\]
Let
\[
\mathbf T:=\partial_r:=-\nabla r,
\qquad
\Hess\,(- r)(X,Y):=-g(\nabla_X\nabla r,Y)=\frac{1}{2} \Lie_{\mathbf T}g(X,Y),
\]
where  $\Lie_{\mathbf T}$ denotes the Lie derivative in the direction of the vector field ${\mathbf T}$. 


\begin{prop}
\label{prop:radial}
It holds that
\begin{align*}
\Lie_{\mathbf T}\Hess\, r+\Hess^{2}r&=g(R(\cdot,\mathbf T)\mathbf T,\cdot).
\end{align*}
\end{prop}

\begin{proof}
Insert \(N=\mathbf T\) with \(g(N,N)=-1\) into the computation of
\cite[Thm.\,3.2.2]{Pet:16}; only the sign of \(g(N,N)\) changes.
\end{proof}

Set \(\Sigma_r:=\{r=\mathrm{const}\}\) the spacelike level hypersurfaces.
Along each integral curve of \(\mathbf T\),
we obtain Gaussian normal coordinates:
\begin{equation}\label{eq:g=split}
g=-\dd r^{2}+g_r,\qquad
g_r:=g|_{T \Sigma_r}.
\end{equation}
Using the expression \eqref{eq:g=split}, Proposition~\ref{prop:radial}
simplifies to
\begin{equation}\label{eq:dergr}
\partial_r g_r = 2\,\Hess\, r,\qquad
\partial_r\Hess\, r+\Hess^{2}r=g(R(\cdot,\mathbf T)\mathbf T,\cdot).
\end{equation}

These two transport equations encode the evolution of the induced
metric and the second fundamental form along the timelike normal flow.

Let $\Sigma^{n}\subset M^{n+1}$ be a spacelike achronal hypersurface.
The \emph{signed time-separation} from $\Sigma$ is the function  $\tau_{\Sigma}:M\to [-\infty, +\infty]$ defined by
\begin{equation}\label{eq:deftauV}
\tau_{\Sigma}(x):=
\begin{cases}
\sup_{y\in \Sigma} \tau(y,x), &\quad \text{ for }x\in I^{+}(\Sigma)\\
-\sup_{y\in \Sigma} \tau(x,y),& \quad \text{ for }x\in I^{-}(\Sigma) \\
0 &\quad \text{ otherwise}.
\end{cases}
\end{equation}

Until the end of this section, we will work under the assumption that  $(M^{n+1}, g)$ is globally hyperbolic and $\Sigma$ is  a compact Cauchy hypersurface.  Under such conditions, it is standard to check that,  for all $x\in I^{+}(\Sigma)$ (resp. for all $x\in I^{-}(\Sigma)$) there exists 
$y_{x}\in \Sigma$ with $\tau_{\Sigma}(x)=\tau(y_{x},x)>0$ (resp. $\tau_{V}(x)=-\tau(x, y_{x})<0$).  Let 
\begin{equation}\label{eq:defUSigma}
U_\Sigma:=\{x\in M\colon \tau_\Sigma \text{ is smooth at $x$ with } g(\nabla  \tau_\Sigma(x), \nabla  \tau_\Sigma(x))=-1\}.
\end{equation}
The following are classical facts:
\begin{itemize}
\item $U_\Sigma\subset M$ is an  open set of full measure;
\item $U_\Sigma\cap I^+(\Sigma)$ is future dense in  $I^+(\Sigma)$; resp.\;  $U_\Sigma\cap I^-(\Sigma)$ is past dense in  $I^-(\Sigma)$;
\item  for each $x\in U_\Sigma$ there exists a unique geodesic $\gamma^{\Sigma, x}$ maximizing the time-separation between $x$ and $\Sigma$, and the support of $\gamma^{\Sigma, x}$ is entirely contained in $U_\Sigma$;
\item  $U_\Sigma$ is diffeomorphic, via the exponential map $\exp_\Sigma$ based on $\Sigma$, to a suitable open subset $\tilde{U}_\Sigma\subset \R\times \Sigma$. Identify $U_\Sigma$ with $\tilde{U}_\Sigma$ via such a diffeomorphism. Setting $r=\tau_\Sigma$ on $U_\Sigma$,  the metric $g$ can be written  as in \eqref{eq:g=split} on $U_\Sigma$.
\end{itemize}

\begin{prop}\label{prop:ConesUSigma}
Let $(M^{n+1},g)$ be a smooth globally hyperbolic spacetime and let $\Sigma^{n}\subset M^{n+1}$ be a compact Cauchy hypersurface. Let $\tau_\Sigma$ and $U_\Sigma$ be defined in \eqref{eq:deftauV} and \eqref{eq:defUSigma}, respectively.
Assume:
\begin{enumerate}
 \item \textbf{$L^1$-lower bound on the timelike sectional curvature}. There exists $\alpha\in L^1_{loc}(\R)$ such that 
 \begin{equation}\label{eq:TsecLBL1}
 {\rm TSec}^M_{x}\geq  \alpha(\tau_\Sigma(x)), \quad \text{for all } x\in M;
 \end{equation}
  \item \textbf{Bounded second fundamental form of $\Sigma$.} There exists $\Lambda>0$ such that the second fundamental form of $\Sigma$  in $M$ satisfies  $\lvert h^\Sigma \rvert \le \Lambda$.
\end{enumerate}
Then there exists a continuous function $C(\cdot):\R\to (0,\infty)$, depending only on $\alpha(\cdot)$ and $\Lambda$, such that
\begin{equation}\label{eq:claimConesUSigma}
    - C(r)^2 \dd r^2 +g_{|T_{\sigma}\Sigma}  \preceq  g^{(r,\sigma)} \text{ for all } (r,\sigma)\in U_\Sigma,
\end{equation}
where $g^{(r,\sigma)}$ is the Lorentzian metric $g$ at the point $(r,\sigma)\in U_\Sigma$, and $g_{|T_{\sigma}\Sigma}$ is the Riemannian metric on $T_{\sigma}\Sigma$ obtained by restriction of $g$.
\end{prop}

Of course, the assumption \eqref{eq:TsecLBL1}  is satisfied if, more strongly, there exists $K_\perp>0$ such that 
$$
{\rm TSec}^M\geq -K_\perp\quad \text{ on } M.
$$

\begin{proof}
Fix $x=(r,\sigma)\in U_\Sigma\cap I^+(\Sigma)$; the arguments for $(r,\sigma)\in U_\Sigma\cap I^-(\Sigma)$  are completely analogous. Let $\gamma^{\Sigma, x}$  be the unique geodesic  maximizing the time-separation between $x$ and $\Sigma$. Recall that the support of $\gamma^{\Sigma, x}$ is entirely contained in $U_\Sigma$. Therefore, the coupled  differential equations  \eqref{eq:dergr} hold in a neighborhood of $\gamma^{\Sigma, x}$. 
  
Since $g^{(r,\sigma)}$ is written as in \eqref{eq:g=split}, for the claim \eqref{eq:claimConesUSigma} it is sufficient to show that
\begin{equation}\label{eq:gtCtg0Pf}
g_r(X,X)\leq C(r)^{-2} \; g_0(X,X), \quad \forall X\in T_\sigma\Sigma.
\end{equation}
Therefore, we would like an upper bound $\partial_r g_r\leq f(r)g_0$ for some $f\in L^1_{loc}(\R)$. Observe that $\Hess\;r|^{r=0}_{T\Sigma\times T\Sigma}$ coincides with the second fundamental form of $\Sigma$, which is bounded by the second assumption.
If the timelike sectional curvature satisfies the first assumption, then the coupled  differential equations \eqref{eq:dergr} imply the desired \eqref{eq:gtCtg0Pf} via standard arguments of ordinary differential equations.
\end{proof}

\begin{thm} [Pre-compactness under curvature assumptions]\label{thm:CurvPrecomp}
Let $n\in \N_{\geq 1}$ and $K, K_\perp, \Lambda,D>0$.
Consider the family $\mathcal{M}_{n,K,K_\perp,\Lambda,D}$ of globally hyperbolic spacetimes $(M^{n+1},g)$ with
\begin{description}
  \item [(a) Timelike sectional bounded below:] For all $x\in M$, it holds that ${\rm TSec}^M_x\geq -  K_\perp$;
\end{description}
 admitting a compact Cauchy hypersurface $\Sigma^{n}\subset M$ with
\begin{description}
 \item[(b) Ambient Ricci lower bound along $\Sigma$:] $\mathrm{Ric}^M \ge -K\,g$ along $\Sigma$,
  \item[(c) Bounded second fundamental form:] $\lvert h^\Sigma \rvert \le \Lambda$,
   \item[(d) Bounded diameter:] the diameter of $\Sigma$ with respect to the induced Riemannian metric is bounded by $D$; i.e., ${\rm diam}_{g_{|T\Sigma}} (\Sigma)\leq D$.
\end{description}
Then: 
\begin{enumerate}
\item $\M_{n,K,K_\perp, \Lambda,D}$ is sequentially pre-compact in the strong $\mathrm{pLGH}$-convergence; i.e., for each sequence in $\M_{n,K,K_\perp, \Lambda,D}$ there is a subsequence that strongly $\mathrm{pLGH}$-converges to a covered \LpLS that satisfies the point distinction property \eqref{eq-pdp}.

\item This limit as in (i) can be forward completed and is still the limit of such a subsequence. 

\item The forward completed limit as in (ii) is \emph{unique} in  the class of properly covered \LpLSn s with continuous time-separation functions $\tau$, closed anti-symmetric causal relation, metrizable chronological topology and that satisfy the point-distinction property \eqref{eq-pdp}.
In particular (see Remark \ref{rem:UniqGH}), it is unique within the class of covered, globally hyperbolic, \LpLSn s.

\item In case the limit is smooth: there exists at most one smooth globally hyperbolic spacetime arising as such a strong limit, up to smooth isometries.
\end{enumerate}
\end{thm}

\begin{proof}
We follow the strategy of the proof of Theorem \ref{thm-geo-pre-comp} but will not use the theorem per se. Instead, we will apply the abstract precompactness result Theorem \ref{thm-pre-comp-I}.  Since by assumption $\Sigma$ is a Cauchy hypersurface, without loss of generality, we can assume that $M=\R\times\Sigma$, up to a diffeomorphism. From Lemma \ref{lem:RicciLBSigma}, we get that the Riemannian metric $g_{|T\Sigma}$ on $\Sigma\subset M$ has Ricci curvature bounded from below. By the Bishop--Gromov inequality, it follows that $(\Sigma, g_{|T\Sigma})$ is doubling and thus, using the bound on the diameter,  it satisfies the assumption on the existence of Riemannian $\epsilon$-nets as in Theorem \ref{thm-geo-pre-comp}. Proposition \ref{prop:ConesUSigma} implies that the assumption on the causal cones in Theorem \ref{thm-geo-pre-comp} is satisfied on $U_\Sigma$. We can thus apply the first part of  Theorem \ref{thm-geo-pre-comp} to $U_\Sigma$ and infer the existence of $\epsilon$-nets as follows.

Fix $T>0$, then there is a $R=R(T)>0$ such that $$\bigr([-T,T]\times\Sigma\bigl)\cap U_\Sigma \subseteq \exp_\Sigma^{-1}\bigr([-R,R]\times\Sigma\bigl).$$ 
Let $C(\cdot):\R\to (0,\infty)$ be given by Proposition \ref{prop:ConesUSigma} and set  $$C^-:=\min_{r\in[-R,R]}C(r)>0.$$ The combination of \eqref{eq:g=split} and  \eqref{eq:claimConesUSigma} yields that
$$-(C^{-})^2 \dd r^2 + g_0 \preceq -C(r)^2 \dd r^2 + g_0 \preceq -\dd r^2 + g_r = g$$ on $U_\Sigma$ (suppressing the diffeomorphism $\exp_\Sigma\colon U_\Sigma\rightarrow \tilde U_\Sigma$).
Consequently, Corollary \ref{cor-card-eps-scal-prod} ensures the existence of an $\eps$-net having an explicit bound on its cardinality and covering $\exp_\Sigma^{-1}\bigr([-R,R]\times\Sigma\bigl)$; i.e., we obtain a finite covering of causal diamonds (with respect to the causal relation in $U_\Sigma$):
\begin{equation*}
    \bigl([-T,T]\times\Sigma\bigr)\cap U_\Sigma \subseteq \bigcup_{i=1}^{N(T)} J(x_i,y_i;U_\Sigma)\,,
\end{equation*}
 where $x_i,y_i\in U_\Sigma$, for all $i=1,\ldots, N(T)$.
 Since the operation of taking the topological closure commutes with finite unions, we obtain that 
\begin{equation*}
    \overline{\bigl([-T,T]\times\Sigma\bigr)\cap U_\Sigma} \subseteq \bigcup_{i=1}^{N(T)} \overline{J(x_i,y_i;U_\Sigma)}
    \subseteq \bigcup_{i=1}^{N(T)} J(x_i,y_i)
    \,,
\end{equation*}
giving an $\eps$-net of bounded cardinality for $\overline{\bigl([-T,T]\times\Sigma\bigr)\cap U_\Sigma}$ with causal diamonds in $M$.
We next claim that the sets $\overline{\bigl([-T,T]\times\Sigma\bigr)\cap U_\Sigma}$ for $T>0$ yield an exhaustion of $M$. To this aim, let $p\in \R\times \Sigma$ and choose $T>0$ such that $p\in (-T,T)\times\Sigma$, which is open in $M$. The density of $U_\Sigma$ in $M$ implies that $p\in \overline{\bigl((-T,T)\times\Sigma\bigr)\cap U_\Sigma}$. As $\tau$ is bounded on each $[-T,T]\times\Sigma$, choosing a sequence $T_n\nearrow\infty$ yields a cover of $M$ fitting the assumptions of Theorem \ref{thm-pre-comp-I}. We thus conclude as in the proof of Theorem \ref{thm-geo-pre-comp}.
\end{proof}

\section{Measured pLGH convergence}\label{sec-meas-conv}
Let us start by defining a \emph{measured \LpLSn}, as a \LpLS (in the sense of Definition \ref{defi-lpls}) endowed with a reference non-negative Borel measure. Recall that a topology is said to be \emph{Polish} if it can be induced by a complete and separable metric.

\begin{defi}[Measured \LpLSn]\label{defi-m-lpls}
 Let $(X, \ell)$  be a \LpLSn \;whose topology is Polish and such that causal diamonds are Borel. Let $\mm$ be a non-negative Borel measure, finite on causal diamonds. Then the triplet $(X, \ell, \mm)$ is called \emph{measured} \LpLSn.
 \end{defi}

 Next, we discuss how to induce a measure into an $\epsilon$-net. Recall Definition \ref{def:V(S)} of vertices of a family of causal diamonds, and Definition \ref{def:epsNet} of $\eps$-net for a set. 

\begin{defi}[Measured $\eps$-net]
Let $(X, \ell, \mm)$ be a measured \LpLSn.
 Let $\eps>0$ and $A\subseteq X$ be a Borel subset. 
 Let $S=(J_i)_{i\in \N}=\big(J(p_i, q_i)\big)_{i\in \N}$  be a countable (or finite) $\eps$-net for $A$. 
 Define the measure
 \begin{equation}\label{eq:defmS}
 \mm_S= \sum_{i=1}^{\infty} \frac{1}{2}\mm\Big((J_{i}\cap A)\setminus \bigcup_{j=1}^{i-1} J_j\Big) \left(\delta_{p_i}+\delta_{q_i}\right).
 \end{equation}
 The pair $(S,\mm_S)$ is called \emph{measured $\epsilon$-net} for $A$.
\end{defi}
Note that \eqref{eq:defmS} defines a measure. Indeed $\mm_S$ is the monotone limit of the increasing family of   finite measures  $\sum_{i=1}^N \frac{1}{2}\mm\left((J_{i}\cap A)\setminus \cup_{j=1}^{i-1} J_j\right) \left(\delta_{p_i}+\delta_{q_i}\right)$.

We next give a measured version of the LGH convergence of subsets. Roughly, we add the weak convergence of measures to the LGH convergence (see Definition \ref{def-con-subs}). The weak convergence of measures is understood in duality with real valued continuous functions with compact support.

\begin{defi}[Measured LGH-convergence of subsets]\label{def-con-subs-Meas}
 Let $(X_n, \ell_n, \mm_n)$, $n\in \N$, and $(X, \ell, \mm)$ be measured \LpLSn s. For each $n\in\N$, let $A_n$ be a Borel subset of $X_n$ and let $A$ be a Borel subset of $X$. We say that $A_n$ converges to $A$ in the (resp.\;strong) measured Lorentzian Gromov--Hausdorff sense ({\rm mLGH} for short), and write $A_n\mLGHtop A$  (resp.\;strongly) if $A_n\LGHtop A$  (resp.\;strongly) and if there exist Borel maps $g^l_n\colon V(S^l_n)\rightarrow V(S^l)$  realizing a $1/l-$correspondence of $V(S^l_n)$ and $V(S^l)$ such that
 \begin{equation}\label{eq:defMLGHA}
\lim_{l\to \infty} \lim_{n\to \infty} (g^l_n)_\sharp \mm_{S^l_n} = \mm\llcorner A \quad \text{weakly as measures}. 
 \end{equation}
 Here $\mm\llcorner A$ denotes the restriction of the measure $\mm$ to the Borel set $A$.
\end{defi}
We can now define a measured version of the pointed Lorentzian Gromov--Hausdorff convergence (Definition \ref{def-con-cov}).

\begin{defi}[Measured pLGH-convergence of covered  measured \LpLSn s]\label{def-con-cov-Meas}
  Let $(X_n, \ell_n, \mm_n, o_n, \U_n)$, $n\in \N$, and $(X, \ell, \mm, o, \U)$ be covered measured \LpLSn s, with $\U=(U_{k,\infty})_{k\in\N}$ and $\U_n = (U_{k,n})_{k\in\N}$ families of Borel subsets. We say that $\bigl((X_n, \ell_n, \mm_n, o_n, \U_n)\bigr)_{n\in\N}$ converges to $(X, \ell, \mm, o, \U)$ in the (resp.\;strong) pointed measured Lorentzian Gromov--Hausdorff sense ({\rm pmGH} for short), and write $$(X_n, \ell_n, \mm_n, o_n, \U_n) \pmLGHtop (X, \ell, \mm, o, \U)\quad  \text(resp.\ strongly),$$ if for each $k\in\N$ it holds that $U_{k,n}\mLGHtop U_{k,\infty}$ (resp.\ strongly) as $n\to\infty$.
\end{defi}

From the construction of the measure $\mm_S$ as in  \eqref{eq:defmS} and following the proof of Theorem \ref{thm-dis-lim},  one can show the next result.

\begin{thm}\label{thm-dis-lim-meas}
 Each smooth globally hyperbolic spacetime $(M,g)$ endowed with a continuous weighted measure $\mm=\exp(\Phi)\, \rm{dvol}_g$, $\Phi\in C^0(M)$, is the strong pointed measured Lorentzian Gromov--Hausdorff limit of countable (discrete) measured \LpLSn s. In fact, each covering set of the approximating sequence can be chosen to be finite. 
\end{thm}

We next refine the pre-compactness Theorem \ref{thm-pre-comp-I}, obtaining a pre-compactness result for the  pmLGH convergence.  

\begin{thm}[Pre-compactness for pmLGH]\label{thm-pre-comp-Ibis}
 Let $\X$ be a class of covered measures \LpLSn s such that each $(X, \ell, \mm, o, \U) \in\X$, with covering $\U=(U_k)_{k\in\N}$,  satisfies the assumptions (i), (ii), (iii) of Theorem \ref{thm-pre-comp-I}, and moreover:
 \begin{enumerate}
 \item[(iv)] For all $k\in\N$ there exists $C_k>1$ such that $\frac{1}{C_k}\leq \mm(U_k) \leq C_k$.
 \end{enumerate}
 
Then any sequence in $\X$ has a converging subsequence in {\rm pmLGH}-sense; i.e., for any sequence  $( (X_n, \ell_n, \mm_n, o_n, \U_n))_n \subset \X$ there exists a subsequence $(n_j)_j\subset \N$ and a covered measured \LpLS $(X, \ell, \mm, o, \U)$ such that 
$$
  (X_{n_j}, \ell_{n_j}, \mm_{n_j}, o_{n_j}, \U_{n_j}) \pmLGHtop (X, \ell, \mm, o, \U)  \quad  \text{strongly, as $j\to \infty$}.
$$ 
\end{thm}

\begin{proof}
Let $( (X_n, \ell_n, \mm_n, o_n, \U_n))_n \subset \X$. By Theorem \ref{thm-pre-comp-I} we know that there exists a subsequence $(n_j)_j\subset \N$ and a covered \LpLS $(X_\infty, \ell_\infty, o_\infty, \U_\infty)$  such that
$$
\XtnjoU\pLGHtop (X_\infty, \ell_\infty, o_\infty, \U_\infty) \quad  \text{strongly, as $j\to \infty$}.
$$

By the construction performed in the proof of Theorem \ref{thm-pre-comp-I}, we have that:
\begin{enumerate}[(a)]
\item $X_\infty$ is a countable set.
\item For every $m\geq 1$, and every $k\in \N$, $U_{k,\infty}$ admits a finite $1/m$-net $S_{k,\infty,m}$, $i=1, \ldots, N_{k,m},$ with a finite set of vertices 
$$
V(S_{k,\infty,m})=\{x^i_{k,\infty,m}, y^i_{k,\infty,m} \}_{i=1,\ldots, N_{k,m}}.
$$
Moreover, the family of vertices is increasing, i.e,
$$V(S_{k,\infty,m})\subset V(S_{k,\infty,m+1}),$$
and 
$U_{k,\infty}$ is obtained as the union of such vertices:
$$
U_{k,\infty}=\bigcup_{m=1}^\infty  V(S_{k,\infty,m}).
$$  
\item For every $m\geq 1$, and every $k\in \N$, there exists maps 
$$
g_{k,j,m}:V(S_{k,j,m})\to V(S_{k,\infty,m})
$$ with distortion less than $1/m$, where $V(S_{k,j,m})$ denotes the set of vertices of the $1/m$-net $S_{k,j,m}$ for $U_{k,j}\subset X_{n_j}$, $U_{k,j}\in \U_{n_j}$. Moreover,
\begin{align}
V(S_{k,j,m-1})\subset V(S_{k,j,m}) \quad  &\text{and} \quad  g_{k,j,m}|_{V(S_{k,j,m-1})}=g_{k,j,m-1} \label{eq:Resm-1=m}  \\
 V(S_{k-1,j,m})\subset V(S_{k,j,m}) \quad  &\text{and}\quad g_{k,j,m}|_{V(S_{k-1,j,m})}=g_{k-1,j,m}.\label{eq:Resk+1=k}
\end{align}

\end{enumerate}
We endow $X_\infty$ with the discrete topology, i.e., every point is an open set.  Since $X_\infty$ is countable, such a topology is Polish.
\\We need to show that $U_{k,n_j}\mLGHtop U_{k,\infty}$, up to a further subsequence. We already know that
 $U_{k,n_j}\LGHtop U_{k,\infty}$, so it is enough to show that, up to a subsequence in $j$, 
 \begin{equation}\label{eq:MLGHU-PreComp3}
\lim_{m\to \infty} \lim_{j\to \infty}(g_{k,j,m})_\sharp \mm_{S_{k,j,m}}= \mm_\infty\llcorner U_{k,\infty} \quad \text{weakly as measures}, 
 \end{equation}
 where $\mm_\infty$ is a suitable measure on $X_\infty$, to be constructed.
\smallskip

\textbf{Step 1}. Case $k=1$.
\\ Due to the finiteness of the set, a measure on $V(S_{1,\infty,m})$ can be identified with a non-negative  function defined on it (giving the weights of the Dirac masses on the vertices). With such an identification,  assumption (iv) guarantees that $(g_{1,j,m})_\sharp \mm_{S_{1,j,m}}$ defines a non-negative function on $V(S_{1,\infty,m})$ bounded by  $C_1<\infty$, uniformly in $j\in \N$; actually, using assumption (iv) together with the fact that every $1/m$-net $S_{1,j,m}$ is a covering of $U_{1,n_j}\subset X_{n_j}$, we get that
\begin{equation}\label{eq:TotalMassC1}
\frac{1}{C_1}\leq \sum_{x\in V(S_{1,\infty,m})} (g_{1,j,m})_\sharp \mm_{S_{1,j,m}}(x) \leq C_1, \quad \text{for all }j,m\in \N.
\end{equation}
Then, the Bolzano-Weierstrass theorem in $\R$, coupled with  a diagonal argument,  implies the existence of a real valued function  $\rho_{1,m}:V(S_{1,\infty,m})\to [0,C_1]$, such that 
$$
\lim_{j\to \infty}(g_{1,j,m})_\sharp \mm_{S_{1,j,m}}(x)= \rho_{1,m}(x), \quad \text{for all } x\in V(S_{1,\infty,m}).
$$
Recalling \eqref{eq:Resm-1=m} and that $U_{1,\infty}= \bigcup_{m\in \N} V(S_{1,\infty,m})$, a diagonal argument produces a function 
$$
\rho_1:U_{1,\infty}\to [0, C_1], \qquad \rho_1|_{V(S_{1,\infty,m})}=\rho_{1,m}\quad \text{for all }m\in N,
$$ 
such that
\begin{equation}\label{eq:mm-1jm}
\lim_{m\to \infty} \lim_{j\to \infty}(g_{1,j,m})_\sharp \mm_{S_{1,j,m}}(x)= \rho_1 (x), \quad \text{for all } x\in U_{1,\infty}.
\end{equation}
Define the measure $\mm_{1,\infty}$ on $U_{1,\infty}$, associated to  $\rho_1$: 
$$\mm_{1,\infty}:=\sum_{x\in U_{1,\infty}} \rho_1(x)\; \delta_x. $$
Observe that \eqref{eq:mm-1jm} yields
\begin{equation}
\lim_{m\to \infty} \lim_{j\to \infty}(g_{1,j,m})_\sharp \mm_{S_{1,j,m}}= \mm_{1,\infty}
\quad \text{weakly as measures}. 
\end{equation}
Moreover, \eqref{eq:TotalMassC1} ensures that
$$
\frac{1}{C_1}\leq \mm_{1,\infty}(U_{1,\infty})\leq C_1.
$$
Let us recall that we consider the weak convergence of measures in duality with compactly supported continuous functions; since the $X_\infty$ is endowed with the discrete topology, such functions vanish identically on the complement of a finite set. 
\smallskip

\textbf{Step 2}. Inductive construction and conclusion.
\\ Assume that, for some $k\geq 2$, we constructed a measure $\mm_{k-1,\infty}$ on $U_{k-1,\infty}$ such that
\begin{equation}
\lim_{m\to \infty} \lim_{j\to \infty}(g_{k-1,j,m})_\sharp \mm_{S_{k-1,j,m}}= \mm_{k-1,\infty}
\quad \text{weakly as measures}. 
\end{equation}
Arguing as in step 1 and recalling   \eqref{eq:Resk+1=k}, we can construct a measure $\mm_{k,\infty}$ on $U_{k,\infty}$ such that
$$
\mm_{k,\infty}\llcorner U_{k-1,\infty}=  \mm_{k-1,\infty}
$$
and such that, up to a further subsequence in $j$, it holds: 
$$
 \lim_{m\to \infty} \lim_{j\to \infty}(g_{k,j,m})_\sharp \mm_{S_{k,j,m}} = \mm_{k,\infty}
\quad \text{weakly as measures}.  
$$
Moreover, arguing as step 1, we have that
\begin{equation}\label{eq:mkinftyCk}
\frac{1}{C_k}\leq \mm_{k,\infty}(U_{k,\infty})\leq C_k, \quad \text{for all }k\in \N.
\end{equation}

Since, by construction $X_\infty=\bigcup_{k\in \N} U_{k,\infty}$ and $U_{k,\infty}\subseteq U_{k+1,\infty}$, we can define $\mm_\infty$ on $X_\infty$ by setting 
$$
\mm_\infty\llcorner U_{k,\infty}:= \mm_{k,\infty}.
$$
The desired \eqref{eq:MLGHU-PreComp3} now follows by the above constructions. Moreover, \eqref{eq:mkinftyCk} ensures that $\mm_\infty$ is finite on causal diamonds in $X_\infty$.
\end{proof}

\section{Applications}\label{sec-app}

In this final section we give four applications of the Lorentzian Gromov--Hausdorff convergence developed above. First, we show that Chru\'sciel--Grant approximations \cite{CG:12} of continuous spacetimes are an instance of the Lorentzian Gromov--Hausdorff convergence of their underlying \LpLSn s. Second, we show that timelike sectional curvature bounds are stable under Lorentzian Gromov--Hausdorff convergence. Third, we introduce blow-up tangents and finally, we prove a precise statement about the main conjecture of causal set theory, an approach to Quantum Gravity.

\subsection{Chru\'sciel--Grant approximations viewed as Lorentzian Gromov--Hausdorff convergence}
Given a continuous Lorentzian metric $g$ on a smooth manifold $M$, Chru\'sciel and Grant \cite{CG:12} showed that there are sequences of smooth Lorentzian metrics $(\check{g}_n)_n, (\hat{g}_n)_n$ that converge locally uniformly to $g$ and have nested lightcones, i.e., $\check{g}_n\preceq g \preceq\hat{g}_n$ for all $n\in\N$. Moreover, any continuous spacetime is a \LpLS in the sense of Definition \ref{defi-lpls} (see also \cite[Subsec.\ 5.1]{KS:18} and \cite{Lin:24} for works using the original definition of  \LpLS given in \cite{KS:18}).
\medskip

Using a refined approximation from the outside given in \cite[Appendix A]{McCS:22} (and that the time-separation functions converge) we show
\begin{thm}[Pointed Lorentzian Gromov--Hausdorff convergence for continuous spacetimes]
 Let $(M,g)$ be a continuous, causally plain\footnote{Using the modified time-separation function of \cite{Lin:24} and adapting the proof of \cite[Lem.\ A.1]{McCS:22} one could drop this assumption.} and globally hyperbolic spacetime and fix $o\in M$. Then there is an approximation $\hat{g}_n\to g$ locally uniformly from the outside (i.e., such that $g\preceq \hat{g}_{n+1}\preceq \hat{g}_n$ for all $n\in\N$) and there are coverings $\U$, $\U_n$ of $M$ with respect to $g, \hat{g}_n$ such that $(M,\ell_{\hat{g}_n},o,\U_n) \pLGHtop (M,\ell_g,o,\U)$ strongly.
\end{thm}
\begin{pr}
We use the approximation $\hat{g}_n$ given by \cite[Prop.\ A.1]{McCS:22} which satisfies $\hat{g}_n\to g$ locally uniformly, $g\preceq \hat{g}_{n+1}\preceq \hat{g}_n$, $-g(v,v)\leq -\hat{g}_n(v,v)$ for all $g$-causal $v\in TM$, and $-\hat{g}_{n+1}(v,v)\leq -\hat{g}_n(v,v)$ for all $\hat{g}_{n+1}$-causal $v\in TM$, for all $n\in\N$.

By stability of global hyperbolicity for continuous metrics \cite[Thm.\ 4.5]{Sae:16} there is a smooth metric $\hat{g}$ such that $(M,\hat{g})$ is globally hyperbolic and $g\prec \hat{g}$.
 Fix $o\in M$, then Lemma \ref{lem-cov-st} gives a covering $\U$ with respect to $g$. We claim that $\U = (U_k)_{k\in\N}$ is also a valid cover for $\hat{g}_n$ for each $n\in\N$. Clearly, the first three points of Definition \ref{def-cov-lpls} are satisfied. Each $U_k$ is relatively compact, hence by \cite[Lem.\ 1.4]{Sae:16} there is a $n_k\in\N$ such that for all $n\geq n_k$ we have $\hat{g}_n\prec \hat{g}$ on $\overline{U}_k$. In particular, $\tau_n:=\tau_{\hat{g}_n}$ is bounded (uniformly in $n$) on $U_k$ and $\tau_n\to\tau_g$ uniformly on $\overline{U}_k$ by \cite[Prop.\ A.2]{McCS:22}.
 
 We will cover $\overline{U}_k$ be $g$-chronological diamonds which have small timelike diameter with respect to $\tau_{0}$ and vertices contained in a countable dense set $D$. To see that this is possible let $x\in M$ and $\eps>0$. Then there are $x^\pm\in D$ such that $x\in I_{g_0}(x^-,x^+)$ with $\tau_0(x^-,x^+)\leq \eps$. By strong causality of $(M,g)$ and since $I_{g_0}(x^-,x^+)$ is an open neighborhood of $x$, there are $\tilde x^\pm \in I_{g_0}(x^-,x^+)\cap D$ such that $x\in I_g(\tilde x^-, \tilde x^+)\subseteq I_g(x^-,x^+)$. Then 
\begin{align*}
 \tau_0(\tilde x^-,\tilde x^+) \leq \tau_0(x^-,\tilde x^-) + \tau_0(\tilde x^-,\tilde x^+) + \tau_0(\tilde x^+,x^+) \leq \tau_0(x^-,x^+)\leq \eps\,,
\end{align*}
as claimed.

Consequently, for every $\eps>0$ there is a finite covering $(I_g(p_i,q_i))_{i=1}^{N_\eps}$ of $\overline{U}_k$, where $\tau_0(p_i,q_i)\leq \eps$ for all $i=1,\ldots,N_\eps$. Thus, we obtain
\begin{align*}
 U_k\subseteq \overline{U}_k \subseteq \bigcup_{i=1}^{N_\eps} I_g(p_i,q_i) 
 \subseteq \bigcup_{i=1}^{N_\eps} J_g(p_i,q_i) \subseteq \bigcup_{i=1}^{N_\eps} J_{g_n}(p_i,q_i)\,,
\end{align*}
and $\tau_g(p_i,q_i)\leq \tau_n(p_i,q_i)\leq\tau_0(p_i,q_i)\leq \eps$ for all $n\in\N$. So $(J_g(p_i,q_i))_{i=1}^{N_\eps}$ and $(J_{g_n}(p_i,q_i))_{i=1}^{N_\eps}$ are finite $\eps$-nets of the same cardinality for $U_k$ with respect to $\ell_g$ and $\ell_n$, respectively. As $\tau_n\to\tau_g$ on $U_k$ it remains to show that $\ell_n\to\ell_g$ on the set of vertices. Here the only relevant case is 
if $p,q$ are vertices with $p\not\leq_g q$. If $p\leq_{g_n} q$ for infinitely many $n\in\N$ then by \cite[Thm.\ 1.5]{Sae:16} we would have $p\leq_g q$ --- a contradiction. Hence $p\leq_{g_n} q$ and so the above convergence of $\tau_n$ to $\tau_g$ applies. This immediately implies that $(M,\ell_n,o,\U)$ converges to the collection of all vertices of $\frac{1}{l}$-nets ($l\in\N,l\geq 1$). Continuing as in the proof of Theorem \ref{thm-dis-lim}, we have by Theorem \ref{thm-lim-compl} that  $(M,\ell_n,o,\U)$ also converges  to any forward completion of this set of vertices. Notice that $(M,\ell_g)$ is one of such forward completions, since $(M,g)$ is globally hyperbolic and hence forward complete by Remark \ref{rem-gh-for-compl}. Moreover, by construction the set of vertices is timelike forward dense in $M$, yielding strong convergence.
\end{pr}

\subsection{Stability of timelike sectional curvature bounds}
In this subsection we show stability of lower timelike sectional curvature bounds in the form of the four-point condition. Other (more-or-less) equivalent notions of synthetic timelike sectional curvature bounds will be stable as well, provided that the approximating spaces have curvature bounds in the global sense and that the limit has a continuous time-separation function.
\medskip

Before we introduce the timelike four-point condition \cite{BKR:24}, we introduce the two-dimensional Lorentzian model spaces of constant curvature $K\in\R$ (see e.g.\ \cite{ONe:83}) as
\begin{align*}
 \LL(K):=\begin{cases}
        \tilde S^2_1(\frac{1}{\sqrt{K}}) & K>0\,,\\
\R^2_1 & K=0\,,\\
\tilde H^2_1(\frac{1}{\sqrt{-K}}) & K<0\,,
         \end{cases}
\end{align*}
which have diameter $D_K:=\frac{\pi}{\sqrt{-K}}$ if $K<0$ and $D_K:=\infty$ otherwise. Here $\tilde S^2_1(r)$ is the simply connected covering manifold of the two-dimensional Lorentzian pseudosphere of radius $r>0$ ($r=1$ is de Sitter space), $\R^2_1$ is two-dimensional Minkowski spacetime and $\tilde H^2_1(r)$ is the simply connected covering manifold of two-dimensional Lorentzian pseudohyperbolic space ($r=1$ is anti-de Sitter space).

\begin{defi}[Four-point configurations]
 Let $\Xt$ be a \LpLSn.
 \begin{enumerate}
  \item A \emph{timelike future endpoint-causal four-point configuration} is a quadruple $(y,x,z_1,z_2)\in X^4$ such that $y\ll x\ll z_1\leq z_2$.
  \item Similarly, a \emph{timelike past endpoint-causal four-point configuration} is a quadruple $(z_2,z_1,x,y)\in X^4$ such that $z_2\leq z_1\ll x\ll y$.
  \item Given a timelike future endpoint-causal four-point configuration $(y,x,z_1,z_2)$ and $K\in\R$, a \emph{four-point comparison configuration} in $\LL(K)$ is a quadruple $(\bar y,\bar x,\bar z_1,\bar z_2)\in \LL(K)^4$ such that
  \begin{enumerate}
   \item $\tau(y,x)=\bar\tau(\bar y,\bar x)$\,,
   \item $\tau(y,z_i) = \bar\tau(\bar y,\bar z_i)$ ($i=1,2$)\,,
   \item $\tau(x,z_i) = \bar\tau(\bar x,\bar z_i)$ ($i=1,2$)\,, and
   \item $\bar z_1$, $\bar z_2$ lie on opposite sides of the line through $\bar y$, $\bar x$.
  \end{enumerate}
  \item Similarly, one defines a four-point comparison configuration for a timelike past endpoint-causal four-point configuration.
 \end{enumerate}
\end{defi}

Next we recall the definition of synthetic timelike sectional lower curvature bounds in the form of the four-point condition as defined by Beran--Kunzinger--Rott \cite[Def.\ 4.6]{BKR:24} following a similar construction in the positive signature case, cf.\ e.g.\ \cite{AKP:23}.
\begin{defi}[Four-point condition]
 Let $\Xt$ be a \LpLS and $K\in\R$. A \emph{$\geq K$-comparison neighborhood} is an open set $U\subseteq X$ such that
 \begin{enumerate}
  \item the time-separation function $\tau$ is continuous on the open set $(U\times U)\cap \tau^{-1}([0,D_K))$ and
  \item for every timelike future endpoint-causal four-point configuration $(y,x,z_1,z_2)$ in $U$ with $\tau(y,z_2)<D_K$ and its comparison configuration $(\bar y,\bar x,\bar z_1, \bar z_2)$ in $\LL(K)$ one has 
  \begin{equation}\label{eq-fu-fou-pt}
   \tau(z_1,z_2)\geq \bar\tau(\bar z_1,\bar z_2)\,.
  \end{equation}
  Moreover, for every timelike past endpoint-causal four-point configuration $(z_2,z_1,x,y)$ in $U$ with $\tau(z_2,y)<D_K$ and its comparison configuration $(\bar z_2, \bar z_1,\bar x,\bar y)$ in $\LL(K)$ one has 
  \begin{equation*}
   \tau(z_2,z_1)\geq \bar\tau(\bar z_2,\bar z_1)\,.
  \end{equation*}
 \end{enumerate}
 Finally, we say that $\Xt$ has \emph{timelike sectional curvature bounded below by $K$} if $X$ can be covered by $\geq K$-comparison neighborhoods and we say that $\Xt$ has \emph{global timelike sectional curvature bounded below by $K$} if $X$ is a $\geq K$-comparison neighborhood.
\end{defi}

The four-point condition is equivalent to the other synthetic timelike sectional curvature bounds for large classes of \LpLSn s, and hence to smooth timelike sectional curvature bounds, see \cite[Thm.\ 5.1]{BKR:24} and \cite{BKOR:25}.
\medskip

We are now in a position to establish stability of the global four-point condition under pointed Lorentzian Gromov--Hausdorff convergence. An analogous statement in the setting of bounded Lorentzian metric spaces \cite{MS:24} has been established in the local case in \cite[Thm.\ 6.7]{MS:24} and in the global case in \cite[Thm.\ 4.2]{BHNR:23}.
\begin{thm}[Stability of the four-point condition]\label{thm-sta-4pt-con}
\ \\Let $\XtnoU \pLGHtop \XtoU$. Assume that each $\Xtn$ has global timelike sectional curvature bounded below by $K\in\R$ and that $\tau$ is continuous. Then $\Xt$ has global timelike sectional curvature bounded below by $K$.
\end{thm}
\begin{pr}
 Let $y\ll x\ll z_1\leq z_2$ be a timelike future endpoint-causal four-point configuration with $\tau(y,z_2)< D_K$ and $(\bar y,\bar x,\bar z_1,\bar z_2)$ a comparison configuration in $\LL(K)$. As the covering sets are increasing there is one fixed $U_k\in\U$ such that $y,x,z_1,z_2\in U_k$. First, we consider the case that all four points are vertices of causal diamonds given by the convergence. Hence, there are corresponding points $y^n,x^n, z_1^n, z_2^n\in X_n$ such that their time-separations converge. In particular, for large $n\in\N$ they form a timelike future endpoint-causal four-point configuration $(y^n,x^n, z_1^n, z_2^n)$ with $\tau_n(y^n,z_2^n)<D_K$. Consequently, for these $n$ we have by Equation \eqref{eq-fu-fou-pt} that $\tau_n(z_1^n, z_2^n)\geq \bar\tau(\bar z_1^n, \bar z_2^n)$, where $(\bar y^n, \bar x^n,\bar z_1^n, \bar z_2^n)$ form a comparison four-point configuration in $\LL(K)$. Since the side-lengths of the comparison configurations also converge, we infer that the right-hand-side converges to $\bar\tau(\bar z_1,\bar z_2)$, whereas the left-hand-side converges to $\tau(z_1,z_2)$. Thus, we conclude $\tau(z_1,z_2)\geq \bar\tau(\bar z_1,\bar z_2)$. The analogous argument works for timelike past endpoint-causal four-point configurations. Finally, for general $(y,x,z_1,z_2)$, one approximates them by vertices (they are dense by the definition of the convergence) and uses continuity of $\tau$ to conclude the proof.
\end{pr}

\subsection{Blow-up tangents}
The study of blow-up tangents (or tangent cones) has been very useful to study the infinitesimal structure of non-smooth metric measure spaces with curvature bounds \cite{BBI:01, CC:97, OS:94}. Here we initiate this program by providing such a notion in the Lorentzian setting and establish existence under suitable local assumptions.

\begin{defi}[Scaling a \LpLSn]
 Let $\Xt$ be a \LpLS and $\lambda>0$. Denote by $\lambda\Xt$ the topological space $X$ with time-separation function $\lambda\ell$.
For a pointed \LpLS $\Xto$, we define a \emph{$\lambda$-blow-up around $o$} as
\begin{equation*}
X^\lambda_o:=I(o^\lambda_-,o^\lambda_+)\,,
\end{equation*}
where $o^\lambda_-\ll o \ll o^\lambda_+$, $\tau(o^\lambda_-,o^\lambda_+)< \frac{1}{\lambda}$ and with time-separation $\lambda\ell$, i.e., it is the pointed \LpLS $(X^\lambda_0,\lambda\ell,o)$. Finally, for a covered \LpLS $\XtoU$ we define the \emph{$\lambda$-blow-up around $o$} as $(X^\lambda_o,\lambda\ell,o)$ with covering $\U^\lambda:=(U_k\cap X^\lambda_o)_{k\in\N}$, where $\U=(U_k)_{k\in\N}$.
\end{defi}

We easily see that $\lambda$-blow-ups exist under fairly natural assumptions.
\begin{lem}[Existence of $\lambda$-blow-up]
 Let $\lambda>0$ and let $\Xto$ be a pointed \LpLS where $\tau$ is upper semi-continuous with respect to the Alexandrov topology and finite at $(o,o)$. If, $o\in\overline{I^\pm(o)}$, then there exists a $\lambda$-blow-up $X^\lambda_0$ of $o$.
\end{lem}

At this point we can introduce the timelike blow-up tangent  by taking the limit $\lambda\to\infty$ in $\lambda$-blow-ups around a fixed base point.
\begin{defi}[Timelike blow-up tangent]\label{def-tl-blow-up-tc}
 Let $\XtoU$ be a covered \LpLSn. A pointed Lorentzian Gromov--Hausdorff strong limit (as $\lambda\to\infty$) of $\lambda$-blow-ups $(X^{\lambda},\lambda\ell,o,\U^\lambda)_\lambda$ around $o$  is called \emph{a blow-up tangent} of $\Xto$.
\end{defi}

Blow-up tangents do not exist in general (as in the positive signature case). However we obtain subsequential existence, under a suitable local control on  $\eps$-nets. To this aim, we next introduce a Lorentzian counterpart to the doubling property for a metric space cf.\ \cite[Sec.\ 10.13]{Hei:01}.
\begin{defi}[Doubling property of causal diamonds]
 Let $\Xt$ be a \LpLS and $U\subseteq X$. Then $U$ is said to have the \emph{doubling property} if there exists a constant $N\in\N$, called the \emph{doubling constant}, such that every causal diamond $J(x,y)\subseteq U$ can be covered by $N$ causal diamonds $J(x_i,y_i)$ with $\tau(x_i,y_i)\leq \frac{\tau(x,y)}{2}$ and $x_i,y_i\in U$ ($i=1,\ldots,N$).
\end{defi}

\begin{thm}[Subsequential existence of blow-up tangents]\label{thm:ExistTang}
  Let $\XtoU$ be a covered \LpLSn, where each covering set $U_k\in\U$ has the doubling property with doubling constant $N_k\in\N$. Then any sequence $(X^{\lambda},\lambda\ell,o,\U^\lambda)_\lambda$ of $\lambda$-blow-ups around $o$  admits a subsequence which converges in the strong pointed Lorentzian Gromov--Hausdorff sense (for $\lambda\to\infty)$.
\end{thm}
\begin{pr}
 We verify that the assumptions of the general pre-compactness theorem \ref{thm-pre-comp-I} are satisfied. First, for all $k\in\N$ and $\lambda>0$, we have 
 \begin{equation*}
  \diam^{\lambda\tau}(U_k^\lambda) = \lambda\tau(o^\lambda_-,o^\lambda_+)\leq \frac{\lambda}{\lambda}= 1\,,
 \end{equation*}
 where $U_k^\lambda:=U_k\cap X^\lambda_o$. Let $\eps>0$ and $l\in\N$ such that $\frac{1}{2^l}\leq \eps$. Then by the doubling property we have
 \begin{equation*}
  U^\lambda_k\subseteq J(o^\lambda_-,o^\lambda_+)\subseteq \bigcup_{i=1}^{N_k^l} J(p_i,q_i)\,,
 \end{equation*}
where $p_i,q_i\in U_k$ satisfy $\tau(p_i,q_i)\leq\frac{\tau(o^\lambda_-,o^\lambda_+)}{2^l}\leq\frac{1}{\lambda\, 2^l}$. Thus $(J(p_i,q_i))_{i=1}^{N_k^l}$ is a finite $\eps$-net for $U_k^\lambda$ with respect to $\lambda\ell$, since $\lambda\tau(p_i,q_i)\leq \frac{1}{2^l}\leq \eps$. Finally,  we extend these $\eps$-nets to have $\eps$-nets for $U^\lambda_{k+1}$ with uniformly bounded cardinality. Thus we can apply the first pre-compactness theorem \ref{thm-pre-comp-I} to conclude the proof.
\end{pr}

Theorem \ref{thm:ExistTang} opens up the possibility to study blow-up tangents in relation to timelike sectional and Ricci curvature bounds and relate them to the timelike tangent cones defined as Minkowski cones over the space of directions, see \cite[Subsec.\ 3.1]{BS:23}.

\subsection{The Hauptvermutung of causal set theory}\label{subsec-haupt}
Causal set theory is an approach to Quantum Gravity based on the principle that, at the fundamental level, spacetime is a discrete partially ordered set and the continuum spacetime emerges macroscopically from order and volume, see e.g.\ \cite{Sur:19} for an introduction. To make this program viable the so-called \emph{Hauptvermutung} (main conjecture) has to hold \cite{BLMS:87}. In short, if two spacetimes $(M,g)$ and $(M',g')$ approximate a given causal set $(C,\leq)$ (up to some scale) then $(M,g)$ and $(M',g')$ should be ``close'' (up to some scale). One of the obstacles to tackling this problem was that there was no notion of ``closeness'' of spacetimes and no notion of convergence of causal sets or spacetimes. With the pointed Lorentzian Gromov--Hausdorff convergence developed in this article, we can now prove a precise statement. However, it does not yet provide a complete solution to the conjecture, as it involves converging sequences and no fixed scale. A similar approach has been taken recently by M\"uller in \cite{Mue:25} using his notion of Gromov--Hausdorff convergence of Cauchy slabs \cite{Mue:22}. Before we come to the statement we briefly recall the basics of causal sets viewed as \LpLSn s.

As shown in \cite[Subsec.\ 5.3]{KS:18} causal sets are \LpLSn s. Indeed for a partially ordered set $(C,\leq)$,  one can define $\ell(x,y)$ to be the length of the longest chain connecting $x$ to $y$. This yields the \LpLS $(C,\ell)$ with the discrete topology.
Causal sets are partially ordered sets that are \emph{locally finite}, i.e., each causal diamond contains only finitely many points. Moreover, a causal set $(C,\leq)$ \emph{faithfully embeds} into a spacetime $(M,g)$ if there is a map $\phi\colon C\rightarrow M$ that is $\leq$-preserving, i.e., for all $x,y\in C$ we have $\phi(x)\leq_g \phi(y)$.

\begin{thm}[Smooth globally hyperbolic limits of causal sets are unique]\label{thm-hau-ver}
Let $(C_n,\leq_n)_n$ be a sequence of causal sets such that their induced \LpLSn s $(C_n,\ell_n)$ (pointed) Lorentzian Gromov--Hausdorff converge strongly to two smooth globally hyperbolic spacetimes $(M,g)$ and $(M',g')$ (using some covering of the $C_n$s and $M$, $M'$). Moreover, assume that each $(C_n, \leq_n)$ faithfully embeds into both spacetimes. Then $(M,g)$ and $(M',g')$ are isometric as smooth spacetimes.
\end{thm}
\begin{pr}
We construct covers $\U_n$ of $C_n$ such that each covering set $U_{k,n}\in\U_n$ is finite (by local finiteness). Using the faithful embeddings, we can construct corresponding covers of $(M,g)$ and $(M',g')$. By finiteness of the covering sets in $C_n$, the set of vertices of $\frac{1}{l}$-nets for $U_{k,n}$ always contains $U_{k,n}$ (for $l$ large). Consequently, these vertices are trivially dense and we can apply Proposition \ref{prop-uni-den} to conclude that $(M,\ell_g)$ and $(M',\ell_{g'})$ are isometric as \LpLSn s and hence $M$ and $M'$ are homeomorphic. Thus, they have the same manifold dimension (by invariance of domain) and so we can apply the Hawking--King--McCarthy theorem \cite{HKM:76} (cf.\ \cite[Prop.\ 3.34]{MS:08}, \cite[Thm.\ 4.17]{BEE:96}) to conclude that they are smoothly isometric.   
\end{pr}

\section*{Acknowledgments}
The authors wish to thank the referee for the detailed report and insightful comments, which substantially improved the exposition and clarity of the paper.

AM and CS (in part) were supported by the European Research Council (ERC), under the European’s Union Horizon 2020 research and innovation programme, via the ERC Starting Grant “CURVATURE”, grant agreement No. 802689. CS was also funded in part by the Austrian Science Fund (FWF) [Grant DOI \href{https://doi.org/10.55776/STA32}{10.55776/STA32}]. 
Part of this research was carried out at the Hausdorff Institute of Mathematics in Bonn, during the trimester program  ``Metric Analysis''. The authors wish to express their appreciation to the  institution for the stimulating atmosphere and the excellent working conditions. AM acknowledges support  by the Deutsche Forschungsgemeinschaft (DFG, German Research Foundation) under Germany's Excellence Strategy – EXC-2047/1 – 390685813. 

For open access purposes, the authors have applied a CC BY public copyright license to any author accepted manuscript version arising from this submission.

\appendix
\section{Limits of globally hyperbolic spacetimes}\label{app}
The goal of this appendix is to show that there is a way to obtain that the limit of smooth globally hyperbolic spacetimes is a globally hyperbolic \LpLS (if the limit is non-degenerate). However, this uses a different topological set-up than the main part of the article. Thus we opted to include it in a separate appendix, which might be of independent interest.

The idea is to use a uniform structure, see e.g.\ \cite[Ch.\ 6]{Kel:75}, on the limit space to define a topology and, more importantly, a notion of Cauchy sequences and hence a completion with respect to such a uniform structure. This should be compared with \cite{BMS:24}, performing a similar procedure in the different setting of Lorentzian metric spaces. For the reader's convenience and the sake of completeness, we briefly recall the definition of a uniform structure. Uniform spaces, i.e., spaces with a uniform structure, can be viewed as intermediate structures between topological spaces and metric spaces. They are topological spaces, additionally having an intrinsic notion of (relative) ``closeness'' of points, which allows to define uniformly continuous maps and Cauchy sequences.
\begin{defi}[Uniform space]
A \emph{uniform space} $(X,\mathfrak{U})$ is a set $X$ with a \emph{uniform structure (or uniformity)} $\mathfrak{U}$, where $\mathfrak{U}$ consists of a family of subsets of $X\times X$, satisfying the following properties.
\begin{enumerate}
    \item Each $U\in\mathfrak{U}$ contains the diagonal, i.e., $\Delta:=\{(x,x):x \in X\}\subseteq U$,
    \item The uniformity is closed under inversion, i.e., for all $U\in\mathfrak{U}$, we have $U^{-1}:=\{(x,y)\in X\times X: (y,x)\in U\}\in\mathfrak{U}$.
    \item  Each $U\in\mathfrak{U}$ contains an element of the uniformity of ``half'' the size, i.e., there is $V\in\mathfrak{U}$ with $$V\circ V:=\{(x,z)\in X\times X:\, \exists y \in X \text{ such that }(x,y),(y,z)\in V\}\subseteq U\,.$$
    \item The uniformity is closed under intersection, i.e., if $U,V\in\mathfrak{U}$ then $U\cap V\in\mathfrak{U}$, and
    \item it is closed under supersets, i.e., if $U\in\mathfrak{U}$ and $V\supseteq U$, then $V\in\mathfrak{U}$.
\end{enumerate}
Finally, for $U\in\mathfrak{U}$ and $x\in X$ we define $$U[x]:=\{y\in X: (x,y)\in U\}\,.$$
\end{defi}
A uniform space has the natural topology generated by the neighborhoods $U[x]$ for $x\in X, U\in\mathfrak{U}$. A map $f\colon (X,\mathfrak{U})\rightarrow (Y,\mathfrak{V})$ between uniform spaces is \emph{uniformly continuous} if for all $V\in\mathfrak{V}$ there is $U\in\mathfrak{U}$ such that $f(U):=\{(f(x),f(x')): (x,x')\in U\}\subseteq V$. Moreover, uniform spaces a are the natural setting for Cauchy nets (or sequences) and hence for defining completeness.
\begin{defi}[Cauchy nets]
Let $(X,\mathfrak{U})$ be a uniform space. A net $(x_\alpha)_{\alpha\in A}$ in $X$ is a \emph{Cauchy net} if for all $U\in\mathfrak{U}$ there is $\alpha_0$ such that for all $\alpha,\beta\geq \alpha_0$ we have that $(x_\alpha,x_\beta)\in U$. The uniform space $(X,\mathfrak{U})$ is \emph{complete} if every Cauchy net converges.
\end{defi}
Each uniform space $(X,\mathfrak{U})$ has a \emph{completion}, i.e., a complete uniform space $(\overline{U},\overline{\mathfrak{U}})$ such that $X$ embeds densely and uniformly continuously into $\overline{X}$, cf.\ e.g.\ \cite[Ch.\ 6, Thm.\ 28]{Kel:75}.

Next, we define a uniform structure on a \LpLS $\Xt$, given by a distinguished subset $S\subseteq X$. For $x\in X$ we denote by $\ell_x$ and $\ell^x$ the time-separation from $x$ and to $x$, respectively. To be precise, $\ell_x,\ell^x\colon X\rightarrow\{-\infty\}\cup[0,\infty]$ are defined as $\ell_x(y):=\ell(x,y)$ and $\ell^x(y):=\ell(y,x)$ for $y\in X$.
\begin{lem}[Uniform structure on \LpLSn]\label{lem-uni-str}
 Let $\Xt$ be a \LpLS with $\tau$ finite-valued and $\emptyset\neq S\subseteq X$. Then there is a uniform structure $\mathfrak{U}_S$ on $X$ generated by the basic entourages (subbase)
 \begin{equation*}
  V^s_\delta:=\{ (y,z)\in X\times X: |(\ell_y - \ell_z)(s)| < \delta, |(\ell^y - \ell^z)(s)| < \delta\}\,,
 \end{equation*}
for $s\in S$ and $\delta>0$. The uniform structure $\mathfrak{U}_S$ is pseudo-metrizable if $S$ is countable and metrizable if additionally $S$ has the point distinction property \eqref{eq-pdp}.
\end{lem}
\begin{pr}
The uniform structure consists of all finite intersections of sets $V^s_\delta$ for $s\in S$ and $\delta>0$. If $S$ is countable then the topology induced by $\mathfrak{U}_S$ is pseudo-metrizable by \cite[Thm.\ 6.13]{Kel:75}. If additionally $S$ distinguishes points then so does $\mathfrak{U}_S$: Let $x,y\in X$ with $x\neq y$, then there is an $s\in S$ such that without loss of generality $\ell(x,s) < \ell(y,s)$. Let $0<\delta< \ell(y,s)-\ell(x,s)$ (this works even for $\ell(x,s)=-\infty$), then $V^s_{\frac{\delta}{2}}[x] \cap V^s_{\frac{\delta}{2}}[y]=\emptyset$. This also implies that $\mathfrak{U}_S$ is metrizable.
\end{pr}

\begin{prop}[Topology and convergence]\label{prop-top-conv}
 Let $\XtnoU \pLGHtop \XtoU$. Then there exists a natural pseudo-metrizable uniform structure on $X$.
\end{prop}
\begin{pr}
Let $\U_n = (U_{k,n})_{k\in\N}$ and $\U=(U_{k,\infty})_{k\in\N}$. Then, by the pointed Lorentzian Gromov--Hausdorff convergence, there are countable sets $S_n\subseteq X_n$ and $S\subseteq X$ consisting of the union of $\frac{1}{l}$-nets of $U_{k,n}$ and $U_{k,\infty}$, respectively ($l,k\in\N$) such that the $\ell_n$-distances of the vertices converge to the corresponding $\ell$-distances and vice versa. The uniform structure is then given by $\mathfrak{U}_S$.
\end{pr}

\begin{defi}[Limit topology and uniform structure]\label{def-lim-top}
  Assume that \\ $\XtnoU\pLGHtop \XtoU$. We call the uniform structure and its topology on $X$ constructed from the collection of $\frac{1}{l}$-nets as in Proposition \ref{prop-top-conv} a \emph{limit uniform structure} and a \emph{limit topology} given by the approximating sequence $\big(\XtnoU\big)_{n\in \N}$. 
\end{defi}

We have seen that it is useful to consider the vertices of $\frac{1}{l}$-nets ($l\in\N$) to construct a uniform structure from this collection. At this point one might want to consider the completion with respect to this uniform structure. In particular, when showing that the limit of globally hyperbolic spaces is globally hyperbolic (cf.\ Theorem \ref{thm-lim-gh} below).

\begin{defi}\label{defi-lpls-uni-compl}
 Let $\Xt$ be a \LpLS with finite-valued time-separation $\tau$ and $S\subseteq X$ countable. Then the \emph{completion} of $X$ with respect to $\mathfrak{U}_S$ is denoted by $\overline{X}$.
\end{defi}

For the completion with respect to the uniform structure $\mathfrak{U}_S$ only $\tau$ is essential not $\ell$. Hence when taking the time-separation quotient as in Section \ref{sec-quo} we will do so with respect to $\tau$. This leads to

\begin{defi}[$\tau$-point distinction property]
Let $\Xt$ be a \LpLS and fix a subset $S\subseteq X$. We say that $S$ has the \emph{$\tau$-point distinction property} if
 for all $x,y\in X$ with $x\neq y$ there is a $z\in S$ such that
 \begin{align}
  \label{eq-wpdp} \tag{$\tau\mathrm{PDP}$} &\tau(x,z) \neq \tau(y,z)\,\text{ or } \,\,\tau(z,x) \neq \tau(z,y)\,.
 \end{align}
\end{defi}
\noindent
See \cite[Def.\ 1.1,(iii)]{MS:24} in the setting of (bounded) Lorentzian metric spaces.

\begin{lem}\label{lem-lpls-compl-lpls}
 Let $\Xt$ be a \LpLS with finite-valued time-separation $\tau$. Let $S\subseteq X$ be countable. If $\tau$ is uniformly continuous with respect to $\mathfrak{U}_S$, then the completion $\overline{X}$ of $X$ is a \LpLSn, with closed causal relation $\overline{\leq}$.  Moreover, we can also take the quotient with respect to $\overline{\tau}$ and so without loss of generality $(\overline{X},\overline{\ell})$ satisfies the $\tau$-point distinction property \eqref{eq-wpdp}, hence is causal.
\end{lem}
\begin{pr}
 By uniform continuity we can uniquely extend $\tau$ to the completion $\overline{X}$ and denote it by $\overline{\tau}$. From the reverse triangle inequality on $X$ we immediately get that $\overline{\tau}$ satisfies the reverse triangle inequality for all points $\bar x, \bar y, \bar z\in\overline{X}$ with $\overline{\tau}(\bar x, \bar y)>0$,  $\overline{\tau}(\bar y, \bar z)>0$. To obtain a useful causal relation on $\overline{X}$ we follow \cite[Def.\ 5.1, Thm.\ 5.7]{MS:24}. Define $\bar x\,\overline{\leq}\,\bar y$ if
 \begin{align*}
  \forall \bar z \in \overline{X}:\ \overline{\tau}(\bar z,\bar y)\geq \overline{\tau}(\bar z,\bar x)\ \text{ and }\ \overline{\tau}(\bar x,\bar z)\geq \overline{\tau}(\bar y,\bar z)\,.
 \end{align*}
This is clearly a reflexive and transitive relation on $\overline{X}\times\overline{X}$, and by continuity of $\overline{\tau}$ it is also closed. Also, it contains the chronological relation $\overline{\ll}:=\overline{\tau}^{-1}((0,\infty))$. To see this let $\overline{\tau}(\bar x, \bar y)>0$. Let $\bar z\in\overline{X}$. If $\overline{\tau}(\bar z,\bar x)>0$, then by the reverse triangle inequality for chronologically related points we obtain $\overline{\tau}(\bar z, \bar y)\geq \overline{\tau}(\bar z, \bar x) + \overline{\tau}(\bar x,\bar y) \geq \overline{\tau}(\bar z,\bar x)$. If $\overline{\tau}(\bar z,\bar x)=0$, then $\overline{\tau}(\bar z,\bar y)\geq 0 = \overline{\tau}(\bar z,\bar x)$ anyway. Analogously, one shows the other inequality and so $\bar x\,\overline{\leq}\,\bar y$.

Moreover, $\overline{\tau}$ satisfies the reverse triangle inequality for all points $\bar x\,\overline{\leq}\,\bar y\,\overline{\leq}\,\bar z$. To see this, note that if $\overline{\tau}(\bar x, \bar y)$,  $\overline{\tau}(\bar y, \bar z)$ are both zero or both positive, the required inequality is trivial or already established above, respectively. So without loss of generality let $\overline{\tau}(\bar x, \bar y)=0$,  $\overline{\tau}(\bar y, \bar z)>0$. Then $\overline{\tau}(\bar x, \bar y) +  \overline{\tau}(\bar y, \bar z) = \overline{\tau}(\bar y, \bar z) \leq \overline{\tau}(\bar x, \bar z)$ as $\bar x\,\overline{\leq}\,\bar y$ (by the defining inequality of~$\overline{\leq}$).

Finally, we define $\overline{\ell}\colon \overline{X}\times\overline{X}\rightarrow\{-\infty\} \cup [0,\infty)$ as
\begin{align*}
 \overline{\ell}(\bar x,\bar y):=\begin{cases}
                                  \overline{\tau}(\bar x,\bar y)\qquad &\bar x\,\overline{\leq}\,\bar y\,,\\
                                  -\infty                              & \text{otherwise}\,.
                                 \end{cases}
\end{align*}
Then $\overline{\ell}$ satisfies the reverse triangle inequality for all points in $\overline{X}$. Moreover, by continuity of $\overline{\tau}$, the topology is finer than the chronological one and so $(\overline{X},\overline{\ell})$ is a \LpLSn.

Finally, taking the quotient with respect to $\tau$ as in Section \ref{sec-quo}, where the quotient was taken with respect to $\ell$, we see that all constructions above are well-defined, so we can without loss of generality assume that $(\overline{X},\overline{\ell})$ satisfies the $\tau$-point distinction property \eqref{eq-wpdp}. It remains to show that causality holds. Let $\bar x\,\overline{\leq}\,\bar y\,\overline{\leq}\,\bar x$. Then by definition of the causal relation $\overline{\leq}$ we have for all $\bar z\in\overline{X}$ that $\overline{\tau}(\bar z,\bar y) = \overline{\tau}(\bar z, \bar x)$ and $\overline{\tau}(\bar x, \bar z) = \overline{\tau}(\bar y,\bar z)$, hence by \eqref{eq-wpdp} $\bar x=\bar y$, as required.
\end{pr}

\begin{defi}[Completion of a \LpLSn]\label{defi-compl-lpls}
 Let $\Xt$ be a \LpLS with finite-valued $\tau$. Let $S\subseteq X$ be countable and let $\tau$ be uniformly continuous with respect to $\mathfrak{U}_S$. Then the \emph{completion} of $\Xt$, denoted by $(\overline{X},\overline{\ell})$, is the completion of $X$ given by Lemma \ref{lem-lpls-compl-lpls}.
 \end{defi}
\noindent
Next, we prove that taking the completion does not affect convergence.
\begin{thm}
 Let $\XtnoU \pLGHtop \XtoU$ strongly, where $\tau$ is uniformly continuous. Then $\XtnoU\pLGHtop (\overline{X},\overline{\ell},o,\overline{\U})$, where $(\overline{X},\overline{\ell})$ is the completion of $\Xt$ and $\overline{\U}=(\overline{U}_k)_{k\in\N}$, where $\U = (U_k)_{k\in\N}$.
\end{thm}
\begin{pr}
 First, the covering is $\overline{\U} = (\overline{U}_k)_{k\in\N}$, where $\overline{U}_k$ is the closure of $U_k$ in $\overline{X}$. Fix $k\in\N$ and let $\eps>0$, then there exists a finite $\frac{1}{m}$-net for $U_k$ with $\frac{1}{m}\leq \eps$ with vertices in the countable set used to construct the uniform structure on $X$. Denote this $\frac{1}{m}$-net by $(J(p_i,q_i))_{i=1}^N$. Note that as $\overline{\leq}$ is closed we have $\overline{J(\bar x,\bar y)} = J(\bar x,\bar y)$ for all $\bar x, \bar y\in\overline{X}$. Consequently, $(J_{\overline{\leq}}(p_i,q_i))_{i=1}^N$ is an $\eps$-net for $\overline{U}_k$. As $U_k\subseteq \bigcup_{i=1}^N J(p_i,q_i)$ we get that
 \begin{align*}
  \overline{U}_k \subseteq \overline{\bigcup_{i=1}^N J(p_i,q_i)} = \bigcup_{i=1}^N J_{\overline{\leq}}(p_i,q_i)\,.
 \end{align*}
Moreover, as $\overline{\tau}(p_i,q_i) = \tau(p_i,q_i)$ for all $i=1,\ldots,N$ and the time-separations of the vertices still converge as they are unchanged, the distortion of the correspondences do not change. This immediately also gives the extension property for the correspondences. Finally, timelike forward density of $S$ in $X$ gives the corresponding forwardness in $\overline{X}$ by density of $X$ in $\overline{X}$ and the uniform structure being countably generated, cf.\ the proof of Theorem \ref{thm-lim-compl}.
\end{pr}

Coming back to the main goal of this appendix, we first start at the level of the discrete approximations.
\begin{lem}\label{lem-dis-app-gh}
Let $\XtnoU \pLGHtop \XtoU$. Let $\mathfrak{U}_S$ be a limit uniform structure. Then for all $p,q\in S$ with $p\ll q$, every sequence $(y_k)_k$ in $I(p,q)\cap S$ has a Cauchy subsequence. 
\end{lem}
\begin{pr}
Let $p\ll q$, $p,q \in S$ and let $(y_k)_k$ be a sequence in $I(p,q)$. Let $p=x_i, q= x_j, y_k = x_{m_k}$ for fixed $i,j\in \N$ and all $k\in\N$. For $n\in\N$ let $x^n_i, x^n_j, x^n_{m_k}$ be corresponding points in $X_n$, i.e., such that for $l,r\in\{i,j\}$
\begin{align*}
 \tau(x_r,x_l)&= \lim_{n\to\infty} \tau_n(x^n_{r},x_{l}^n)\,,\\
 \tau(x_r,y_k)&= \lim_{n\to\infty} \tau_n(x^n_{r},x_{m_k}^n)\,,\\
  \tau(y_k,x_r)&= \lim_{n\to\infty} \tau_n(x^n_{m_k},x_{r}^n)\,.
\end{align*}
For fixed $k\in\N$ it holds that $\tau_n(x^n_i,x_{m_k}^n)\to \tau(p,y_k)>0$, and $\tau_n(x_{m_k}^n,x^n_j)\to \tau(y_k,q)>0$ as $n\to\infty$.  Thus, eventually for large $n$, $x^n_{m_k}\in I(x^n_i,x^n_j)$, which is relatively compact by assumption. By a double diagonal procedure and continuity of $\tau_n$ we can choose a subsequence $(x^n_{m_{k_l}})$ with limit $x^n_{m_n}$, in such a way that for all $n,r,l\in\N$
\begin{align*}
|\tau_n(x^n_{m_{k_l}},x^n_r) - \tau_n(x^n_{m_n},x^n_r)|&\leq \frac{1}{l}\,,\\
|\tau_n(x^n_r, x^n_{m_{k_l}}) - \tau_n(x^n_r, x^n_{m_n})|&\leq \frac{1}{l}\,. 
\end{align*}
We claim that $(y_{k_l})_l = (x_{m_{k_l}})_l$ is a Cauchy sequence (in $I(p,q)$). Let $\eps>0$ and $r\in\N$. Choose $N\in\N$ such that $\frac{1}{N}< \frac{\eps}{2}$. Then by the above we have for all $l,l'\geq N$ that
\begin{align*}
 (\tau_{y_{k_l}} - \tau_{y_{k_{l'}}})(x_r) &= |\tau(y_{k_l},x_r) - \tau(y_{k_{l'}},x_r)|\\
 &= |\lim_{n\to\infty} \tau_n(y^n_{k_l},x^n_r) - \lim_{n\to\infty} \tau_n(y^n_{k_{l'}},x^n_r)|\\
 &=  \lim_{n\to\infty} | \tau_n(y^n_{k_l},x^n_r) - \tau_n(y^n_{k_{l'}},x^n_r)|\\ &\leq  \lim_{n\to\infty} | \tau_n(y^n_{k_l},x^n_r) - \tau_n(x^n_{m_n},x^n_r)|\\
 &\ \quad+  \lim_{n\to\infty} | \tau_n(x^n_{m_n},x^n_r) - \tau_n(y^n_{k_{l'}},x^n_r)|< \eps\,.
\end{align*}
Analogously, one shows that $(\tau^{y_{k_l}}-\tau^{y_{k_{l'}}})(x_r)<\eps$. By construction, the points are timelike related and so it suffices to consider the difference of the $\tau$s.
\end{pr}

Finally, we show that the limit of globally hyperbolic spaces is globally hyperbolic if the limit is non-degenerate, i.e., if it  does not contain chronologically isolated points, cf.\  \cite[Subsec.\ 1.4]{KP:67} and \cite[Fig.\ 1, p.\ 6]{MS:24}.
\begin{defi}[Isolated points]
\ 
\begin{enumerate}
  \item  A \LpLS $\Xt$ is said to have a \emph{future/past chronologically isolated point} if there is $x\in X$ with $I^+(x)= \emptyset$ or $I^-(x)=\emptyset$, respectively.
  \item A \LpLS $\Xt$ \emph{without chronologically isolated points} does not have future or past chronologically isolated points.
 \end{enumerate}
\end{defi}

\begin{rem}
Let $\Xt$ be a \LpLSn. If $X$ is \emph{localizable} \cite[Def.\ 3.16]{KS:18} or every point of $X$ is an interior point of a timelike curve, then $X$ has no chronologically isolated points. Moreover, if $X$ satisfies the $\tau$-point distinction property \eqref{eq-wpdp}, then there is at most one future and past chronologically isolated point. Minguzzi and Suhr call such a point \emph{spacelike boundary} and denote it by $i^0$ \cite[Rem.\ 1.2.2]{MS:24}. See also Remark \ref{rem-null-sl-inf} in the main part of the present article.
\end{rem}

We will use the following characterization of global hyperbolicity, which is well-known for  smooth spacetimes  and has been recently studied in relation to non-smooth approaches to Lorentzian geometry in \cite{BM:25}.

\begin{lem}[Characterization of global hyperbolicity]\label{lem-char-gh}
 Let $\Xt$ be a \LpLS without chronologically isolated points and with closed causal relation $\leq$. Then the following are equivalent.
 \begin{enumerate}
  \item The causal diamonds $J(p,q)$ are compact for all $p,q\in X$.
  \item The chronological diamonds $I(p,q)$ are relatively compact for all $p\ll q$.
 \end{enumerate}
\end{lem}
\begin{pr}
 First, if $J(p,q)$ is compact, then $\overline{I(p,q)} \subseteq J(p,q)$ and so $\overline{I(p,q)}$ is compact as a closed subset of a compact set.

 Second, assume that the chronological diamonds are relatively compact and let $p\leq q$ (otherwise there is nothing to prove). Let $p^-\ll p\leq q\ll q^+$, then $J(p,q)$ is a closed subset of the compact set $\overline{I(p^-,q^+)}$.
\end{pr}

We denote by $I^\pm(A):=\bigcup_{a\in A} I^\pm(a)$ for a subset $A\subseteq X$.
\begin{thm}[Global hyperbolicity is preserved in the timelike interior]\label{thm-lim-gh}
Let $\XtnoU \pLGHtop \XtoU$, where $\tau$ is uniformly continuous (with respect to a limit uniform structure). If all $(X_n,\ell_n)$ are globally hyperbolic, then the \emph{timelike interior} $I_{\overline{\ll}}^+(\overline{X})\cap I_{\overline{\ll}}^-(\overline{X})$ of the completion $\overline{X}$ of $X$ is globally hyperbolic.
\end{thm}
\begin{pr}
 First, let $S = (s_m)_m$ be a countable dense set used to construct the limit uniform structure $\mathfrak{U}_S$ on $X$. Let $\bar x, \bar x'\in I_{\overline{\ll}}^+(\overline{X})\cap I_{\overline{\ll}}^-(\overline{X})$ with $\bar x\, \overline{\ll}\, \bar x'$. We claim that $I_{\overline{\ll}}(\bar x,\bar x')$ is relatively compact, which suffices by Lemma \ref{lem-char-gh} as the causal relation $\overline{\leq}$ is closed on $\overline{X}$ by Lemma \ref{lem-lpls-compl-lpls}. Let $(\bar x_k)_k$ be a sequence in $I_{\overline{\ll}}(\bar x,\bar x')$.
 \medskip 

\noindent\underline{Case 1:} $\bar x = x, \bar x' = x',\bar x_k = x_k\in S$ for all $k\in\N$.\\
By Lemma \ref{lem-dis-app-gh}, $(x_k)_k$ has a Cauchy subsequence, hence this converges in $\overline{X}$.
\medskip 

\noindent\underline{Case 2:} $\bar x = x, \bar x' = x'\in S$.\\
In this case we have that for all $k\in\N: \bar x_k = \lim_{l\to \infty} x^k_l$, where $(x^k_l)_l$ is a Cauchy sequence in $X$. Eventually, each sequence $(x^k_l)_l$ lies in $I(x,x')$, hence we are in the setting of Case 1 above.
Thus, by a diagonal argument we obtain a Cauchy subsequence of $(x_k)_k$ and the claim follows.
\medskip 

\noindent\underline{Case 3:} The general case. \\
By density of $S$ in $I_{\overline{\ll}}^+(\overline{X})\cap I_{\overline{\ll}}^-(\overline{X})$ there are $s^-\in I^-_{\overline{\ll}}(\bar x)\cap S$ and $s^+\in I^+_{\overline{\ll}}(\bar s')\cap S$. Then for all $k\in\N$ we have $\bar x_k\in I_{\overline{\ll}}(s^-,s^+)$ and so the claim follows by Case 2 above. 
\medskip

Finally, note that $\overline{X}$ is causal by Lemma \ref{lem-lpls-compl-lpls}, hence globally hyperbolic.
\end{pr}

A direct consequence of the theorem above is that if the completion $\overline{X}$ has no chronologically isolated points, then the timelike interior of $\overline{X}$  is $\overline{X}$ itself. Hence it is globally hyperbolic.
\begin{cor}
Let $\XtnoU \pLGHtop \XtoU$, where each $(X_n,\ell_n)$ is globally hyperbolic and $\tau$ is uniformly continuous (with respect to a limit uniform structure). If the completion $\overline{X}$ of $X$ has no chronologically isolated points, then it is globally hyperbolic.
\end{cor}

\sloppy
\small{
\bibliographystyle{halpha-abbrv}
\phantomsection
\addcontentsline{toc}{section}{References}
\bibliography{Master}
}

\end{document}